\documentclass[preprint,sort&compress,12pt]{elsarticle}
\usepackage{bbm}
\usepackage{amstext}
\usepackage{amsmath}
\usepackage{amssymb}
\usepackage{mathrsfs}
\usepackage{stmaryrd}
\usepackage{bm}
\usepackage{pstricks}
\usepackage{pst-coil}
\usepackage{pst-3d}
\usepackage{amsfonts}
\usepackage{float}
\usepackage{amsthm}

\usepackage{multirow}

\usepackage[titles]{tocloft}

\def\journal{
 }
\def\ftimes{\hfill Completed on September 7, 2024}

\newcommand{\be}{\begin{equation}}
\newcommand{\bea}{\begin{equation}\begin{aligned}}
\newcommand{\beas}{\begin{equation*}\begin{aligned}}
\newcommand{\eeas}{\end{aligned}\end{equation*}}
\newcommand{\eea}{\end{aligned}\end{equation}}
\newcommand{\ee}{\end{equation}}

\begin{document}
\begin{frontmatter}
\title{
On the Sobolev stability threshold for 3D Navier-Stokes equations \\ with rotation near the Couette flow}

\author{Wenting Huang$^1$}

\author{Ying Sun$^2$}

\author{Xiaojing Xu$^{1}$}

\address{1. School of Mathematical Sciences, Beijing Normal University, Beijing, 100875, PR  China}

\address{2. School of Science, Beijing University of Posts and Telecommunications, Beijing, 100876, PR  China}
	
\tnotetext[2]{E-mails:  hwting702@163.com (W. Huang), sunying@bnu.edu.cn(Y. Sun), xjxu@bnu.edu.cn (X. Xu). }


\begin{abstract}
Rotation is a crucial characteristic of fluid flow in the atmosphere and oceans, which is present in nearly all meteorological and geophysical models. The global existence of solutions to the 3D Navier-Stokes equations with large rotation has been established through the dispersion effect resulting from Coriolis force (i.e., rotation). In this paper, we investigate the dynamic stability of periodic, plane Couette flow in the three-dimensional Navier-Stokes equations with rotation at high Reynolds number $\mathbf{Re}$. Our aim is to determine the stability threshold index on $\mathbf{Re}$: the maximum range of perturbations within which the solution  remains stable. Initially, we examine the linear stability effects of a linearized perturbed system. Comparing our results with those obtained by Bedrossian, Germain, and Masmoudi [Ann.  Math. 185(2): 541--608 (2017)], we observe that mixing effects (which correspond to enhanced dissipation and inviscid damping) arise from Couette flow while Coriolis force acts as a restoring force inducing a dispersion mechanism for inertial waves that cancels out lift-up effects occurred at zero frequency velocity. This dispersion mechanism exhibits favorable algebraic decay properties distinct from those observed in classical 3D Navier-Stokes equations. Consequently, we demonstrate that if initial data satisfies $\left\|u_{\mathrm{in}}\right\|_{H^{\sigma}}<\delta \mathbf{Re}^{-1}$   for any $\sigma>\frac{9}{2}$ and some $\delta=\delta(\sigma)>0$ depending only on $\sigma$, then the solution to the 3D Navier-Stokes equations with rotation is  global in time without transitioning away from Couette flow. In this sense, Coriolis force contributes as a factor enhancing fluid stability by improving its threshold from $\frac{3}{2}$ to 1.
\end{abstract}

\end{frontmatter}

\newtheorem{thm}{Theorem}[section]
\newtheorem{lem}{Lemma}[section]
\newtheorem{pro}{Proposition}[section]
\newtheorem{concl}{Conclusion}[section]
\newtheorem{cor}{Corollary}[section]
\newproof{pf}{Proof}
\newdefinition{rem}{Remark}[section]
\newtheorem{definition}{Definition}[section]

\tableofcontents

\section{Introduction}\label{introud}
\numberwithin{equation}{section}

\textit{1.1. Presentation of the problem.}
The stability and instability of laminar shear flows at high Reynolds numbers have been a significant and active area of research in applied fluid mechanics since the early experiments conducted by Reynolds. In three-dimensional (3D) hydrodynamics, one of the most prevalent phenomena is {subcritical transition}: when a laminar flow becomes unstable and transitions to turbulence in experiments or computer simulations at sufficiently high Reynolds numbers, despite potentially being spectrally stable. In recent years, numerous mathematicians have also focused their efforts on researching mathematical theories related to this problem and have made substantial contributions. In 1887, Kelvin  \cite{Kelvin1887} proposed that the flow may exhibit stability; however, the stability threshold decreases as $\nu \rightarrow 0$, leading to transition at a finite Reynolds number. Subsequently, an intriguing and crucial question was first posed by Trefethen et al. \cite{TTRD1993} is that:
\begin{itemize}
	\item  Given a norm $\|\cdot\|_{X}$, find a $\gamma=\gamma(X)$ so that
	\begin{align*}
	\left\|  u_{\mathrm{in}}  \right\|_{X}& \leqslant \; \mathbf{Re}^{-\gamma } \,\, \Rightarrow  \text{ stability}, \nonumber \\
	\left\|  u_{\mathrm{in}}  \right\|_{X}&  \gg  \;\mathbf{Re}^{-\gamma } \,\, \Rightarrow  \text{ instability},
	\end{align*}
	where the exponent $\gamma=\gamma (X)>0$  is called the \textit{stability threshold}. And it is also known as the \textit{transition threshold} in almost all applied literature. Obviously, the \textit{transition threshold} problem is more stringent and complicated than the nonlinear stability problem.
\end{itemize}
Later on, numerous studies have been conducted to estimate $\gamma$ through physical experiments and numerical simulations \cite{Lundbladh1994, MR1886008, MR2736445, MR1631950, MR3185102}. Recently, rigorous mathematical results regarding the stability threshold problem for the Couette flow in classical Navier-Stokes equations have emerged \cite{MR3612004, MR4126259, BGM2022, CWZ2024, MR4373161}, along with relevant references. 

Rotation is a fundamental characteristic of fluid flow in the atmosphere and oceans, which is present in almost all meteorological and geophysical models. One of the major challenges lies in determining whether there exists a global smooth solution to the 3D Navier-Stokes equations with general smooth initial data. However, this problem has already been positively resolved for 3D Navier-Stokes equations with large rotation (also see \cite{MR1400515}). Evidently, the effect of rotation supports higher regularity in solutions. This paper focuses on investigating stable dynamics within the following 3D incompressible Navier-Stokes equations with rotation, also known as Navier-Stokes-Coriolis equations
\begin{align}\label{1.1}
 (NSC) \quad \quad  \begin{cases}
\partial_{t} V-\nu \Delta V+V \cdot \nabla V+\beta e_{3} \times V+\nabla P=0, \\
\nabla \cdot V=0, \\
V(t=0)=V_{\mathrm{in}},
\end{cases} 
\end{align}
where $V=\left(V^1, V^2, V^3\right)^{\text{T}} \left(t, x, y, z \right) \in \mathbb{R}^3$ and $P=P(t, x, y, z) \in \mathbb{R}$ are the velocity field and  the pressure of the fluid, respectively. The term $e_3=(0, 0, 1)^{\text{T}}$ is the unit vector in the vertical direction. The term $\beta e_{3} \times V$ denotes the Coriolis force with the Coriolis parameter $\beta \in \mathbb{R}\backslash\{0\}$. The positive constant $\nu$ is the kinetic viscosity which is the inverse of Reynolds number, i.e., $\nu=\mathbf{Re}^{-1}$. Note that when $\beta=0$, \eqref{1.1} corresponds to the classical Navier-Stokes equations.

Obviously, the Couette flow at high Reynolds number regime, which may be not only the simplest  nontrivial stationary solution to the 3D  Navier-Stokes equations, but also a steady solution to  (NSC). In this paper,  we  study the dynamic stability behavior of solutions to  (NSC) around the nontrivial stationary solution (which is known as Couette flow)
\begin{equation}\label{1.2}
V^s=(y, 0, 0)^{\mathrm{T}},  \quad \partial_{y} P^s=-\beta y
\end{equation}
with $\beta \in \mathbb{R}\backslash\{0\}$.
To avoid the boundary effect, we consider the problem in a simple domain $(t, x, y, z) \in \mathbb{R}_{+} \times \mathbb{T} \times \mathbb{R} \times \mathbb{T}$ (the torus $\mathbb{T}$ is the periodized interval $[0, 1]$). We denote $u$ for the perturbation of the Couette flow, that is, we set  $V=( y, 0,0)^{\mathrm{T}}+u$, then it satisfies
\begin{equation}\label{1.3}
\begin{cases}
\partial_{t} u+ y \partial_{x} u+\begin{pmatrix}
(1-\beta)u^2 \\
\beta u^1 \\
0
\end{pmatrix}-\nu \Delta u+\nabla p^{L} =-u \cdot \nabla u-\nabla p^{NL}, \\
\nabla \cdot u=0,\\
u(t=0)=u_{\mathrm{in}},
\end{cases}
\end{equation}
where the initial data $u_{\mathrm{in}}=V_{\mathrm{in}}-(y, 0,0)^{\mathrm{T}}$, and the pressure $p^{NL}$ and $p^{L}$ are determined by 
\begin{align}
\Delta p^{NL}&=-\partial_{i}u^j \partial_{j}u^i, \label{1.4} \\
\Delta p^{L}&=(\beta-2)\partial_{x} u^2-\beta\partial_{y} u^1. \label{1.5}
\end{align}
The term $p^{NL}$  is  nonlinear contribution to the pressure from the convective term and   $p^{L}$  is linear contribution to the pressure from the interaction between the disturbance near the Couette flow and  the rotation of fluid.

The motion of fluids strongly affected by rotation is fundamentally distinct from that of non-rotating fluids. In order to elucidate or predict phenomena in atmospheres, oceans, planets, and astrophysics, scientists are increasingly focusing on the study of rotating fluids. Simultaneously, research on rotating fluid also encompasses fundamental problems and can be extended from the stability analysis of centrifuges to that of rotating spacecraft. The rotational effect has recently been recognized as crucial in prolonging solution lifespans. Thus, it is not surprising that numerous articles have addressed theoretical or experimental analyses of this system.

Under large-scale rotation, both the atmosphere and ocean can be treated as thin nearly two-dimensional spherical shell fluid layers or unbounded fluid layers. Rotation typically exhibits anisotropic relations and dispersion characteristics leading to amplitude decay. Due to these effects, extensive investigations have been conducted regarding the global existence of solutions for (NSC) when rotation is significant and initial data represents a large perturbation around a zero stationary state (i.e., $V^s=(0, 0, 0)^{\mathrm{T}}$). Relevant studies include those by Babin et al. \cite{MR1736966, MR1480996, MR1855663}, Chemin et al. \cite{MR2228849}, Iwabuchi and Takada \cite{MR3468733, MR3096523, MR3229792}, Koh et al. \cite{MR3229600}, among others.

Even in cases where there is no viscous term $\nu\Delta V$, the rotational effect with dispersion mechanism still significantly contributes to extending the existence time of solutions for Euler equations as demonstrated in previous works such as \cite{KLT2014, T2016, JW2020, GPW2023, GHPW2023} and references therein.

A natural and interesting problem is to study the stabilization mechanism of (NSC).    We know that   the rotation of the fluid not only makes the global well-posedness of (NSC),  but also brings a strong coupling relationship between the equations.
In particular, \eqref{1.3} exhibits both dispersion which is generated by the rotation and mixing effects, which are the main contributions of the fluid stability. A crucial challenge hereby lies in the fact that although one may expect that the stability arose  from a combination of such mechanisms, their distinct nature makes itself difficult to treat in a combined fashion and need to  the quantitative extra analysis.

\textit{1.2. Background.}
Let us first review the research on the stability results of the Navier-Stokes equations. 
Of course,  there are many results for the stability of 2D Navier-Stokes or other fluid systems near Couette flow,  see  \cite{MR3867637, MR3448924, DWXZ2021, MR4176913, MR4451473, MR3974608, BHIW1-2024, BHIW2-2024,MR4121130,CWZ2023} and so on, but few in the 3D case.  In fact, the vortex structure of the 2D Navier-Stokes equations has a good effect on the dynamics stability. In addition,  as early as 1975, Ellingsen and Palm \cite{Ellingsen1975} investigated a fundamental linear mechanism, commonly referred to as the lift-up effect, which is the main effect causing instability in 3D shear flow. Note that this effect is not present in 2D Navier-Stokes equations due to the vanishing of the term $u_{0}^{2}$ in $u_{0}^{1}$ equation by incompressible condition.

To the best of our knowledge, there are currently no stability results available for the system (NSC) near the Couette flow. However, significant progress has been made on the stability threshold problem for 3D classical Navier-Stokes equations near the Couette flow in a series of works by Bedrossian, Germain and Masmoudi \cite{MR3612004, MR4126259, BGM2022}. Specifically, in their work \cite{MR3612004}, they studied the stability threshold for 3D Navier-Stokes equations and demonstrated that if the initial perturbation satisfies $\left\| u_0 \right\|_{H^{\sigma}} \leqslant \delta \nu^{\frac{3}{2}}$ for any $\sigma>\frac{9}{2}$, then the solution is globally well-posed and remains close to the Couette flow within $O(\nu^{\frac{1}{2}})$ in $L^2$ norm for all time. Later on, Wei and Zhang \cite{MR4373161} showed that the transition threshold $\gamma=1$ of 3D Navier-Stokes equations in Sobolev space $H^2$ is optimal. They further provided a mathematically rigorous proof that the regularity of initial perturbations does not affect this transition threshold. The stability threshold has also been investigated in other fluid dynamics equations such as MHD equation \cite{MR4115008}, Boussinesq equation \cite{Zelati2023, ZZW2024}, Patlak-Keller-Segel equation \cite{CWW2024}. For 3D classical Navier-Stokes equations in finite channels with different geometries like Couette flow or plane Poiseuille flow, we refer to papers such as \cite {CWZ2024} and\cite {CDLZ2024} respectively. 

Due to applications involving atmospheric or oceanic flows among other geophysically complex fluids controlled by (NSC), see reference \cite{GSR2007} for more information. When considering wind or ocean currents specifically, it is necessary to account for Coriolis force due to rotation of reference frame. This provides a natural physical basis for studying nonlinear dynamic stability of (NSC) near Couette flow.

\textit{1.3. Structural differences between (NSC) and Navier-Stokes equations.} 
To begin with, we study the linear effect of (NSC) near the Couette flow. 
We neglect the nonlinear contributions in \eqref{1.3},  then the linearized perturbation equation satisfies
\begin{equation}\label{HSX1.6}
\partial_{t} u+ y \partial_{x} u-\nu \Delta u+\begin{pmatrix}
(1-\beta)+(\beta-2)\partial_{x} \Delta^{-1}\partial_{x} \\
(\beta-2)\partial_{y} \Delta^{-1}\partial_{x} \\
(\beta-2)\partial_{z} \Delta^{-1}\partial_{x}
\end{pmatrix}u^2-\beta\begin{pmatrix}
 \partial_{x} \Delta^{-1} \partial_{y} \\
-1+ \partial_{y} \Delta^{-1} \partial_{y}\\
\partial_{z} \Delta^{-1} \partial_{y}
\end{pmatrix}u^1=0.
\end{equation}
\eqref{HSX1.6} is structurally distinct from the classical Navier-Stokes equations.  Indeed, the linearized perturbation of the 3D Navier-Stokes equations near the Couette flow is written as
\begin{align}\label{1..7}
&\partial_{t} u+ y \partial_{x} u -\nu \Delta u+\begin{pmatrix}
1-2 \partial_{x} \Delta^{-1} \partial_{x}  \\
-2 \partial_{y} \Delta^{-1} \partial_{x}  \\
-2 \partial_{z} \Delta^{-1} \partial_{x} 
\end{pmatrix}u^{2}=0,
\end{align} 
where the  $u^{2}$ equation is decoupled from the velocity fields $u^{1}$ and $u^{3}$. This allows us to consider the $u^2$ equation separately when we  study the stability of \eqref{1..7} by energy method. 

The disparities in linear effects between \eqref{HSX1.6} and \eqref{1..7} are summarized in Table \ref{tab1}. The stability issue of various dynamic equations near Couette flow is closely associated with four linear phenomena, namely lift-up, dispersion, inviscid damping, and enhanced dissipation. Among these phenomena, the impacts of dispersion, inviscid damping, and enhanced dissipation exhibit favorable characteristics that contribute to the decay property of solutions.
\begin{table}[H]\label{tab1}
	\centering
	\caption{Four linear effects of Navier-Stokes equations and (NSC)}
	\begin{tabular}{c|c|c}\hline
		{Linear effects} & Navier-Stokes & (NSC)  \\  \hline
		Lift-up  & $u_0^1$  & None \\ 
		Dispersion  & None  & $u_0^1$, $u_0^2$, $u_0^3$ \\ 
		Inviscid damping  & $u_{\neq}^2$  & $u_{\neq}^2$ \\ 
		Enhanced dissipation  & $u_{\neq}^1$, $u_{\neq}^2$, $u_{\neq}^3$  & $u_{\neq}^1$, $u_{\neq}^2$, $u_{\neq}^3$ \\ \hline
	\end{tabular}
\end{table}

It is worth noting that the lift-up effect of 3D classical Navier-Stokes equations arises from the linear term $u_0^2$ in the equation for $u_0^1$. However, our current focus lies on investigating the impact of rotation, which corresponds to the Coriolis force $\beta e_3 \times V$ in \eqref{1.1} on the stability of the Navier-Stokes equations near Couette flow. The most significant observation is the emergence of a novel linear term, $\nabla (-\Delta)^{-1} \partial_{y} u^1$, within the linear perturbation equations \eqref{HSX1.6}. This leads to a strongly coupled structure between $u^1$ and $u^2$. Unlike  \eqref{1..7}, rotation converts the lift-up effect exerted on simple-zero frequency velocity along $u^1$ direction, occurring over a short time scale of order $\nu^{-1}$, into a dispersive mechanism characterized by oscillations, as observed in subsection \ref{section3.3}. This represents another distinction between (NSC) and classical Navier-Stokes equations regarding their linear effects. 

It should be emphasized that rotation acts here to counterbalance the lift-up effect and induce dispersion, as discussed in subsections  \ref{section3.2} and \ref{section3.3} respectively. Simultaneously, rotation does not affect enhanced dissipation but exhibits slight differences in inviscid damping originally caused by Couette flow.

\textit{1.4. Re-formulation of the equations.}
To better understand  the  perturbation equations \eqref{1.3}, we first introduce two quantities $q^{2}\triangleq\Delta u^{2}$ and   $w^{2} \triangleq \partial_{z} u^{1}-\partial_{x} u^{3}$ which is observed in many physical literatures such as \cite{Kelvin1887, MR1886008, CWW2024, MR1801992, MR4126259}. Then these  unknowns satisfy the system:
\begin{align}\label{2.4}
\begin{cases}
\partial_{t} q^{2}+ y \partial_{x} q^{2}-\nu \Delta q^{2}+\beta \partial_{z} w^{2}\\
\quad   =-u \cdot \nabla q^{2}-q \cdot \nabla u^{2}-2 \partial_{i} u^{j} \partial_{i j} u^{2}+\partial_{y}\left(\partial_{i} u^{j} \partial_{j} u^{i}\right), \\
\partial_{t} w^2+  y \partial_{x} w^{2}-\nu \Delta w^2+\left( 1-\beta \right)\partial_{z} \Delta^{-1} q^2  \\
\quad = -\partial_{z} \left( u \cdot \nabla u^1 \right) +\partial_{x} \left( u \cdot \nabla u^3 \right) ,\\
q^2(t=0)=q_{\text {in }}^2, \quad w^2(t=0)=w_{\text {in }}^2.
\end{cases}
\end{align}

In the following, we will introduce a coordinate transform  described in detail in \cite{MR4126259, MR3612004, Zelati2023},  which  is a central tool for our analysis. Due to the characteristics of the transport structure in (\ref{2.4}), we define a linear change of variable 
\begin{equation}\label{2.5}
X=x- t y, \quad \quad Y=y, \quad \quad Z=z.
\end{equation}
We will use the convention of  capital letters when considering any function $f$ in the moving frame (\ref{2.5}), hence defining $F$ by $F(t,X,Y,Z)=f(t,x,y,z)$. In the new coordinates, the corresponding differential operators are changed as follows 
\begin{equation}\label{2.6}
\nabla f(t, x, y, z) =\left(\begin{array}{c}
\partial_{x} f \\
\partial_{y} f \\
\partial_{z} f
\end{array}\right)=\left(\begin{array}{c}
\partial_{X} F \\
\left(\partial_{Y}- t \partial_{X}\right) F \\
\partial_{Z} F
\end{array}\right) \\
=\left(\begin{array}{c}
\partial_{X} F \\
\partial_{Y}^{L} F \\
\partial_{Z} F
\end{array}\right)=\nabla_{L} F(t, X, Y, Z) .
\end{equation}
Using the notation
\begin{equation*}
\Delta_{L}=\nabla_{L} \cdot \nabla_{L}= \partial_{X}^2+(\partial_{Y}^{L})^2+\partial_{Z}^2.
\end{equation*}
The Laplacian transforms as
\begin{equation}\label{2.7}
\Delta f=\Delta_{L}F=\left(\partial_{X}^2+(\partial_{Y}^{L})^2+\partial_{Z}^2 \right) F.
\end{equation}
Combining the definition of the new coordinates as in (\ref{2.5}) and denoting $U^i(t, X, Y, Z)=u^i(t, x, y, z)$ $(i=1, 2, 3)$, $  Q^2(t, X, Y, Z)=q^2(t, x, y, z)$ and $W^2(t, X, Y, Z)=w^2(t, x, y, z)$.  \eqref{2.4} then can be  rewritten as follows
\begin{align}\label{2.10}
\begin{cases}
\partial_{t} Q^{2}-\nu \Delta_{L} Q^{2}+\beta \partial_{Z} W^2\\
\quad  =\partial_{Y}^{L}\left(\partial_{i}^{L} U^{j} \partial_{j}^{L} U^{i}\right)-Q \cdot \nabla_{L} U^{2}-U \cdot \nabla_{L} Q^{2}-2 \partial_{i}^{L} U^{j} \partial_{i j}^{L} U^{2}, \\
\partial_{t} W^{2}-\nu \Delta_{L} W^{2}+\left( 1-\beta \right) \partial_{Z} \Delta_{L}^{-1} Q^2 \\
\quad =-\partial_{Z} \left( U \cdot \nabla_{L} U^1\right)+\partial_{X} \left( U \cdot \nabla_{L} U^3\right),\\
Q^{2}(t=0)=Q_{\text{in}}^{2},\quad  \quad W^2(t=0)=W_{\text {in }}^2.
\end{cases}
\end{align}
Although most work is done directly on the system (\ref{2.10}), for certain steps it will be necessary to use the momentum form of the equations
\begin{align}\label{2.11}
\begin{cases}
\partial_{t} U-\nu \Delta_{L} U+\begin{pmatrix}
\left(1-\beta \right) U^2 \\
\beta U^1 \\
0\end{pmatrix}+\left( \beta-2 \right) \nabla_{L} \Delta_{L}^{-1} \partial_{X} U^2-\beta \nabla_{L} \Delta_{L}^{-1} \partial_{Y}^{L} U^1 \\
\quad=-U \cdot \nabla_{L} U+\nabla_{L} \Delta_{L}^{-1}(\partial_{i}^{L} U_j \partial_{j}^{L} U_i),\\
U(t=0)=U_{\text{in}}.
\end{cases}
\end{align}
It is worth mentioning here that under the coordinate transformation \eqref{2.5}, the initial value conditions in \eqref{2.10}  and  \eqref{2.11} are essentially the same as the initial values in  \eqref{2.4} and \eqref{1.3}, respectively. That is, $U_{\text{in}}=u_{\text{in}},$ $ Q_{\text{in}}^{2}=q_{\text{in}}^{2}$ and $W_{\text{in}}^{2}=w_{\text{in}}^{2}$. In the below section, we will not make a special distinction to explain.

\textit{1.5. Summary of main results.} To state our result, we define the projections on the zero frequency and the nonzero frequencies in $x$ of a function $f$  as follows:
\begin{equation}\label{HSX333}
P_{k=0}f=f_{0}=\int_{\mathbb{T}} f(x,y,z)dx, \quad	P_{\neq}f=f_{\neq}=f-f_{0}. 
\end{equation}
Similar to the definition of the zero frequency in $x$, it will also be convenient to projection onto the zero frequency in  $z$. For this, we use the following notation
$$P_{l= 0}f= \overline{f}=\int_{\mathbb{T}}f(x,y,z)dz, \quad \quad \widetilde{f}=f-\overline{f}.$$ 
Furthermore, the projection onto the zero frequency in $x$ and $z$ of a function $f(x,y,z)$ is denoted by
\begin{equation*}
	\overline{f}_0=f_{0, 0}=\int_{\mathbb{T}} \int_{\mathbb{T}} f(x,y,z)dx dz, \quad \quad \widetilde{f}_0=f_0-\overline{f}_0,
\end{equation*}
where $\overline{f}_{0}$ represents the double zero frequency and $\widetilde{f}_{0}$ stands for the simple zero frequency of $f$.  

Our first main result gives the linear stability of (NSC) near the Couette flow as follows.  
\begin{thm}[\textbf{The linear stability}]\label{1.1.}
	Assume   that $\nu>0$ and  $\beta \in \mathbb{R}$ with $|\beta| \geqslant 2$. Let the initial data $u_{\mathrm{in}}$ be a divergence-free, smooth vector field, then the solution $u=u_{\neq}+u_0$ to the linearized  equation \eqref{HSX1.6}  satisfies the following  linear stability estimates:
	
	$(1)$  The enhanced dissipation  and inviscid damping of the non-zero   frequency $U_{\neq}$:
	\begin{align}
	\left\| (U^1_{\neq}, U^3_{\neq}) (t) \right\|_{H^{s}} \lesssim & \, e^{-\frac{1}{24} \nu t^3} \left(\left\|  U^2_{\mathrm{in}}\right\|_{H^{s+2}} + \left\|  W^2_{\mathrm{in}}\right\|_{H^{s+1}} \right) , \label{1.12} \\
	\left\|  U^2_{\neq} (t) \right\|_{H^{s}} \lesssim & {\left\langle {t} \right\rangle}^{-1} e^{-\frac{1}{24} \nu t^3} \left(\left\|  U^2_{\mathrm{in}}\right\|_{H^{s+3}} + \left\|  W^2_{\mathrm{in}}\right\|_{H^{s+2}} \right)  \label{1.11}
	\end{align}
	for any $s\geqslant 0$. 
	
	$(2)$ Cancellation of lift-up effect on the zero  frequency $u_{0}=\overline{u}_0+\widetilde{u}_0$ by rotation. For $k=0$,  if $l \neq 0$, considering the differential operators 
	\begin{align}\label{3.1.4}
	\mathcal{L}_{+} \triangleq \nu \Delta_{y, z}+i \sqrt{\beta (\beta-1)} |\partial_{z}||\nabla_{y, z}|^{-1}, \quad \mathcal{L}_{-} \triangleq \nu \Delta_{y, z}-i \sqrt{\beta (\beta-1)} |\partial_{z}||\nabla_{y, z}|^{-1},
	\end{align}
	then the simple zero frequency $\widetilde{u}_{0}$  can be written as
	\begin{equation}\label{1.13}
	\begin{cases}
	\widetilde{u}_{0}^{1}(t, y, z)=\frac{1}{2}\left(e^{\mathcal{L}_{+} t}+e^{\mathcal{L}_{-} t}\right) \widetilde{u}_{0 \mathrm{in}}^{1}+\frac{i}{2} \sqrt{\frac{\beta-1}{\beta}} |\nabla_{y, z}| |\partial_{z}|^{-1}\left(e^{\mathcal{L}_{-} t} -e^{\mathcal{L}_{+} t} \right) \widetilde{u}_{0 \mathrm{in}}^{2}, \\
	\widetilde{u}_{0}^{2}(t, y, z)=\frac{1}{2}\left(e^{\mathcal{L}_{+} t}+e^{\mathcal{L}_{-} t}\right)  \widetilde{u}_{0 \mathrm{in}}^{2}+\frac{i}{2}\sqrt{\frac{\beta}{\beta-1}} |\nabla_{y, z}|^{-1} |\partial_{z}|  \left(e^{\mathcal{L}_{+} t}-e^{\mathcal{L}_{-} t}\right) \widetilde{u}_{0 \mathrm{in}}^{1}, \\
	\widetilde{u}_{0}^{3}(t, y, z)=-\partial_{z}^{-1} \partial_{y} \widetilde{u}_{0}^{2}(t, y, z).
	\end{cases}
	\end{equation}
	If  $l=0$, then \eqref{HSX1.6} reduces to two decoupled heat equations and the double zero frequency $\overline{u}_0$ satisfies
	\begin{equation}\label{3.1.5}
	\overline{u}_0^1(t, y)=e^{\nu \partial_{yy} t} \overline{u}_{0 \mathrm{in}}^1, \quad \quad
	\overline{u}_0^2(t, y)=0,  \quad \quad
	\overline{u}_0^3(t, y)=e^{\nu \partial_{yy} t} \overline{u}_{0 \mathrm{in}}^3.
	\end{equation}
	
   $(3)$ The dispersive estimates on the simple zero frequency $\widetilde{u}_0$  satisfy
	\begin{align}
		\left\| \widetilde{u}_{0}^{1}  \right\|_{L^{\infty}(\mathbb{R} \times \mathbb{T})} \lesssim & e^{-\nu t} \left(  \sqrt{\beta (\beta-1)} \, t\right)^{-\frac{1}{3}} \left( \left\|  \widetilde{u}_{0 \mathrm{in}}^{1}  \right\|_{W^{4, 1}(\mathbb{R} \times \mathbb{T})}+  \left\|  \widetilde{u}_{0 \mathrm{in}}^{2}  \right\|_{W^{5, 1}(\mathbb{R} \times \mathbb{T})}   \right), \label{3.6}\\
		\left\| \widetilde{u}_{0}^{2}  \right\|_{L^{\infty}(\mathbb{R} \times \mathbb{T})} \lesssim & e^{-\nu t} \left(  \sqrt{\beta (\beta-1)} \, t\right)^{-\frac{1}{3}}  \left\| \left( \widetilde{u}_{0 \mathrm{in}}^{1}, \widetilde{u}_{0 \mathrm{in}}^{2}  \right) \right\|_{W^{4, 1}(\mathbb{R} \times \mathbb{T})}, \label{3.6.1}\\
		\left\| \widetilde{u}_{0}^{3}  \right\|_{L^{\infty}(\mathbb{R} \times \mathbb{T})} \lesssim  &e^{-\nu t} \left(  \sqrt{\beta (\beta-1)} \, t\right)^{-\frac{1}{3}}  \left\| \left( \widetilde{u}_{0 \mathrm{in}}^{1}, \widetilde{u}_{0 \mathrm{in}}^{2}  \right) \right\|_{W^{5, 1}(\mathbb{R} \times \mathbb{T})}. \label{3.6.2}
	\end{align}
\end{thm}

\begin{rem}
In Theorem \ref{1.1.}, the implication of \eqref{1.11} is that the algebraic decay rate ${\left\langle{t} \right\rangle}^{-1}$ of inviscid damping for $U_{\neq}^2$ is slower than the rate $  {\left\langle {t} \right\rangle}^{-2}$ observed in the classical Navier-Stokes equations. This discrepancy arises from the introduction of the new variable $K^2$ and Fourier multiplier $m$, which are employed to handle the stretching term in \eqref{3.7}$_{2}$, as evident from their uniform boundedness and also supported by \eqref{3..100} in section \ref{section3}.
\end{rem}
\begin{rem}	 
	 The explicit expressions of the zero frequency velocity field are provided by \eqref{1.13} and \eqref{3.1.5}. This observation suggests that the rotational effect hinders the lift-up effect that is a crucial factor in the instability near the Couttee flow.
\end{rem}
\begin{rem}
	In comparison to the classical Navier-Stokes equations, rotation induces a dispersion effect on $\widetilde{u}_0$, leading to additional algebraic decay estimates in time. It should be noted that although we present dispersion estimates for simple zero frequencies, they do not affect our investigation into nonlinear stability as these estimates do not impact interactions between non-zero frequencies.
\end{rem}

Our second main result concerns the nonlinear stability of (NSC) near the Couette flow in Sobolev space as follows.
\begin{thm}[\textbf{The nonlinear stability threshold}]\label{1..1}
	 Let $\nu \in (0,1)$ and $\beta \in \mathbb{R}$ with $|\beta| \geqslant 2$. For all  $\sigma > \frac{9}{2}$, there exists  $\delta=\delta(\sigma)$  such that if the initial data $u_{\mathrm{in}}$  is divergence-free with
	\begin{align}\label{1..11-2}
    \left\|u_{\mathrm{in}}\right\|_{H^{\sigma}}=\varepsilon<\delta \nu, 
	\end{align}
	then there exists a unique global solution $u$ to \eqref{1.3}  and satisfying the following estimates:
	\begin{align}
	&\left\| u_0^{1, 3} \right\|_{L^{\infty} H^{\sigma-1}}+ \nu^{\frac{1}{2}} \left\| \nabla u_0^{1, 3} \right\|_{L^2 H^{\sigma-1}}\lesssim  \varepsilon,   \label{1...9} \\
	&\left\| u_0^{2} \right\|_{L^{\infty} H^{\sigma}}+ \nu^{\frac{1}{2}} \left\| \nabla u_0^{2} \right\|_{L^2 H^{\sigma}}\lesssim  \varepsilon,   \label{1...10}\\
	&\left\| U_{\neq}^{1, 2} \right\|_{L^{\infty} H^{\sigma-2}}+ \nu^{\frac{1}{2}}\left\| \nabla_{L} U_{\neq}^{1, 2} \right\|_{L^2 H^{\sigma-2}}\lesssim  \varepsilon,   \label{1...11} \\
	&\left\| U_{\neq}^3 \right\|_{L^{\infty} H^{\sigma-2}}+ \nu^{\frac{1}{2}}  \left\| \nabla_{L} U_{\neq}^3 \right\|_{L^2 H^{\sigma-2}}\lesssim   \varepsilon.   \label{1...12}
	\end{align}
\end{thm}

\begin{rem}
	Compared with the nonlinear stability threshold of 3D classical Navier-Stokes equations, such as Bedrossian, Germain and Masmoudi \cite{MR3612004} with $\gamma=\frac{3}{2}$,  Wei and Zhang \cite{MR4373161} with $\gamma=1$, we establish a stability threshold of $\gamma=1$ in Theorem \ref{1..1}. This demonstrates that rotation contributes to the stability of Couette flow. However, given the presence of rotation, it is necessary to investigate further mechanisms for reducing the threshold below 1, which will be explored in future work.
\end{rem}

\begin{rem}
  The global existence of solutions to (NSC) with large initial data is guaranteed when the Coriolis parameter $\beta$ reaches a sufficiently large value. However, Theorem \ref{1.1.} and \ref{1..1} only require $|\beta|\geqslant 2$. In fact, this requirement can be relaxed in a quantified manner by considering $\beta$ to be strictly greater than 1 or less than 0. This choice of range for $\beta$ ensures that the zero frequency velocity field does not exhibit any lift-up effect, which is crucial for improving stability thresholds.
\end{rem}

\textit{1.6. Brief comment of key ideas.}
To give the rough ideas and techniques in proving Theorems \ref{1.1.} and \ref{1..1}, we review the dynamics it captures here. We refer to the sections \ref{section3} and \ref{section4} for the details by selecting the unknowns, setting up our main bootstrap argument and more elaborate overview of proof. The proper choice of unknowns also includes translation to a moving frame. For the sake of simplicity, we ignore this more delicate point in this preliminary presentation.

\textit{1.6.1. Linearized dynamics.} Due to the rotational effect and incompressible condition, $u_0^3$ is compelled by $u_0^2$, while $u_0^1$ and $u_0^2$ are coupled in an oscillatory manner. It can be observed that rotation counteracts the classical lift-up instability mechanism present in the Navier-Stokes equations near Couette flow. Additionally, dispersion caused by rotation occurs perpendicular to the rotation vector, resulting in a time decay of associated waves in the $L^\infty$-norm. This phenomenon manifests itself within the simple zero frequency $\widetilde{u}_0$ of the velocity field and plays a crucial role in studying the stability of (NSC). Therefore, it is imperative to accurately analyze the linear system \eqref{1.6}, which encompasses two key features categorized based on their respective linear effects as follows:

$(1)$ (Non-zero frequencies $k \neq 0$). 
The effect of friction, caused by the change of $\Delta_{L}$ in the new coordinate system \eqref{2.5}, is amplified due to the Couette flow. This results in an enhanced dissipation \eqref{1.12} and \eqref{1.11} of the nonzero (i.e., $x$-dependent) modes in the linear system. This effect becomes significant on a timescale of $O(\nu^{-\frac{1}{3}})\ll O(\nu^{-1})$. Due to the strong coupling induced by fluid rotation, it is challenging to directly estimate non-zero frequencies during the linear stability analysis process. Inspired by the approach proposed in \cite{CWZ2024}, we introduce two quantities, namely $W^{2}$ and $Q^{2}$, which exhibit a well-defined structure as shown in \eqref{3.28-1}. Specifically, for non-zero frequencies (i.e., $k\neq 0$), we will consider two cases separately. In the case when $l=0$,  we naturally have
\begin{equation}
\begin{cases}
\partial_{t} \widehat{\overline{Q_{\neq}^{2}}}+\nu \left(k^2+\left(\eta-kt\right)^2\right) \, \widehat{\overline{Q_{\neq}^{2}}}=0, \\
\partial_{t} \widehat{\overline{W_{\neq}^{2}}}+\nu \left(k^2+\left(\eta-kt\right)^2\right)  \widehat{\overline{W_{\neq}^{2}}}=0.
\end{cases}
\end{equation}
It is evident that both $\widehat{\overline{Q_{\neq}^{2}}}$ and $\widehat{\overline{W_{\neq}^{2}}}$ exhibit enhanced dissipation effects. The most challenging aspect lies in the case $l \neq 0$. To address this, we introduce a good unknown $K^2$ and a Fourier multiplier $m$ to handle potential stretching terms in \eqref{3.7}$_{2}$ for $K^2$, thereby obtaining a uniform boundedness estimate of $m \widehat{Q_{\neq}^{2}}$ and $m \widehat{K_{\neq}^{2}}$ with respect to time, as shown in \eqref{3.24}. By fully exploiting the relationship between $W_{\neq}^{2}$, $Q_{\neq}^{2}$, and $U_{\neq}$, we establish that equations governing the evolution of $U^{1}$ and $U^{2}$ are as follows:
\begin{equation*}
\begin{cases}
\left(\partial_{X X}+\partial_{Z Z}\right) U^{1}=\partial_{Z} W^{2}-\partial_{X Y}^{L} U^{2},  \\
\left(\partial_{X X}+\partial_{Z Z}\right) U^{3}=-\partial_{Z Y}^{L} U^{2}-\partial_{X} W^{2}.
\end{cases}
\end{equation*}
This combined with Fourier transform gives the enhanced dissipation estimate of velocity field $U_{\neq}$, see Proposition \ref{pro3.2} for details.

$(2)$ (Zero frequency $k=0$). The velocity field with zero frequency $k=0$ consists of a simple zero frequency component $\widetilde{u}_{0}$ and a double zero frequency component $\overline{u}_0$. It is important to note that the zero-frequency velocity does not exhibit enhanced dissipation effects. Furthermore, the appearance of rotation makes the zero frequency velocity $u_0$ not only prevent lift-up effects but also introduce dispersion estimates. 

As for the double zero frequency component $\overline{u}_0$, \eqref{1.6} follows heat equations under incompressible conditions, where the dissipation term $\nu\Delta\overline{u}_0$ degenerates into $\nu \partial_{yy} \overline{u}_0$. The explicit expression of the double zero frequency component \eqref{3.1.5} is directly obtained through attenuation of the heat kernel. When dealing with the simple zero frequency component $\widetilde{u}_{0}$, due to coupling between equations $\widetilde{u}_{0}^1$ and $\widetilde{u}_{0}^2$, we can calculate that the eigenvalues of linear operator $\mathcal{A}$ as shown in \eqref {3..44} form a pair of conjugate complex roots as follows:
$$\lambda_1=-\nu |\eta, l|^2+i \sqrt{\beta (\beta-1)} \frac{|l|}{|\eta, l|} \quad \text{  and  } \quad \lambda_2=-\nu |\eta, l|^2-i \sqrt{\beta (\beta-1)} \frac{|l|}{|\eta, l|},$$
which play an important role in the process of deducing the expression of the simple zero frequency $\widetilde{u}_{0}$, see Proposition \ref{pro3.3} for details. The appearance of the term $i\sqrt{\beta (\beta-1)}\frac{|l|}{|\eta, l|}$ indicates the presence of a dispersion relation. By fully exploiting the dispersion mechanism and oscillatory integrals (see subsection \ref{section3.3}), we establish Theorem \ref{1.1.}--(3), which provides insights into the amplitude decay behavior of $\widetilde{u}_{0}$. This linear dispersion effect is induced by (NSC) near Couette flow. 

Furthermore, understanding how this linear behavior interacts with nonlinearity is essential for studying the nonlinear stability of (NSC).

\textit{1.6.2. Nonlinear behavior.}  
The stability threshold we obtain for the nonlinear system \eqref{1.3} is determined by the nonlinear interactions, which are established through a perturbative approach based on the linear dynamics and a nonlinear bootstrap argument (see section \ref{4.2}). A precise analysis of the quadratically nonlinear interactions is crucial in this regard. The different types of nonlinear interactions can be classified as follows:
\begin{itemize}
	\item Zero frequency and zero frequency interaction: $0 \cdot 0 \rightarrow 0$.
	\item Zero frequency and non-zero frequencies interaction:  $0   \cdot \neq \rightarrow \neq$.
	\item Non-zero frequencies and non-zero frequencies interaction: $\neq  \cdot \neq \rightarrow \neq$ and $\neq  \cdot \neq \rightarrow 0$.
\end{itemize}

$(1)$ (Non-zero frequencies $k \neq 0$). From the above classification, in order to lead a non-zero frequency output through nonlinear interactions, at least one non-zero frequency needs to be included. However, due to rotation, it is not possible to directly close the energy estimates within the framework of $Q_{\neq}^{i}$ $(i=1, 2, 3)$ as done by Bedrossian, Germain and Masmoudi \cite{MR3612004} for the classical Navier-Stokes equations. Fortunately, when dealing with the linear stability problem in section \ref{section3}, we introduce a good unknown variable  $K_{\neq}^2$. For non-zero frequencies, by combining this variable with $Q_{\neq}^2$, along with the smallness assumption \eqref{1..11-2} and Fourier multipliers $m$ and $M$, we can establish a priori assumptions on $m M \check{K}_{\neq}^2$ and $m M Q_{\neq}^2$ (see \eqref{4..8}-\eqref{4..5}) in the nonlinear system \eqref{HSX88} and \eqref{HSX99}, under our constructed bootstrap framework. It should be noted that utilizing the structure of $K^{2}$ and $Q^{2}$ allows us to establish the following relationships
\begin{align*}
\left| \widehat{U_{\neq}^{1}} \right|+  \left| \widehat{U_{\neq}^{3}} \right| & \lesssim |k,l|^{-2} \left( \left|m M K_{\neq}^{2}\right|+\left| m M \widehat{Q_{\neq}^{2}}\right| \right), \\
\left| \widehat{U_{\neq}^{2}}\right| & \lesssim |k,l|^{-1} |k, \eta-kt, l|^{-1} \left| m M \widehat{Q_{\neq}^{2}}\right|.
\end{align*} 
Hence, we immediately get the estimates of the non-zero frequencies velocity $U_{\neq}^{i}$ $(i=1, 2, 3)$, see Theorem \ref{1..1}.

$(2)$ (Simple zero frequencies $k=0, l \neq 0$).  $\widetilde{u}_{0}^{j}$ ($j=1, 2$) satisfy the following equations
\begin{align*}
\begin{cases}
\partial_{t}\widetilde{u}_{0}^{1}-\nu\Delta_{y,z}\widetilde{u}_{0}^{1}+(1-\beta)\widetilde{u}_{0}^{2}=-\widetilde{\left(u\cdot\nabla u^{1} \right)_{0}}, \\
\partial_{t}\widetilde{u}_{0}^{2}-\nu\Delta_{y,z}\widetilde{u}_{0}^{2}+\beta \partial_{zz} \Delta_{y, z}^{-1}  \widetilde{u}_{0}^{1} =-\widetilde{\left(u\cdot\nabla u^{2} \right)_{0}}+\partial_{y} \Delta_{y, z}^{-1}\widetilde{ \left(\partial_{i}u_j \partial_{j}u_i \right)_{0}},
\end{cases}
\end{align*}
which contains the cross terms $(1-\beta)\widetilde{u}_{0}^{2}$ and $\beta \partial_{zz} \Delta_{y, z}^{-1}  \widetilde{u}_{0}^{1}$. This undermines the validity of attempting to directly close a priori assumptions \eqref{4..11}--\eqref{4..12} by using the $\widetilde{u}_{0}^{j}$ equations alone. To overcome it, our strategy is to use Duhamel's principle to represent $\widehat{\widetilde{u}_{0}^{1}}(t, \eta, l)$ as
\begin{align*}
\widehat{\widetilde{u}_{0}^{1}}(t, \eta, l)&= e^{-\nu(\eta^{2}+l^{2})t}\cos(ht)\widehat{\widetilde{u}_{0in}^{1}}+\frac{\beta-1}{h}e^{-\nu(\eta^{2}+l^{2})t}\sin(ht)\widehat{\widetilde{u}_{0in}^{2}}\nonumber\\
&\quad -\int_{0}^{t}e^{-\nu(\eta^{2}+l^{2})(t-\tau)}\widehat{\widetilde{\left(u\cdot \nabla u^{1} \right)_{0}}}(\tau)d\tau,
\end{align*}
where $h\triangleq \sqrt{\beta(\beta-1)}\frac{|l|}{|\eta,l|}$.  We then use the decay property of the heat semigroup (see Lemma \ref{lem2.3}), the triangle inequality of the norm and \eqref{4..16}--\eqref{4..24} to close a priori assumption \eqref{4..11}. Closing the priori hypothesis \eqref{4..12}--\eqref{4..13} for $\widetilde{u}_{0}^2$ and $\widetilde{u}_{0}^3$ is much easier by incompressible conditions $\widetilde{u}_{0}^{3}=-\partial_{z}^{-1} \partial_{y} \widetilde{u}_{0}^{2}$ and the relation $\left|  \widehat{\widetilde{u}_{0}^{2}} \right|= \left|  \widehat{ \Delta_{y, z}^{-1} \widetilde{Q}_{0}^{2}} \right| \leqslant \left| \widehat{Q_{0}^{2}}  \right|$.

$(3)$ (Double zero frequencies $k=l=0$). Due to the relatively slow effect of linear heat equations, the nonlinear control of double zero frequencies critically depends on the absence of self-interactions within these frequencies. Assuming \eqref{1..11-2}, favorable bounds can be established for $ \bar{u}_0^1$ and $ \bar{u}_0^3$.

Based on our treatment of nonlinear behavior, all bootstrap assumptions (see Proposition \ref{pro4.1}) can be satisfied when the initial condition with smallness satisfies \eqref{1..11-2}. In other words, the threshold for nonlinear stability is $\gamma=1$.

The rest of the paper is organized as follows. In section \ref{Preliminaries} we give some preliminaries for later use. In section \ref{section3}, we discuss the linear stability effects of solutions and give a rigorous mathematical proof of Theorem \ref{1.1.}. Sections \ref{section4} and \ref{4.2} is devoted to the brief idea of proof  of Theorem \ref{1..1}, including the bootstrap hypotheses and the weighted energy estimates of non-zero frequencies and zero frequency. The  proof details of nonlinear stability are given in sections \ref{sec5} and \ref{sec7}.
 
\section{Preliminaries}\label{Preliminaries} \label{sec2}
\subsection{Notation}
Given two quantities $A$ and $B$, we 
denote $A \lesssim B$, if there exists a constant $C$ such that $A\leqslant CB$ where  $C>0$ depends only on $\sigma$, but not on $\delta$, $\nu$ and $\beta$. We similarly denote $A \ll B$ if $A \leqslant cB$ for a small constant $c \in(0, 1)$ to emphasize the small size of the implicit constant.  $\sqrt{1+t^2}$ is represented by ${\left\langle {t} \right\rangle}$ and  $\sqrt{k^2+\eta^2+l^2}$ is  denoted by $|k, \eta, l|$.
We use the shorthand notation $dV=dxdydz.$ For two functions $f$ and $g$ and a norm $\|\cdot\|_{X}$, we write 
$$\|(f,g)\|_{X}=\sqrt{\|f\|_{X}^{2}+\|g\|_{X}^{2}}.$$
Unless specified otherwise, in the rest of this section $f$ and $g$ denote functions from $\mathbb{T} \times \mathbb{R} \times \mathbb{T}$ to $\mathbb{R}^{n}$ for some $n\in \mathbb{N}$.

\subsection{Fourier analysis and multiplier} 
We introduce the following notation for the Fourier transform  of a function $f= f(x, y, z)$. Given $(k,\eta,l)\in \mathbb{Z}\times \mathbb{R}\times \mathbb{Z}$, define 
\begin{equation}\label{2.1}
\mathcal{F}(f)=\widehat{f}\left(k, \eta, l\right)=\int_{x \in \mathbb{T}}\int_{y \in \mathbb{R}} \int_{z \in \mathbb{T}}f\left(x, y, z\right)e^{- i 2\pi \left(k x+ \eta y+ l z\right)}  dxdydz.
\end{equation}
Denoting the inverse operation to the Fourier transform by $\mathcal{F}^{-1}$ or $\check{f}$
\begin{equation}\label{2.2}
\mathcal{F}^{-1}(f)=\check{f} \left(x, y, z\right)=\sum_{k \in \mathbb{Z}}\int_{\eta \in \mathbb{R}} \sum_{l \in \mathbb{Z}} f \left(k, \eta, l\right)e^{2\pi i \left(k x+ \eta y+ l z\right)}  d\eta.
\end{equation}
In addition, we also define the Fourier transform in the $z$-direction of a function $h= h(y, z)$ by
\begin{equation*}
h_l (y) = \int_{z \in \mathbb{T}}h\left(y, z\right)e^{- 2\pi i  l z} dz.
\end{equation*}
For a general Fourier multiplier with symbol  $m(k,\eta,l)$, we write $mf$ to denote $\mathcal{F}^{-1}(m(k,\eta,l)\hat{f})$. Thus, it holds
\begin{equation*}
\widehat{m(D)f}=m(k,\eta,l)\widehat{f}.
\end{equation*}
We also define the Fourier symbol of $\nabla_{L}$ and  $-\Delta_{L}$ denoted by (\ref{2.6}) and (\ref{2.7}) respectively as
\begin{equation}\label{HSX444}
|\widehat{\nabla_{L}}|=|k,\eta-kt,l|
\end{equation}
and 
\begin{equation}\label{2.8}
p(t, k, \eta, l)\triangleq \widehat{-\Delta_{L}}=k^2+(\eta- kt)^2+l^2.
\end{equation}
Note that $p$ is time dependent and thus its time derivative is 
\begin{equation}\label{2.9}
\dot{p}=-2 k (\eta- kt).
\end{equation}

To avoid conflicting with the subscript just defined, when $f$ is vector valued we use a superscript to denote the components. For example, if $f$ is valued in $\mathbb{R}^{3}$ then we write $f=(f^{1},f^{2},f^{3})$.

\subsection{Functional spaces}
The Schwartz space $\mathcal{S}(\mathbb{R}^{d})$ with $d$ being the spatial dimension is the set of smooth functions $f$ 
on $\mathbb{R}^{d}$ such that for any $k\in \mathbb{N}$, we have
$$\|f\|_{k,\mathcal{S}}\triangleq \sup_{|\alpha|\leqslant k,~ x\in \mathbb{R}^{d}}\left(1+|x|\right)^{k}\left|\partial^{\alpha}f(x)\right|<\infty.$$

The Sobolev space $H^{s} (s\geqslant0) $ is given  by  the norm 
\begin{equation}\label{HSX-A}
	\left\| f \right\|_{H^{s}}=\left\| {\left\langle {D} \right\rangle}^{s} f \right\|_{L^2}=\left\| {\left\langle {k, \eta, l} \right\rangle}^{s} \widehat{f} \right\|_{L^2}.
\end{equation}
Recall that, for $s>\frac{3}{2}$, $H^{s}$ is an algebra. Hence, for any $f,g\in H^{s}$, one has
\begin{align}\label{2.3}
	\left\|f g\right\|_{H^{s}} \lesssim \left\|f \right\|_{H^{s}}\left\|g\right\|_{H^{s}}.
\end{align}
We write the associated inner product as
$$\langle f,g \rangle_{H^{s}}=\int \langle\nabla\rangle^{s}f\cdot \langle \nabla \rangle^{s} g dV.$$
For function $f(t,x,y,z)$ of space and time defined on the time interval $(a,b)$, we define the Banach space $L^{p}(a,b;H^{s})$ for $1\leqslant p\leqslant \infty$ by the norm
$$\|f\|_{L^{p}(a,b; H^{s})}=\|\|f\|_{H^{s}}\|_{L^{p}(a,b)}.$$
For simplicity of notation, we usually simply write $\|f\|_{L^{p}H^{s}}$ since the time-interval of integrating in this work will be the same basically everywhere.

Let us recall definitions of the Besov space $B_{p,q}^{s}$. For details, see Bergh and L$\ddot{\text{o}}$fstr$\ddot{\text{o}}$m \cite{BL1976}. We first introduce the Littlewood-Paley decomposition by means of $\{\varphi_{j}\}_{j=-\infty}^{\infty}$. Take a function $\phi\in C_{0}^{\infty}(\mathbb{R}^{3})$ with $
\text{supp}~ \phi=\{\xi\in \mathbb{R}^{3}: 1/2\leqslant |\xi|\leqslant 2\}$ such that $\sum_{j=-\infty}^{\infty}\phi (2^{-j}\xi)=1$ for all $\xi\neq 0$. The functions $\varphi_{j} (j=0, \pm 1,...)$ and $\psi$ are defined by
$$ \mathcal{F}\varphi_{j}(\xi)=\phi(2^{-j}\xi), \quad \mathcal{F}\psi(\xi)=1-\sum_{j=1}^{\infty}\phi(2^{-j}\xi),  $$
where $\mathcal{F}$ denotes the Fourier transform. Then, for $s\in \mathbb{R}$ and $1\leqslant p,q\leqslant \infty,$ we write
\begin{align}
   \|f\|_{B_{p,q}^{s}}&\equiv\|\psi\ast f\|_{L^{p}}+\left(\sum_{j=1}^{\infty}(2^{sj}\|\varphi_{j}\ast f\|_{L^{p}})^{q}\right)^{1/q},\quad 1\leqslant q<\infty,\nonumber\\
   \|f\|_{B_{p,\infty}^{s}}&\equiv \|\psi\ast f\|_{L^{p}}+\sup_{1\leqslant j<\infty} (2^{sj}\|\varphi_{j}\ast f\|_{L^{p}}),\quad  q=\infty,\nonumber
\end{align}
where $L^{p}$ denotes the usual Lebesgue space on $\mathbb{R}^{3}$ with the norm $\|\cdot\|_{L^{p}}$. The Besov space $B_{p,q}^{s}$ is defined by 
$$B_{p,q}^{s}\equiv\{f\in \mathcal{S}^{'}:\|f\|_{B_{p,q}^{s}}<\infty \}.$$
When $p=q=2$, it is easy to verify that Besov space $B_{2,2}^{s}(\mathbb{R}^{3})$ coincides with the Sobolev space $H^{s}(\mathbb{R}^{3})$.

The following lemma collects a few properties of  $f_0$  and  $f_{\neq}$, which will be frequently used in the subsequent sections. These properties can be easily verified via the  definition of zero  and non-zero frequencies as in \eqref{HSX333}. 
\begin{lem}
	Assume that the function  $f(x,y,z)$   is sufficiently regular. 
	Then it holds
	
	$(1)$  $f_0$  and  $f_{\neq}$  obey the following basic properties
	\begin{align*}
	&\left(\partial_{x} f\right)_0=\partial_{x} f_0=0, \quad \left(\partial_{y} f\right)_0=\partial_{y} f_0, \quad \left(\partial_{z} f\right)_0=\partial_{z} f_0, \\
	&\left( f_{\neq} \right)_0=0, \quad \left( \partial_{y} f \right)_{\neq}=\partial_{y} f_{\neq}, \quad \left( \partial_{z} f \right)_{\neq}=\partial_{z} f_{\neq}.
	\end{align*}
	
	$(2)$ If  $f$  is a divergence-free vector field, namely  $\nabla \cdot f=0$, then  $f_0$  and  $f_{\neq}$  are also divergence-free
	\begin{equation*}
	\nabla \cdot f_0=0 \quad \text { and } \quad \nabla \cdot f_{\neq}=0.
	\end{equation*}
	
	$(3)$  $f_0$  and  $f_{\neq}$  are orthogonal in  $L^{2}$, namely
	\begin{equation*}
	(f_0, f_{\neq})\triangleq\int f_0 f_{\neq} d V=0, \quad\|f\|_{L^{2}}^{2}=\|f_0\|_{L^{2}}^{2}+\|f_{\neq}\|_{L^{2}}^{2} .
	\end{equation*}
	In addition,  $\| f_0 \|_{L^{2}} \leqslant \|f\|_{L^{2}}$  and  $\|f_{\neq}\|_{L^{2}} \leqslant \|f\|_{L^{2}}$.
	
	Furthermore, if we replace $f_0$  with $\bar{f}$ and $f_{\neq}$ with $\widetilde{f}$, respectively, the above  properties are also valid.
\end{lem}

Note that the enhanced dissipation only affects the non-zero frequencies part  of fluid. So, it is essential to separate the zero frequency and the non-zero frequencies of velocity  in the calculation process. For
given functions $f$ and $h$, it is  not hard to find that 
\begin{equation*}
\left( f h \right)_0 = f_0 h_0 + \left(f_{\neq} h_{\neq} \right)_0 \quad \text{ and } \quad \left( f g \right)_{\neq} = f_0 h_{\neq} + f_{\neq} h_0 + \left( f_{\neq} h_{\neq} \right)_{\neq}.
\end{equation*}

We give the following  decay estimates of the heat semigroup  $e^{t \Delta}$, which is essential to get the energy estimates of zero frequency in Besov and Sobolev spaces.
\begin{lem}[\cite{KOT2003}]\label{lem2.3}
	Let  $-\infty < s_{0} \leqslant s_{1} < +\infty$, $1 \leqslant p, q \leqslant +\infty$. Then there holds
	\begin{align}\label{2..9}
		\left\|e^{t \Delta} f \right\|_{B_{p, q}^{s_{1}}(\mathbb{R}^3)} \leqslant & C\left(1+t^{-\frac{1}{2}\left(s_{1}-s_{0}\right)}\right)\| f \|_{B_{p, q}^{s_{0}}(\mathbb{R}^3)},
	\end{align}
in particular, for $p=q=2$, one has
	\begin{align}\label{2..10}
	\left\|e^{t \Delta} f \right\|_{H^{s_{1}}(\mathbb{R}^{3})} \leqslant & C\left(1+t^{-\frac{1}{2}\left(s_{1}-s_{0}\right)}\right)\| f \|_{H^{s_{0}}(\mathbb{R}^{3})}.
	\end{align}
Furthermore,  \eqref{2..9} and \eqref{2..10} also hold for the spatial domain  $\Omega=\mathbb{T} \times \mathbb{R} \times \mathbb{T}$.
\end{lem}

In the field of harmonic analysis, the Van der Corput lemma is an estimate for oscillatory integrals. In this paper, we use this lemma to obtain the dispersive estimates on the zero frequency of  velocity. 
\begin{lem}[\textbf{Van der Corput lemma} \cite{SM1993}]\label{lem2.1}
	Suppose  $\phi(\xi)$ is real-valued and smooth in $(a, b)$, and that $\left|\phi^{(k)}(\xi)\right| \geqslant 1$  for all  $\xi \in(a, b)$. For any  $\lambda \in \mathbb{R}$, there is a positive constant  $c_{k}$, which does not depend on  $\phi$ and  $\lambda$  such that
	\begin{equation}
		\left|\int_{a}^{b} e^{i \lambda \phi(\xi)} d\xi \right| \leqslant c_{k} \lambda^{-\frac{1}{k}}
	\end{equation}
	holds when:
	
	$(1)$  $k \geqslant 2$, or
	
	$(2)$  $k=1$ and $\phi(\xi)$ is monotone.
\end{lem}
 
\section{The linear stability }\label{section3}
In this section, we elaborate on  the linear stability results and corresponding proof of (NSC) near the Couette flow,  see Theorem \ref{1.1.}.  Now we neglect the effect of nonlinear terms on the right-hand side of \eqref{1.3},  that is, we only need to consider the  stability of the linearized  equations:
\begin{equation}\label{1.6}
	\begin{cases}
			\partial_{t} u+ y\partial_{x} u-\nu \Delta u+\begin{pmatrix}
					(1-\beta)u^2 \\
					\beta u^1    \\
					0
				\end{pmatrix}+(\beta-2)\nabla \Delta^{-1}\partial_{x}u^{2}-\beta\nabla \Delta^{-1}\partial_{y}u^{1}=0, \\
			u(t=0)=u_{\mathrm{in}}.
		\end{cases}
\end{equation}

The proof of Theorem \ref{1.1.}    can be derived by the following Propositions  \ref{pro3.2}, \ref{pro3.3} and \ref{pro3.4}. In order to achieve this, we list the subsections \ref{section3.1}-\ref{section3.3} below.

\subsection{Linear enhanced dissipation and inviscid damping on $U_{\neq}$}\label{section3.1}
In this section, we are committed to deriving the estimates of the enhanced dissipation on the non-zero frequency velocity field $U_{\neq}$, and as a by-product, obtaining the inviscid damping of $U_{\neq}^2$. To do this, by introducing three quantities $q^{2}\triangleq\Delta u^2$,   $W^2$ and   $K^2$ in Fourier space and using the decay property  provided by the Fourier multiplier $m$,  we establish the appropriate energy functional to get \eqref{1.12} and \eqref{1.11}. To begin with, we give the definitions of multipliers $m$ and $M$.
Consider the following linear equation
\begin{equation}\label{4...1}
	\frac{1}{2}\partial_{t}f+ \partial_{X Y}^{L} \Delta_{L}^{-1}f-\nu \Delta_{L} f=0,
\end{equation}
which can be seen as a competition between the linear stretching term  $\partial_{X Y}^{L} \Delta_{L}^{-1}f$ and the dissipation term  $\nu \Delta_{L} f$.
Applying the Fourier transform to \eqref{4...1} yields
\begin{equation*}
	\frac{1}{2} \partial_{t}\widehat{f}+\frac{ k (\eta- kt)}{k^2+(\eta- kt)^2+l^2} \widehat{f}+\nu \left(k^2+(\eta- kt)^2+l^2\right) \widehat{f}=0.
\end{equation*}
If $k\neq 0$, the factor $\frac{ k (\eta- kt)}{k^2+(\eta-kt)^2+l^2}$ is positive for $ t<\frac{\eta}{k}$,  in which case the term $ \partial_{X Y}^{L} \Delta_{L}^{-1}f$ can be viewed as a damping term. As for the factor $\nu \left(k^2+(\eta- kt)^2+l^2\right)$, this implies enhanced dissipation for $k\neq 0$. When $k\neq 0$, we find that  the following inequality is always true, which compares the sizes of these two factors 
\begin{equation}\label{4...3}
	\nu \left( k^2+\left(\eta- kt\right)^2+l^2 \right) \gg \frac{ |k \left( \eta- kt\right)|}{k^2+\left(\eta- kt\right)^2+l^2}, \text{ if } \Big| t-\frac{\eta}{k} \Big|\gg \nu^{-\frac{1}{3}}.
\end{equation}
Indeed,    $\frac{ |k \left( \eta- kt\right)|}{\nu \left(k^2+\left(\eta- kt\right)^2+l^2\right)^2} \leqslant \frac{| t-\frac{\eta}{k}|}{\nu \left(1+| t - \frac{\eta}{k}|^2\right)^2} \leqslant 1$ and note that  $\dfrac{ x}{\nu(1+x^2)^2} \ll 1$ for $|x|\gg  \nu^{-\frac{1}{3}}$. 
Thus, from \eqref{4...3}, if $\Big| t-\frac{\eta}{k} \Big|\gg \nu^{-\frac{1}{3}}$, the dissipation term $\nu \Delta_{L} f$ overcomes the stretching term $ \partial_{XY}^{L}\Delta_{L}^{-1}f$. As a result, when  $0 < t -\frac{\eta}{k}\lesssim \nu^{-\frac{1}{3}}$, the stretching overcomes dissipation. The following  Figure \ref{fig:diagram1} better illustrates the game relationship between dissipation and stretching
\begin{figure}[H]
	\centering
	\includegraphics[width=0.8\textwidth]{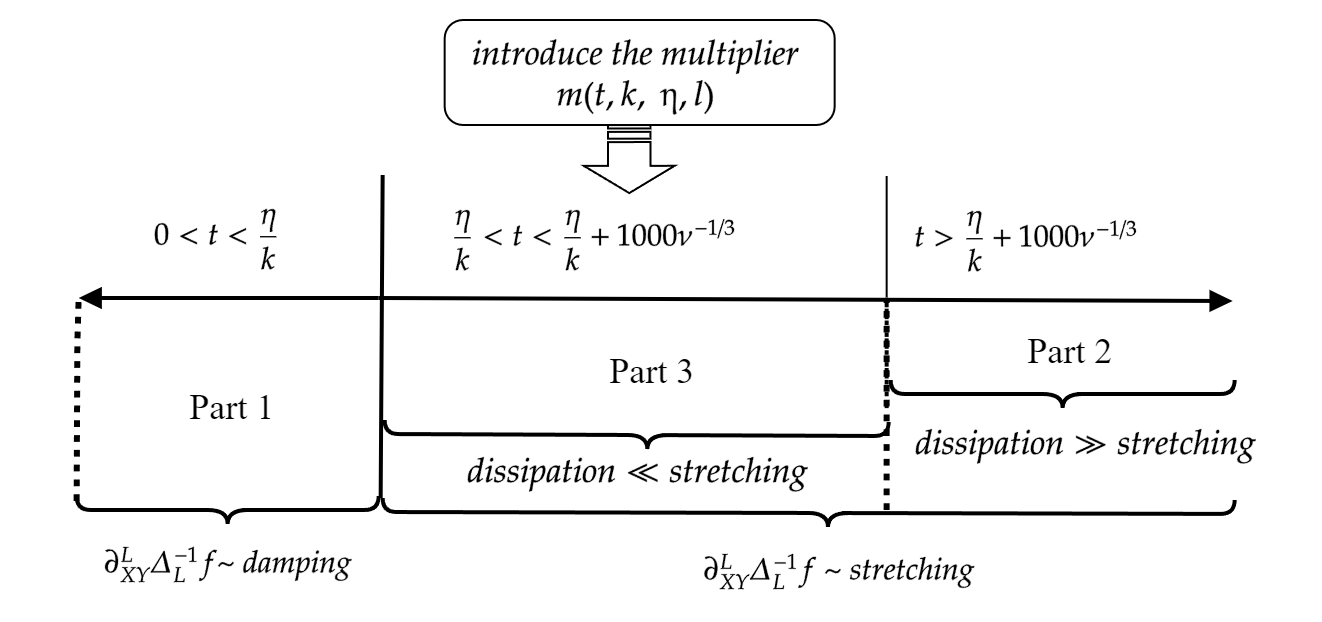}
	\caption{dissipation versus stretching}
	\label{fig:diagram1}
\end{figure}

In fact,  the trickiest is the case when  $0 <  t -\frac{\eta}{k}\lesssim \nu^{-\frac{1}{3}}$. To deal with this rang of $t$,  we  introduce some definitions of  the Fourier multipliers.

(1) Define the multiplier $m(t, k, \eta, l)$ by $m(t=0, k, \eta, l)=1$ and the following ordinary differential  equations:
\begin{align}\label{4...4}
	\frac{\dot{m}}{m}=\left\{\begin{array}{ll}
		\frac{ k \left( \eta- kt \right)}{k^2+\left(\eta- kt\right)^2+l^2} & \text{ if }  t \in \left[ \frac{\eta}{k}, \frac{\eta}{k}+1000 \nu^{-\frac{1}{3}} \right], \\
		0  & \text{ if }  t \notin \left[ \frac{\eta}{k}, \frac{\eta}{k}+1000 \nu^{-\frac{1}{3}} \right].
	\end{array}\right.
\end{align}
This multiplier is such that if $f$ solves the above equation (\ref{4...1}) and $0< t-\frac{\eta}{k}<1000\nu^{-1/3} $, then $mf$ solves
$$ \partial_{t}(mf)-\nu \Delta_{L}(mf)=0,$$
and this equation is perfectly well behaved. That is, the growth that $f$ undergoes is balanced by the decay of the multiplier $m$.

$(2)$ Define the additional multiplier $M(t, k, \eta, l)$ by $M(t=0, k, \eta, l)=1$ and
\begin{itemize}
	\item if  $k=0$, $M(t, k=0, \eta, l)=1$ for all $t$;
	\item if  $k \neq 0$,
	\begin{align}\label{4.5}
		\frac{\dot{M}}{M}&= \frac{- \nu^{\frac{1}{3}} }{\left[\nu^{\frac{1}{3}} \big( t-\frac{\eta}{k} \big)\right]^{2}+1}.
	\end{align}
\end{itemize}

For the expressions and properties of the multiplier $m$, there is a similar situation as in \cite{MR3612004}, so we use it as a known conclusion and ignore the proof process.
\begin{lem}
	$(1)$ The multiplier $m(t, k, \eta, l)$ can be given by the following exact formula:
	\begin{itemize}
		\item if $k=0$: $m(t, k=0, \eta, l)=m(t=0, k, \eta, l)=1$;
		\item if  $k \neq 0$, $\frac{\eta}{k}+1000  \nu^{-\frac{1}{3}}<0$: $m(t, k, \eta, l)=1$;
		\item if  $k \neq 0$, $\frac{\eta}{k}<0$ and $\frac{\eta}{k}+1000 \nu^{-\frac{1}{3}}>0$:
		\begin{align}
			m(t, k, \eta, l)=\left\{\begin{array}{ll}
				\dfrac{\sqrt{k^{2}+\eta^{2}+l^{2}}}{\sqrt{k^{2}+(\eta- k t)^{2}+l^{2}}}  & \text{ if  } 0< t<\frac{\eta}{k}+1000  \nu^{-\frac{1}{3}}, \\
				\dfrac{\sqrt{k^{2}+\eta^{2}+l^{2}}}{\sqrt{k^{2}+\left(1000 k \nu^{-\frac{1}{3}}\right)^{2}+l^{2}}} & \text{ if  }  t>\frac{\eta}{k}+1000 \nu^{-\frac{1}{3}};
			\end{array}\right.
		\end{align}
		\item if  $k \neq 0$, $\frac{\eta}{k}>0$:
		\begin{align}
			m(t, k, \eta, l)=\left\{\begin{array}{lll}
				1  &\text{if}~  t<\frac{\eta}{k}, \\
				\dfrac{\sqrt{k^{2}+l^{2}}}{\sqrt{k^{2}+(\eta- k t)^{2}+l^{2}}} & \text{if}~ \frac{\eta}{k}<  t<\frac{\eta}{k}+1000 \nu^{-\frac{1}{3}}, \\
				\dfrac{\sqrt{k^{2}+l^{2}}}{\sqrt{k^{2}+\left(1000 k  \nu^{-\frac{1}{3}}\right)^{2}+l^{2}}}  & \text{ if}~   t>\frac{\eta}{k}+1000  \nu^{-\frac{1}{3}}.
			\end{array}\right.
		\end{align}
	\end{itemize}
	
	$(2)$ In particular, $m(t, k, \eta, l)$ are bounded above and below, but its lower bound depends on $\nu$ 
	\begin{equation}\label{4.11}
	\nu^{\frac{1}{3}}\lesssim m\left(t, k, \eta, l \right) \leqslant 1.
	\end{equation}
	
	$(3)$ $m(t, k, \eta, l)$ and the frequency have the following relationship: 
	\begin{equation}\label{lem4.1.11}
		m\left(t, k, \eta, l \right) \gtrsim \frac{\sqrt{k^2+l^2}}{\sqrt{k^2+\left(\eta- kt\right)^2+l^2}}.
	\end{equation}	
\end{lem}

The following lemma embodies some properties of the ghost multiplier $M(t, k, \eta, l)$.
\begin{lem}\label{lem4.2}
	$(1)$ The multiplier $M$ have positive upper and lower bounds:
	\begin{equation}\label{4.13}
		0<c<M(t, k, \eta, l) \leqslant 1
	\end{equation}
	for a universal constant $c$ that does not depend on $\nu$ and frequency. 
	
	$(2)$ For $k \neq 0$, we have
	\begin{equation}\label{4.15}
		1 \lesssim \nu^{-\frac{1}{6}} \sqrt{-\dot{M} M}\left(k, \eta, l\right)+\nu^{\frac{1}{3}}\big|k, \eta- kt, l \big|.
	\end{equation}

\end{lem}

\begin{proof}
	First, we shall prove $(1)$. According to \eqref{4.5}, we have
	$\frac{\dot{M}}{M}\leqslant 0$ for any $k$, $l \in \mathbb{Z}$, $\eta \in \mathbb{R}$.
	\begin{equation*}
		\ln M (t, k, \eta, l) \leqslant \ln M (t=0, k, \eta, l)=0.
	\end{equation*}
	For any $t\geqslant 0,$	we can easily obtain the upper bound of $M$
	\begin{equation*}
		M (t, k, \eta, l) \leqslant 1. 
	\end{equation*}
	If $k=0$, then $M(t, k=0, \eta, l)=1$. For $k \neq 0$, by \eqref{4.5}, we have
	\begin{align*}
		\ln M (t, k, \eta, l)& =\int_{0}^{t} \frac{-\nu^{\frac{1}{3}} }{\left[\nu^{\frac{1}{3}} \big( \tau-\frac{\eta}{k} \big)\right]^{2}+1} d\tau   \\
		&=- \nu^{\frac{1}{3}} \int_{-\frac{\eta}{k}}^{t-\frac{\eta}{k}} \frac{1} {\left[\nu^{\frac{1}{3}} s\right]^{2}+1} ds \\
		&=-\int_{-\nu^{\frac{1}{3}}\frac{ \eta}{k}}^{\nu^{\frac{1}{3}} \left(t-\frac{\eta}{k}\right)} \frac{1}{r^{2}+1} dr \\
		&=\arctan\left( -\nu^{\frac{1}{3}}\frac{ \eta}{k} \right)- \arctan \left(\nu^{\frac{1}{3}} \left(t-\frac{\eta}{k}\right)\right).
	\end{align*}
	Note that $\arctan r \in [-\frac{\pi}{2}, \frac{\pi}{2}]$. Hence, 
	\begin{equation*}
		M (t, k, \eta, l)=\exp \left[ \arctan\left( -\nu^{\frac{1}{3}}\frac{ \eta}{k} \right)- \arctan \left(\nu^{\frac{1}{3}} \left( t-\frac{\eta}{k}\right)\right) \right] \in [e^{-\pi}, 1].
	\end{equation*}
	
	Next, we give the proof of (2). For $k\neq 0$, we have $|k, \eta- kt, l|\geqslant | t-\frac{\eta}{k}|$.  If $\big|t-\frac{\eta}{k} \big| \geqslant \nu^{-\frac{1}{3}}$, then   (\ref{4.15})  immediately holds. Otherwise, if $\big| t-\frac{\eta}{k} \big| \leqslant \nu^{-\frac{1}{3}}$, then we have
	\begin{equation*}
		\sqrt{-\dot{M} M}=\sqrt{-\frac{\dot{M}}{M} (M)^2} \gtrsim \sqrt{-\frac{\dot{M}}{M}}=\sqrt{\frac{ \nu^{\frac{1}{3}} }{\left[\nu^{\frac{1}{3}} \big| t-\frac{\eta}{k} \big|\right]^{2}+1}} \geqslant \frac{\nu^{\frac{1}{6}}}{\sqrt{2}}.
	\end{equation*}
\end{proof}
Using Lemma \ref{lem4.2}, we directly give the following inference.
\begin{cor}\label{cor4.1}
	For any $g$ and $s \geqslant 0$, the following inequality holds
	\begin{equation}\label{4.17}
		\left\|  g_{\neq}  \right\|_{L^2 H^{s}} \lesssim \nu^{-\frac{1}{6}} \left( \left\|  \sqrt{-\dot{M} M} g_{\neq} \right\|_{L^2 H^{s}}+ \nu^{\frac{1}{2}} \left\|  \nabla_{L} g_{\neq} \right\|_{L^2 H^{s}}\right),
	\end{equation}
	which will be used frequently below.
\end{cor}

Now, we prove the linear enhanced dissipation and inviscid damping on $U_{\neq}$. To achieve this, we give attenuation estimates for intermediate variables.
\begin{pro}\label{pro3.1}
	Assume that $\nu > 0$ and $\beta \in \mathbb{R}$ with $|\beta| \geqslant 2$, then there holds
	\begin{align}
	\Big|m \widehat{Q_{\neq}^{2}} \Big|^2+\Big|m K_{\neq}^2 \Big|^2 \leqslant &  e^{-\frac{1}{12} \nu t^{3}} \left( \Big|\widehat{Q_{\mathrm{in}}^{2}} \Big|^2+\Big|K_{\mathrm{in}}^2 \Big|^2\right), \label{3.24}\\
	\Big| \widehat{U_{\neq}^2} \Big|^2 \lesssim & e^{-\frac{1}{12} \nu t^{3}} p^{-1} \left( \Big|\widehat{Q_{\mathrm{in}}^{2}} \Big|^2+\Big|K_{\mathrm{in}}^2 \Big|^2\right), \label{3.25}\\
	\Big| \widehat{W_{\neq}^2} \Big|^2 \lesssim & e^{-\frac{1}{12} \nu t^{3}} \left( \Big|\widehat{Q_{\mathrm{in}}^{2}} \Big|^2+\Big|K_{\mathrm{in}}^2 \Big|^2\right) \label{3.26}.
	\end{align}
\end{pro}

\begin{proof}
First, we apply $\Delta$ to the $u^2$ equation
\begin{equation}\label{3..1}
	\partial_{t} \Delta u^{2}-\nu \Delta \Delta u^{2}+y \partial_{x} \Delta u^{2}+\beta \partial_{x y} u^{2}+\beta\left(\partial_{x x}+\partial_{z z}\right) u^{1}=0.
\end{equation}
Inspired by \cite{CWZ2024}, we introduce the quantity $w^{2} \triangleq \partial_{z} u^{1}-\partial_{x} u^{3}$.  Combining with $u^{1}$ and $u^{3}$ equations, we obtain the following fact
\begin{equation}\label{3.118}
	\partial_{t} w^{2}-\nu \Delta w^{2}+y \partial_{x} w^{2}+(1-\beta) \partial_{z} u^{2}=0.
\end{equation}
Furthermore, due to the incompressible condition $\partial_{x} u^1+\partial_{z} u^3=-\partial_{y}u^2$ and the definition of $w^2$, we have
\begin{equation}\label{3.3}
	\left(\partial_{x x}+\partial_{z z}\right) u^{1}=\partial_{z} w^{2}-\partial_{x y} u^{2}.
\end{equation}
Let $q^2\triangleq \Delta u^2$. Combining \eqref{3..1} with  \eqref{3.118} gives
\begin{equation}\label{3.4}
	\begin{cases}
		\partial_{t} q^{2}+y \partial_x q^{2}-\nu \Delta q^{2}+\beta \partial_{z} w^{2}=0, \\ 
		\partial_{t} w^{2}+y \partial_x w^{2}-\nu \Delta w^{2}+(1-\beta) \partial_{z} \Delta^{-1} q^{2}=0.
	\end{cases}
\end{equation}

In the moving frame \eqref{2.5}, we denoting $Q^2(t, X, Y, Z)=q^2(t, x, y, z)$ and $W^2(t, X, Y, Z)=w^2(t, x, y, z)$. Thus,  by taking the Fourier transform of the spatial variable,  \eqref{3.4} can be written
\begin{equation}\label{3.28-1}
	\begin{cases}
		\partial_{t} \widehat{Q_{\neq}^{2}}+i \beta l \, \widehat{W_{\neq}^{2}}=-\nu p \, \widehat{Q_{\neq}^{2}}, \\
		\partial_{t} \widehat{W_{\neq}^{2}}-i(1-\beta)l p^{-1}  \widehat{Q_{\neq}^{2}}=-\nu p  \widehat{W_{\neq}^{2}}.
	\end{cases}
\end{equation}
Note that when $l=0$, it is easy to obtain the enhanced dissipation on $\widehat{Q^2_{\neq}}$ and $\widehat{W^2_{\neq}}$. Here we 
assume $l \neq 0$ and $|\beta| \geqslant 2$, we introduce the new unknowns    
\begin{equation}\label{3.13}
	K^2=\frac{i\sqrt{\beta}}{\sqrt{\beta-1}}p^{\frac{1}{2}}\widehat{W^{2}},   \quad \check{K}^2= -i\frac{\sqrt{\beta}}{\sqrt{\beta-1}}|\nabla_{L}|W^{2},
\end{equation}
then $\widehat{Q_{\neq}^2}$ and $K_{\neq}^2$ satisfy
\begin{equation}\label{3.7}
	\begin{cases}
		\partial_{t} \widehat{Q_{\neq}^2} + \sqrt{\beta (\beta-1)} \, l p^{-\frac{1}{2}}  K_{\neq}^2 = -\nu p \, \widehat{Q_{\neq}^2}, \\
		\partial_{t} K_{\neq}^2 - \dfrac{1}{2} \dfrac{\dot{p}}{p} K_{\neq}^2 - \sqrt{\beta (\beta-1)} \, l p^{-\frac{1}{2}}  \widehat{Q_{\neq}^2} = -\nu p K_{\neq}^2.
	\end{cases}
\end{equation}
By basic energy estimate, we give
\begin{equation}\label{3.8}
	\frac{1}{2} \frac{d}{dt} \left( \big|\widehat{Q_{\neq}^2} \big|^2+ \big|K_{\neq}^2 \big|^2\right) - \frac{1}{2} \frac{\dot{p}}{p} \big|K_{\neq}^2 \big|^2 +\nu p \left( \big|\widehat{Q_{\neq}^2}\big|^2 + \big|K_{\neq}^2\big|^2 \right) = 0,
\end{equation}
that is
\begin{align}\label{3.16}
	&\frac{1}{2} \frac{d}{dt} \left( \big|\widehat{Q_{\neq}^2} \big|^2+ \big|K_{\neq}^2 \big|^2\right) +\nu \left( k^2+\left(\eta-kt\right)^2+l^2 \right) \left( \big|\widehat{Q_{\neq}^2}\big|^2 + \big|K_{\neq}^2\big|^2 \right) \nonumber\\&\quad+ \frac{k\left(\eta-kt\right)}{k^2+\left(\eta-kt\right)^2+l^2} \big|K_{\neq}^2 \big|^2= 0.
\end{align}
Combining the properties of the multiplier $m$ with \eqref{3.16}, we further obtain
\begin{align}\label{HSX0000}
&\frac{1}{2} \frac{d}{dt}\left( \Big| m \widehat{Q_{\neq}^{2}} \Big|^{2}+ \Big|m K_{\neq}^{2} \Big|^{2} \right)+ \Big| \sqrt{-\dot{m}m } \,\widehat{Q_{\neq}^{2}} \Big|^2+ \Big| \sqrt{-\dot{m}m } \, K_{\neq}^{2}  \Big|^2 + \nu p\left(\Big|m \widehat{Q_{\neq}^{2}} \Big|^{2}+ \Big|m K_{\neq}^{2}\Big|^{2}\right)  \nonumber\\&\quad
=\frac{-k(\eta-kt)}{k^2+(\eta-kt)^2+l^2} \Big|m K_{\neq}^{2}\Big|^{2} \cdot 1_{\{t\in \mathbb{R}_+: \, 0<t<\frac{\eta}{k}\}} -\frac{\dot{m}}{m}\Big|m K_{\neq}^{2}\Big|^{2} \cdot 1_{\{t\in \mathbb{R}_+: \, \frac{\eta}{k}<t<\frac{\eta}{k}+1000 \nu^{-\frac{1}{3}}\}}  \nonumber\\&\qquad + \frac{-k(\eta-kt)}{k^2+(\eta-kt)^2+l^2} \Big|m K_{\neq}^{2}\Big|^{2} \cdot 1_{\{t\in \mathbb{R}_+: \, t>\frac{\eta}{k}+1000 \nu^{-\frac{1}{3}}\}}  \nonumber\\&\quad
\leqslant \frac{\nu}{2} p \Big|m K_{\neq}^{2}\Big|^{2}+ \Big| \sqrt{-\dot{m}m } \, K_{\neq}^{2}  \Big|^2,  
\end{align}
where we have used \eqref{4...3}. Indeed,   dissipation overcomes stretching if $ t>\frac{\eta}{k}+1000  \nu^{-\frac{1}{3}}$ and $\frac{1}{2}$ in the last term as in the above inequality can be chosen.   The two terms on the right-hand side of \eqref{HSX0000} are absorbed by the dissipative term,  which gives
\begin{equation}\label{3.28}
\frac{1}{2} \frac{d}{dt}\left( \Big| m \widehat{Q_{\neq}^{2}} \Big|^{2}+ \Big|m K_{\neq}^{2} \Big|^{2} \right)+ \frac{\nu}{2} p\left(\Big|m \widehat{Q_{\neq}^{2}} \Big|^{2}+ \Big|m K_{\neq}^{2}\Big|^{2}\right) \leqslant 0.
\end{equation}
Directly calculate gives
$$\int_{0}^{t}p(s)ds=(k^{2}+l^{2})t+\left((\eta-\frac{1}{2}kt)^{2}+\frac{1}{12}k^{2}t^{2}\right)t\geqslant\frac{1}{12} k^{2} t^{3}. $$ 
Thus, integrating \eqref{3.28} with respect to $t$ and  obtaining that
\begin{equation}\label{3.14}
\Big| m \widehat{Q_{\neq}^{2}} \Big|^{2}+ \Big|m K_{\neq}^{2} \Big|^{2} \leqslant e^{-\frac{1}{12} \nu k^2 t^3} \left( \Big| \widehat{Q_{\mathrm{in}}^{2}} \Big|^2+ \Big| K_{\mathrm{in}}^{2} \Big|^2\right).
\end{equation}

Finally, recalling the definition \eqref{3.13} of $K^2$ and $|\beta| \geqslant 2$, by \eqref{lem4.1.11} we have the following fact
\begin{align}\label{3..100}
p^{\frac{1}{2}}\Big| \widehat{U_{\neq}^2} \Big| \leqslant \frac{\sqrt{k^2+l^2}}{\sqrt{k^2+(\eta-kt)^2+l^2}} \left(k^2+(\eta-kt)^2+l^2 \right) \Big| \widehat{U_{\neq}^2} \Big| \lesssim \Big| m \widehat{\Delta_{L} U_{\neq}^{2}} \Big| =\Big| m \widehat{Q_{\neq}^{2}} \Big|, 
\end{align}
and 
\begin{equation}
\Big| \widehat{W_{\neq}^2} \Big| \leqslant \frac{\sqrt{k^2+l^2}}{\sqrt{k^2+(\eta-kt)^2+l^2}} \sqrt{k^2+(\eta-kt)^2+l^2}  \Big| \widehat{W_{\neq}^2} \Big| \lesssim \Big| m \frac{i \sqrt{\beta}}{\sqrt{\beta-1}} p^{\frac{1}{2}} \widehat{W_{\neq}^2} \Big|= \Big| m K_{\neq}^{2} \Big|.\nonumber
\end{equation}
These together with \eqref{3.14} lead to \eqref{3.24}--\eqref{3.26}.
\end{proof}
With the  Proposition \ref{pro3.1} at hand, we give the following result  about  the enhanced dissipation of the $L^2$-norm of $U_{\neq}$ and the inviscid damping effect of $U_{\neq}^2$,  which corresponds to \eqref{1.12} and \eqref{1.11} in  Theorem \ref{1.1.} when $s=0$. In fact, when $s>0$, the enhanced dissipation of $U_{\neq}$ can also be obtained by using a similar treatment. 
\begin{pro}\label{pro3.2}
	Assume that  $\nu > 0$  and  $\beta \in \mathbb{R}$ with $|\beta| \geqslant 2$, it holds
	\begin{align}
	\left\|  U^2_{\neq} (t) \right\|_{L^2} \lesssim  & {\left\langle {t} \right\rangle}^{-1} e^{-\frac{1}{24} \nu t^3} \left(\left\|  u^2_{\mathrm{in}}\right\|_{H^{3}} + \left\|  w^2_{\mathrm{in}}\right\|_{H^{2}} \right), \label{3.31}\\
	\left\| (U^1_{\neq}, U^3_{\neq}) (t) \right\|_{L^2} \lesssim & \,  e^{-\frac{1}{24} \nu t^3} \left(\left\|  u^2_{\mathrm{in}}\right\|_{H^{2}} + \left\|  w^2_{\mathrm{in}}\right\|_{H^{1}} \right).  \label{3.32}	  
	\end{align}
\end{pro}

\begin{proof}

Let $U^2(t, X, Y, Z)=u^2(t, x, y, z)$.  Due to the fact $p^{-1} \lesssim {\left\langle {t} \right\rangle}^{-2} |k, \eta, l|^2$ and Proposition \ref{pro3.1}, we have
\begin{align}\label{3.37}
\left\|u_{\neq}^{2}(t)\right\|_{L^{2}}^{2}=\left\|U_{\neq}^{2}(t)\right\|_{L^{2}}^{2}& \lesssim  \sum_{(k, l) \in \mathbb{Z}^2}  e^{-\frac{1}{12} \nu t^{3}} \int_{\mathbb{R}} p^{-1}\left(\Big|\widehat{Q_{\mathrm{in}}^{2}}\Big|^{2}+\Big|K_{\mathrm{in}}^{2}\Big|^{2}\right) d \eta \nonumber \\
&\lesssim \sum_{(k, l) \in \mathbb{Z}^2}  \frac{e^{-\frac{1}{12} \nu t^{3}}}{{\left\langle {t} \right\rangle}^2} \int_{\mathbb{R}} |k, \eta, l|^2 \left(|k, \eta, l|^4 \Big|\widehat{U_{\mathrm{in}}^{2}}\Big|^{2}+|k, \eta, l|^2 \Big|\widehat{W_{\mathrm{in}}^{2}}\Big|^{2}\right) d \eta  \nonumber\\
& \lesssim \frac{e^{-\frac{1}{12} \nu t^{3}}}{{\left\langle {t} \right\rangle}^2} \left( \left\| u_{\mathrm{in}}^2 \right\|_{H^{3}}^2+ \left\| w_{\mathrm{in}}^2 \right\|_{H^{2}}^2\right).
\end{align}

Let $U^1(t, X, Y, Z)=u^1(t, x, y, z)$ and $U^3(t, X, Y, Z)=u^3(t, x, y, z)$. Using the incompressible condition $\partial_{x} u^1+\partial_{z} u^3=-\partial_{y}u^2$ and the definition of $w^2$, one gets
\begin{equation}\label{3.20}
\begin{cases}
\left(\partial_{X X}+\partial_{Z Z}\right) U^{1}=\partial_{Z} W^{2}-\partial_{X Y}^{L} U^{2},  \\
\left(\partial_{X X}+\partial_{Z Z}\right) U^{3}=-\partial_{Z Y}^{L} U^{2}-\partial_{X} W^{2}.
\end{cases}
\end{equation}
By \eqref{3.20}$_{1}$, we get $\widehat{U_{\neq}^1}$ which satisfies
$$\widehat{U_{\neq}^1}=-i|k, l|^{-2} \, l \, \widehat{W_{\neq}^{2}}-|k, l|^{-2} k (\eta-kt) \widehat{U_{\neq}^2},
$$
this together with \eqref{3.25}--\eqref{3.26} leads to $L^2$-norm of $u_{\neq}^1(t)$
\begin{align}\label{3.21}
\left\| u_{\neq}^1(t) \right\|_{L^2}^2=\left\| U_{\neq}^1(t) \right\|_{L^2}^2 &\leqslant  \sum_{(k, l) \in \mathbb{Z}^2} \left(\int_{\mathbb{R}} |k, l|^{-4} |l|^{2} \Big|\widehat{W_{\neq}^2}(t)\Big|^{2} d\eta+\int_{\mathbb{R}}|k, l|^{-4} |k|^{2} |\eta-kt|^2 \Big|\widehat{U_{\neq}^2} (t)\Big|^{2} d\eta\right) \nonumber \\
&\lesssim \sum_{(k, l) \in \mathbb{Z}^2} \int_{\mathbb{R}} \left(  \Big|\widehat{W_{\neq}^2}(t)\Big|^{2} + e^{-\frac{1}{12} \nu t^{3}}   |k, l|^{-4} |k|^{2} |\eta-kt|^2 p^{-1}\left(\Big|\widehat{Q_{\mathrm{in}}^{2}} \Big|^2+\Big|K_{\mathrm{in}}^2 \Big|^2\right) \right) d\eta \nonumber \\
&\lesssim \sum_{(k, l) \in \mathbb{Z}^2} e^{-\frac{1}{12} \nu t^{3}} \int_{\mathbb{R}}  \left(\Big|\widehat{Q_{\mathrm{in}}^{2}} \Big|^2+\Big|K_{\mathrm{in}}^2 \Big|^2 \right) d \eta \nonumber\\
& \lesssim \sum_{(k, l) \in \mathbb{Z}^2} e^{-\frac{1}{12} \nu t^{3}} \int_{\mathbb{R}} \left(|k, \eta, l|^4 \Big|\widehat{U_{\mathrm{in}}^{2}}\Big|^{2}+|k, \eta, l|^2 \Big|\widehat{W_{\mathrm{in}}^{2}}\Big|^{2} \right) d \eta \nonumber\\
& \lesssim e^{-\frac{1}{12} \nu t^{3}} \left( \left\| u_{\mathrm{in}}^2 \right\|_{H^2}^2+ \left\| w_{\mathrm{in}}^2 \right\|_{H^1}^2\right).
\end{align}

Finally, to estimate the $L^2$-norm of $u_{\neq}^3$, we use \eqref{3.20}$_{2}$ to obtain
\begin{align}
\widehat{U_{\neq}^3} & =-|k, l|^{-2} \, l  \, (\eta-k t) \widehat{U_{\neq}^2}+i k |k, l|^{-2} \widehat{W_{\neq}^2}.
\end{align}
Similar to \eqref{3.21}, one immediately gets
\begin{align}\label{3.41}
\left\| u_{\neq}^3(t) \right\|_{L^2}^2=\left\| U_{\neq}^3(t) \right\|_{L^2}^2 &\leqslant  \sum_{(k, l) \in \mathbb{Z}^2} \left(\int_{\mathbb{R}} |k, l|^{-4} |l|^{2}  |\eta-kt|^2 \Big|\widehat{U_{\neq}^2} (t)\Big|^{2} d\eta+ \int_{\mathbb{R}}|k, l|^{-4} |k|^{2} \Big|\widehat{W_{\neq}^2}(t)\Big|^{2} d\eta \right) \nonumber \\
&\lesssim \sum_{(k, l) \in \mathbb{Z}^2} \int_{\mathbb{R}}\left(   e^{-\frac{1}{12} \nu t^{3}}   |k, l|^{-4} |l|^{2} |\eta-kt|^2 p^{-1}\left(\Big|\widehat{Q_{\mathrm{in}}^{2}} \Big|^2+\Big|K_{\mathrm{in}}^2 \Big|^2\right)+\Big|\widehat{W_{\neq}^2}(t)\Big|^{2} \right)d\eta \nonumber \\
& \lesssim e^{-\frac{1}{12} \nu t^{3}} \left( \left\| u_{\mathrm{in}}^2 \right\|_{H^2}^2+ \left\| w_{\mathrm{in}}^2 \right\|_{H^1}^2\right).
\end{align}
Hence, together \eqref{3.37}, \eqref{3.21} with \eqref{3.41} gives \eqref{3.31} and \eqref{3.32} and the proof of Proposition \ref{pro3.2} is finished.  
\end{proof}


\subsection{Cancellation of lift-up effect on $u_{0}$}\label{section3.2}
Recall the lift-up effect for the 3D homogeneous Navier-Stokes equations. Setting $k=0$ and integrating the velocity equations in $x$ yields the system 
\begin{equation}\label{HSX1}
\partial_{t}u_{0}+\begin{pmatrix}
u_{0}^{2}\\
0\\
0
\end{pmatrix}
=\nu \Delta u_{0}.
\end{equation}
The explicit solution of (\ref{HSX1}) can be written as
\begin{align}\label{1.14}
\begin{cases}
u_{0}^1(t, y, z)=e^{\nu \Delta_{y, z}t}{\left( u_{0 \mathrm{in}}^1-t u_{0 \mathrm{in}}^2 \right)},\\
u_{0}^2(t, y, z)=e^{\nu \Delta_{y, z}t}{u_{0 \mathrm{in}}^2},\\
u_{0}^3(t, y, z)=e^{\nu \Delta_{y, z}t}{u_{0 \mathrm{in}}^3}.
\end{cases}
\end{align}
The linear growth  of $u_{0}^{1}$ predicted by the formula above for $t\lesssim \nu^{-1}$  is known as the lift-up effect. In general, the best global in time estimate for the Navier-Stokes equations that one can expect is 
\begin{equation*}
\|u_{0}^{1}\|_{L^{\infty}H^{s}}^{2}+\nu\|\nabla u_{0}^{1}\|_{L^{2}H^{s}}^{2}\lesssim\left(\varepsilon \nu^{-1}\right)^{2}.
\end{equation*}

In this section, we would like to emphasize that the rotation effect can inhibit the lift-up effect in (NSC) near the Couette flow.

Specifically, the zero frequency of the velocity field $u_0(t, y, z)=\overline{u}_0(t, y)+\widetilde{u}_{0}(t, y, z)$ satisfies the following system
\begin{equation}\label{1.7}
\begin{cases}
\partial_{t} u_0^1-\nu \Delta_{y, z} u_0^1+
(1-\beta)u_0^2=0, \\
\partial_{t} u_0^2-\nu \Delta_{y, z} u_0^2+
\beta u_0^1-\beta \partial_{y} \Delta_{y, z}^{-1} \partial_{y} u_0^1=0, \\
\partial_{t} u_0^3-\nu \Delta_{y, z} u_0^3-\beta \partial_{z} \Delta_{y, z}^{-1} \partial_{y} u_0^1=0, \\
u(t=0)=u_{\mathrm{in}},
\end{cases}
\end{equation}

For the linearized system \eqref{1.7}, we express the solution in an explicit form via Fourier transform and Duhamel?s principle. One gives the following proposition.
\begin{pro}\label{pro3.3}
	$(1)$ Assume that $l=0$, then \eqref{1.7} degenerates into the following heat equations 
	\begin{equation}\label{3.2.2}
	\begin{cases}
	\partial_{t} \overline{u}_0^1-\nu \partial_{y}^2 \overline{u}_0^1=0, \\
	\partial_{t} \overline{u}_0^3-\nu \partial_{y}^2 \overline{u}_0^3=0, \\
	\overline{u}_{0}^{2}=0,\\
	\overline{u}(t=0)=\overline{u}_{\mathrm{in}}.
	\end{cases}
	\end{equation}
	In this case, the double zero frequency $\overline{u}_0$  of \eqref{3.2.2} is immediately obtained, which is \eqref{3.1.5} as in Theorem \ref{1.1.}.
	
	$(2)$ Assume that $l \neq 0$, the simple zero frequency $\widetilde{u}_{0}$ of system \eqref{1.7} can be expressed as the following form
	\begin{equation}\label{3.2.4}
	\begin{cases}
	\widehat{\widetilde{u}_{0}^{1}}(t)=e^{-\nu(\eta^2+l^2) t} \cos (h t) \, \widehat{\widetilde{u}_{0 \mathrm{in}}^{1}}+\frac{(\beta-1)}{h} e^{-\nu(\eta^2+l^2) t} \sin (ht) \, \widehat{\widetilde{u}_{0 \mathrm{in}}^{2}}, \\
	\widehat{\widetilde{u}_{0}^{2}}(t)=e^{-\nu(\eta^2+l^2) t} \cos (h t) \, \widehat{\widetilde{u}_{0 \mathrm{in}}^{2}}-\frac{h}{(\beta-1)} e^{-\nu(\eta^2+l^2) t} \sin (h t)  \, \widehat{\widetilde{u}_{0 \mathrm{in}}^{1}}, \\
	\widehat{\widetilde{u}_{0}^{3}}(t)=e^{-\nu\left(\eta^{2}+l^{2} \right)t} \widehat{\widetilde{u}_{0 \mathrm{in}}^{3}}+e^{-\nu\left(\eta^{2}+l^2 \right) t} \sqrt{\frac{\beta}{(\beta-1)}} \, \frac{|\eta|}{|\eta, l|}  \, \sin (h t)  \,\widehat{\widetilde{u}_{0 \mathrm{in}}^{1}} \\
	\qquad\quad-e^{-\nu\left(\eta^{2}+l^2 \right) t} \, \frac{\eta}{l} \, (\cos(ht)-1)  \, \widehat{\widetilde{u}_{0 \mathrm{in}}^{2}},
	\end{cases}
	\end{equation}
	where $h(\beta, \eta, l)\overset{\Delta}{=}\sqrt{\beta (\beta-1)} \frac{|l|}{|\eta, l|}$.
\end{pro}
\begin{rem}
In fact, $u_{0}=\overline{u}_{0}+\widetilde{u}_{0}$ is the solution of the linearized problem \eqref{1.7}, where $\overline{u}_{0}$ and $\widetilde{u}_{0}$ are given by \eqref {3.2.2} and \eqref{3.2.4}, respectively. Compared with the three-dimensional classical Navier-Stokes equations, (NSC) has no lift-up effect. This indicates that the rotation promotes the stability of fluid near Couette flow.

\end{rem}
\begin{proof}
	Here, we only need to give a proof of the case when $l\neq 0$. Based on the observation of \eqref{1.7}, we find that both $\widetilde{u}_{0}^{1}$ and $\widetilde{u}_{0}^{2}$ equations are decoupled from $\widetilde{u}_{0}^{3}$. To simplify the calculation, we first take the Fourier transform to the $\widetilde{u}_{0}^{1}$ and $\widetilde{u}_{0}^{2}$ equations in \eqref{1.7}, which yields
	\begin{equation}
	\partial_{t} \begin{bmatrix}
	\widehat{\widetilde{u}_{0}^{1}} \\
	\widehat{\widetilde{u}_{0}^{2}} 
	\end{bmatrix}= \mathcal{A}\begin{bmatrix}
	\widehat{\widetilde{u}_{0}^{1}} \\
	\widehat{\widetilde{u}_{0}^{2}} 
	\end{bmatrix},\nonumber
	\end{equation}
	where $\mathcal{A}$ represents the multiplier matrix of the linear operators
	\begin{equation}\label{3..44}
    \mathcal{A}\overset{\Delta}{=}\begin{bmatrix}
	-\nu (\eta^2+l^2)  & \beta-1\\
	-\beta \frac{l^2}{\eta^2+l^2} & -\nu (\eta^2+l^2)
	\end{bmatrix}.
	\end{equation}
	We compute the fundamental matrix $e^{\mathcal{A} t}$ explicitly. The characteristic polynomial associated
	with $\mathcal{A}$ is given by
	\begin{equation*}
	\left[ \lambda+ \nu \left( \eta^2+l^2 \right)\right]^2+\beta \left( \beta-1\right) \frac{l^2}{\eta^2+l^2}=0,
	\end{equation*}
	and thus the spectra of $\mathcal{A}$ as follows
	\begin{equation*}
	\lambda_1=-\nu |\eta, l|^2+i \sqrt{\beta (\beta-1)} \frac{|l|}{|\eta, l|}, \quad  \quad \lambda_2=-\nu |\eta, l|^2-i \sqrt{\beta (\beta-1)} \frac{|l|}{|\eta, l|}.
	\end{equation*}
	Note that the eigenvalues $\lambda_1$ and $\lambda_2$ are a pair of conjugate complex roots. The eigenvectors $\mathcal{V}_{i}$ $(i=1, 2)$ associated with  $\lambda_{i}$  satisfy
	\begin{align*}
	\left( \lambda_{i} \textbf{I} - \mathcal{A} \right) \mathcal{V}_{i}=\begin{bmatrix}
	\lambda_{i}+\nu (\eta^2+l^2)  & 1-\beta\\
	\beta \frac{l^2}{\eta^2+l^2} & \lambda_{i}+\nu (\eta^2+l^2)
	\end{bmatrix} \begin{bmatrix}
	\mathcal{V}_{i 1}\\
	\mathcal{V}_{i 2}
	\end{bmatrix}=0,
	\end{align*}
    which gives
	\begin{equation*}
	\mathcal{V}_{1}=\begin{bmatrix}
	\beta-1 \\
	i \sqrt{\beta (\beta-1)} \frac{|l|}{|\eta, l|}
	\end{bmatrix}, \quad \quad
	\mathcal{V}_{2}=\begin{bmatrix}
	1-\beta \\
	i \sqrt{\beta (\beta-1)} \frac{|l|}{|\eta, l|}
	\end{bmatrix}.
	\end{equation*}
	To simplify the expression, we sometimes use $h(\beta, \eta, l)$ to represent $\sqrt{\beta (\beta-1)} \frac{|l|}{|\eta, l|}$.
	Let $\widehat{\widetilde{u}_{0}^{1, 2}}(t)= c_1 e^{\lambda_1 t} \mathcal{V}_{1}+c_2 e^{\lambda_2 t} \mathcal{V}_{2}$ with the initial data $\widehat{\widetilde{u}_{0 \mathrm{in}}^{1}}$ and $\widehat{\widetilde{u}_{0 \mathrm{in}}^{2}}$, then $c_1$ and $c_2$ satisfy
	\begin{equation}
	c_{1}=\frac{i h \widehat{\widetilde{u}_{0 \mathrm{in}}^{1}}+(\beta-1) \widehat{\widetilde{u}_{0 \mathrm{in}}^{2}}}{2 i h(\beta-1)}, \quad \quad 
	c_{2}=\frac{(\beta-1) \widehat{\widetilde{u}_{0 \mathrm{in}}^{2}}-i h \widehat{\widetilde{u}_{0 \mathrm{in}}^{1}}}{2 i h(\beta-1)}.
	\end{equation}
	As a consequence,  when  $l\neq 0$, the solution $\widetilde{u}_{0}^{1}$ and $\widetilde{u}_{0}^{2}$ is given by
	\begin{equation}
	\begin{cases}\label{3.44}
	\widehat{\widetilde{u}_{0}^{1}}(t)=\frac{1}{2}\left(e^{-\nu(\eta^2+l^2) t+i h t} +e^{-\nu(\eta^2+l^2) t-i h t}\right) \widehat{\widetilde{u}_{0 \mathrm{in}}^{1}} \\
	\quad \quad \quad +\frac{i(\beta-1)}{2 h}\left(e^{-\nu(\eta^2+l^2) t-i h t}-e^{-\nu(\eta^2+l^2) t+i h t} \right) \widehat{\widetilde{u}_{0 \mathrm{in}}^{2}}, \\
	\widehat{\widetilde{u}_{0}^{2}}(t)=\frac{1}{2}\left(e^{-\nu(\eta^2+l^2) t+i h t}+e^{-\nu(\eta^2+l^2) t-i h t}\right)  \widehat{\widetilde{u}_{0 \mathrm{in}}^{2}}\\
	\quad \quad \quad+\frac{i h}{2(\beta-1)}\left(e^{-\nu(\eta^2+l^2) t+i h t}-e^{-\nu(\eta^2+l^2) t-i h t}\right) \widehat{\widetilde{u}_{0 \mathrm{in}}^{1}}.
	\end{cases}
	\end{equation}
	Due to $e^{i h t}+e^{-i h t}=2 \cos (h t)$ and $e^{i h t}-e^{-i h t}=2 i \sin (h t)$, \eqref{3.44} is rewritten as
	\begin{equation}\label{3.45}
	\begin{cases}
	\widehat{\widetilde{u}_{0}^{1}}(t)=e^{-\nu(\eta^2+l^2) t} \cos (h t) \, \widehat{\widetilde{u}_{0 \mathrm{in}}^{1}}+\frac{(\beta-1)}{h} e^{-\nu(\eta^2+l^2) t} \sin (ht) \, \widehat{\widetilde{u}_{0 \mathrm{in}}^{2}}, \\
	\widehat{\widetilde{u}_{0}^{2}}(t)=e^{-\nu(\eta^2+l^2) t} \cos (h t) \, \widehat{\widetilde{u}_{0 \mathrm{in}}^{2}}-\frac{h}{(\beta-1)} e^{-\nu(\eta^2+l^2) t} \sin (h t)  \, \widehat{\widetilde{u}_{0 \mathrm{in}}^{1}}.
	\end{cases}
	\end{equation}
	
By using the incompressible condition, $\widetilde{u}_{0}^{3}$   can be expressed as $-\partial_{z}^{-1} \partial_{y} \widetilde{u}_{0}^{2}(t, y, z)$. In addition, we can also use the Duhamel principle and \eqref{3.45} to represent it more explicitly
	\begin{align}\label{3.46}
	\widehat{\widetilde{u}_{0 }^{3}}(t)=&e^{-\nu(\eta^{2}+l^{2}) t} \widehat{\widetilde{u}_{0 \mathrm{in}}^{3}}+\int_{0}^{t} e^{-\nu(\eta^{2}+l^{2})(t-\tau)}  \frac{\beta \eta l}{\eta^{2}+l^{2}} \widehat{\widetilde{u}_{0}^{1}}(\tau) d \tau  \nonumber\\
	=&e^{-\nu(\eta^{2}+l^{2}) t} \widehat{\widetilde{u}_{0 \mathrm{in}}^{3}}+\int_{0}^{t} e^{-\nu(\eta^{2}+l^{2})(t-\tau)}  \frac{\beta\eta l}{\eta^{2}+l^{2}}\left[e^{-\nu\left(\eta^{2}+l^{2}\right) \tau} \cos (h \tau) \widehat{\widetilde{u}_{0 \mathrm{in}}^{1}}\right.  \nonumber\\
	&\left.+\frac{(\beta-1)}{h} e^{-\nu(\eta^2+l^2) \tau} \sin (h\tau) \, \widehat{\widetilde{u}_{0 \mathrm{in}}^{2}} \right] d \tau  \nonumber\\
	=&e^{-\nu\left(\eta^{2}+l^{2} \right)t} \widehat{\widetilde{u}_{0 \mathrm{in}}^{3}}+e^{-\nu\left(\eta^{2}+l^2 \right) t} \sqrt{\frac{\beta}{(\beta-1)}} \, \frac{|\eta|}{|\eta, l|}  \, \sin (h t)  \,\widehat{\widetilde{u}_{0 \mathrm{in}}^{1}} \nonumber\\
	&-e^{-\nu\left(\eta^{2}+l^2 \right) t} \, \frac{\eta}{l} \, (\cos(ht)-1)  \, \widehat{\widetilde{u}_{0 \mathrm{in}}^{2}}.
	\end{align}
	Finally,  it follows from inverse Fourier transform, \eqref{3.45} and \eqref{3.46} that \eqref{1.13} holds.
\end{proof}

Contrast to the classical Navier-Stokes equations, by means of \eqref{3.45} and \eqref{3.46} we find that although rotation brings about fluid oscillation, it cancels the lift-up effect. In other words, in the zero frequency part, the oscillation effect caused by the rotation is significantly stronger than the lift-up effect caused by the shear flow $(y, 0, 0)$. Thus, for (NSC) we can expect the zero frequency of the velocity field $u_0$ to have the following best
global in time estimate
\begin{equation}\label{3.521}
\|u_{0} \|_{L^{\infty}H^{s}}^{2}+\nu\| \nabla u_{0} \|_{L^{2}H^{s}}^{2}\lesssim \varepsilon^{2}.
\end{equation}

\subsection{The dispersive estimates on $\widetilde{u}_{0}$}\label{section3.3}
In section \ref{section3.2}, we give the  expression for the zero frequency of the velocity field. In the process of its proof, we notice that the imaginary part of the spectrum of the multiplier matrix is non-zero, which implies that it may have some dispersion mechanism. In this section, we aim to further prove that the rotation effect  leads to dispersion estimates for the simple zero frequency velocity $\widetilde{u}_{0}$.
 
Dispersive phenomena often play a crucial role in the study of evolution partial differential equations. Mathematically, exhibiting dispersion often amounts to proving a decay estimate for the $L^{\infty}$-norm of the solution at time $t$ in terms of negative power of $t$ and of the $L^1$-norm of the initial data.
The basic idea of proving these dispersive estimates relies on the Van der Corput lemma and on a explicit representation of the solution.   Thus, we gives the dispersive estimates on $\widetilde{u}_{0}^{i}$ $(i=1, 2, 3)$ in the following proposition.

\begin{pro}\label{pro3.4}
Assume that $\nu > 0$   and $\beta \in \mathbb{R}$ with $|\beta| \geqslant 2$.  For any  $\widetilde{u}_{0\mathrm{in}} \in \mathcal{S}(\mathbb{R} \times \mathbb{T})$ 
and  $a \geqslant 0$,   there holds that
\begin{align}
\left\|\left|\partial_{z}\right|^{a}\left|\nabla_{y, z}\right|^{-a} e^{\mathcal{L}_{+} t} \widetilde{u}_{0 \mathrm{in}}^{2} \right\|_{L^{\infty}(\mathbb{R} \times \mathbb{T})} \lesssim & e^{-\nu t}\left(  \sqrt{\beta (\beta-1)} \, t\right)^{-\frac{1}{3}} \|\widetilde{u}_{0 \mathrm{in}}^{2}\|_{W^{4,1}(\mathbb{R} \times \mathbb{T})}, \label{3.49} \\
\left\|\left|\partial_{z}\right|^{a}\left|\nabla_{y, z}\right|^{-a} e^{\mathcal{L}_{+} t} \widetilde{u}_{0 \mathrm{in}}^{1} \right\|_{L^{\infty}(\mathbb{R} \times \mathbb{T})} \lesssim & e^{-\nu t}\left(  \sqrt{\beta (\beta-1)} \, t\right)^{-\frac{1}{3}} \|\widetilde{u}_{0 \mathrm{in}}^{1}\|_{W^{4,1}(\mathbb{R} \times \mathbb{T})}, \label{3.51} \\
\left\|\left|\partial_{z}\right|^{-a}\left|\nabla_{y, z}\right|^{a} e^{\mathcal{L}_{+} t} \widetilde{u}_{0 \mathrm{in}}^{2} \right\|_{L^{\infty}(\mathbb{R} \times \mathbb{T})} \lesssim & e^{-\nu t}\left(  \sqrt{\beta (\beta-1)} \, t\right)^{-\frac{1}{3}} \|\widetilde{u}_{0 \mathrm{in}}^{2}\|_{W^{4+a,1}(\mathbb{R} \times \mathbb{T})}. \label{3.51.3}
\end{align}
Furthermore, if we replace $\mathcal{L}_{+}$ with $\mathcal{L}_{-}$, the estimates \eqref{3.49}--\eqref{3.51.3} are also valid.

\end{pro}

\begin{proof}
	Without loss of generality, here we only estimate the $\mathcal{L}_{+}$ operator, and $\mathcal{L}_{-}$ is similar. Motivated by  Proposition 3.1 in \cite{ZZW2024}, the definition  of $\mathcal{L}_{+}$ as in \eqref{3.1.4} and Lemma \ref{lem2.1}, 	for any $\mu>0$ we have
	\begin{align}\label{3.52}
	|l|^a \left\| (l^2-\partial_{y}^2)^{-\frac{a}{2}} e^{(\mathcal{L}_{+}-\nu \Delta_{y, z}) t}  \widetilde{u}_{0 \mathrm{in}}^{2}  \right\|_{L^{\infty}(\mathbb{R})} \leqslant C|l| \left(\sqrt{\beta (\beta-1)} \, t\right)^{-\frac{1}{3}} \left\| \widetilde{u}_{0 \mathrm{in}}^{2} \right\|_{W^{\frac{3}{2}+\mu, 1}(\mathbb{R})}.
	\end{align}
 Hence,  using \eqref{3.52} yields that
	\begin{align}\label{3.53}
	\left| \left|\partial_{z}\right|^{a}\left|\nabla_{y, z}\right|^{-a} e^{\mathcal{L}_{+} t} \widetilde{u}_{0 \mathrm{in}}^{2}(y, z) \right|= & \left| \sum_{l \neq 0} e^{2 \pi i z l-\nu l^{2} t} \int_{\eta} e^{2\pi i y \eta-\nu t \eta^{2}+i t \sqrt{\beta (\beta-1)} \frac{|l|}{|\eta, l|}}|l|^{a}\left(l^{2}+\eta^{2}\right)^{-\frac{a}{2}} \widehat{\widetilde{u}_{0 \mathrm{in}}^{2}} (\eta, l) d \eta \right| \nonumber\\
	\leqslant&  \sum_{l \neq 0} e^{-\nu t}|l|^{a}\left|\int_{\eta} e^{2 \pi i y \eta+i t \sqrt{\beta(\beta-1)} \frac{|l|}{|\eta, l|}} e^{-\nu t \eta^{2}}\left(l^{2}+\eta^{2}\right)^{-\frac{a}{2}} \widehat{\widetilde{u}_{0 \mathrm{in}}^{2}} (\eta, l) d \eta\right|  \nonumber\\
	\leqslant & e^{-\nu t} \sum_{l \neq 0}|l|^{a}\left\| e^{(\mathcal{L}_{+}-\nu \Delta_{y, z}) t}\left(e^{\nu t \partial_{y}^{2}}\left(l^{2}-\partial_{y}^{2}\right)^{-\frac{a}{2}} \widetilde{u}_{0 \mathrm{in} l}^{2}(y) \right)\right\|_{L^{\infty}(\mathbb{R})}  \nonumber\\
	\lesssim & e^{-\nu t}(\sqrt{\beta (\beta-1)} \,t )^{-\frac{1}{3}} \sum_{l \neq 0}|l|\left\|e^{\nu  \partial_{y}^{2} t} \widetilde{u}_{0 \mathrm{in} l}^{2}(y)\right\|_{W^{\frac{3}{2}+\mu, 1}(\mathbb{R})}  \nonumber\\
	\lesssim &e^{-\nu t}(\sqrt{\beta (\beta-1)} \,t )^{-\frac{1}{3}} \sum_{l \neq 0}|l|^{-1-\mu}|l|^{2+\mu}\left\|{\widetilde{u}_{0 \mathrm{in} l}^{2}(y)}\right\|_{W^{\frac{3}{2}+\mu, 1}(\mathbb{R})}  \nonumber\\
	\lesssim &e^{-\nu t}(\sqrt{\beta (\beta-1)} \,t )^{-\frac{1}{3}} \sum_{l \neq 0}|l|^{-1-\mu}\|\widetilde{u}_{0 \mathrm{in}}^{2}\|_{W_{y}^{\frac{3}{2}+\mu, 1} W_{z}^{2+\mu, 1}(\mathbb{R} \times \mathbb{T})} . 
	\end{align}
	
    Similarly for $\widetilde{u}_{0 \mathrm{in}}^{1}$, it also holds
	\begin{equation}\label{3.54}
	\left| \left|\partial_{z}\right|^{a}\left|\nabla_{y, z}\right|^{-a} e^{\mathcal{L}_{+} t} \widetilde{u}_{0 \mathrm{in}}^{1}(y, z) \right| \lesssim e^{-\nu t}(\sqrt{\beta (\beta-1)} \,t )^{-\frac{1}{3}} \sum_{l \neq 0}|l|^{-1-\mu}\|\widetilde{u}_{0 \mathrm{in}}^{1}\|_{W_{y}^{\frac{3}{2}+\mu, 1} W_{z}^{2+\mu, 1}(\mathbb{R} \times \mathbb{T})}.
	\end{equation}
	Next, deducing from Proposition 3.1 in \cite{ZZW2024}, we have the following fact
	\begin{align}
	\left\| |l|^{-a} \left( l^2-\partial_{y}^2  \right)^{\frac{a}{2}}  e^{(\mathcal{L}_{+}-\nu \Delta_{y, z}) t}  \widetilde{u}_{0 \mathrm{in}}^{2}  \right\|_{L^{\infty}(\mathbb{R})} \lesssim |l|(\sqrt{\beta (\beta-1)} \, t)^{-\frac{1}{3}} \left\| \widetilde{u}_{0 \mathrm{in}}^{2} \right\|_{W^{\frac{3}{2}+a+\mu, 1}(\mathbb{R})}.
	\end{align}
Hence, we get
	\begin{align}
	\left| \left|\partial_{z}\right|^{-a}\left|\nabla_{y, z}\right|^{a} e^{\mathcal{L}_{+} t} \widetilde{u}_{0 \mathrm{in}}^{2}(y, z) \right|= & \left| \sum_{l \neq 0} e^{2\pi i z l-\nu l^{2} t} \int_{\eta} e^{2\pi i y \eta-\nu t \eta^{2}+i t \sqrt{\beta (\beta-1)} \frac{|l|}{|\eta, l|}}|l|^{-a}\left(l^{2}+\eta^{2}\right)^{\frac{a}{2}} \widehat{\widetilde{u}_{0 \mathrm{in}}^{2}}(\eta, l) d \eta \right| \nonumber\\
	\leqslant&  \sum_{l \neq 0} e^{-\nu t}|l|^{-a}\left|\int_{\eta} e^{2\pi i y \eta+i t \sqrt{\beta(\beta-1)} \frac{|l|}{|\eta, l|}} e^{-\nu t \eta^{2}}\left(l^{2}+\eta^{2}\right)^{\frac{a}{2}} \widehat{\widetilde{u}_{0 \mathrm{in}}^{2}}(\eta, l) d \eta\right|  \nonumber\\
	\leqslant & e^{-\nu t} \sum_{l \neq 0}|l|^{-a}\left\| e^{(\mathcal{L}_{+}-\nu \Delta_{y, z}) t}\left(e^{\nu t \partial_{y}^{2}}\left(l^{2}-\partial_{y}^{2}\right)^{\frac{a}{2}} \widetilde{u}_{0 \mathrm{in} l}^{2} (y) \right)\right\|_{L^{\infty}(\mathbb{R})}  \nonumber\\
	\lesssim &e^{-\nu t}(\sqrt{\beta (\beta-1)} \,t )^{-\frac{1}{3}} \sum_{l \neq 0}|l|^{-1-\mu}|l|^{2+\mu}\left\|{\widetilde{u}_{0 \mathrm{in} l}^{2}}(y)\right\|_{W^{\frac{3}{2}+a+\mu, 1}(\mathbb{R})}  \nonumber\\
	\lesssim &e^{-\nu t}(\sqrt{\beta (\beta-1)} \,t )^{-\frac{1}{3}} \sum_{l \neq 0}|l|^{-1-\mu}\|\widetilde{u}_{0 \mathrm{in}}^{2}\|_{W_{y}^{\frac{3}{2}+a+\mu, 1} W_{z}^{2+\mu, 1}(\mathbb{R} \times \mathbb{T})}.
	\end{align}
Therefore, we finish the proof of Proposition \ref{pro3.4}.
\end{proof}

In the following, we complete the proof of  Theorem \ref{1.1.} and give the dispersion estimate \eqref{3.6}--\eqref{3.6.2}.  Using the explicit solution \eqref{1.13} and choose $a=0$ and $a=1$ in \eqref{3.49} and \eqref{3.51} respectively, we get
\begin{align}\label{3.55}
  \left|  \widetilde{u}_{0}^{2}(t, y, z)  \right| \lesssim & \left| e^{\mathcal{L}_{+} t}  \widetilde{u}_{0 \mathrm{in}}^{2} \right|+\left| e^{\mathcal{L}_{-} t} \widetilde{u}_{0 \mathrm{in}}^{2} \right|+ \left| |\nabla_{y, z}|^{-1} |\partial_{z}|  e^{\mathcal{L}_{+} t} \widetilde{u}_{0 \mathrm{in}}^{1} \right| +\left| |\nabla_{y, z}|^{-1} |\partial_{z}|  e^{\mathcal{L}_{-} t} \widetilde{u}_{0 \mathrm{in}}^{1} \right| \nonumber \\
  \lesssim & e^{-\nu t}(\sqrt{\beta (\beta-1)} \,t )^{-\frac{1}{3}} \sum_{l \neq 0}|l|^{-1-\mu} \| \left(\widetilde{u}_{0 \mathrm{in}}^{1}, \widetilde{u}_{0 \mathrm{in}}^{2} \right)\|_{W_{y}^{\frac{3}{2}+\mu, 1} W_{z}^{2+\mu, 1}(\mathbb{R} \times \mathbb{T})}.
\end{align}
For $\widetilde{u}_{0}^{1}(t, y, z)$, by taking $a=1$ in \eqref{3.51.3}, it yields
\begin{align}\label{3.56}
	\left| \widetilde{u}_{0}^{1}(t, y, z) \right|  &\lesssim \left| e^{\mathcal{L}_{+} t}  \widetilde{u}_{0 \mathrm{in}}^{1} \right|+ \left| e^{\mathcal{L}_{-} t}  \widetilde{u}_{0 \mathrm{in}}^{1} \right| + \left|  |\partial_{z}|^{-1} |\nabla_{y, z}| e^{\mathcal{L}_{-} t}  \widetilde{u}_{0 \mathrm{in}}^{2} \right|+ \left|  |\partial_{z}|^{-1} |\nabla_{y, z}| e^{\mathcal{L}_{+} t}  \widetilde{u}_{0 \mathrm{in}}^{2} \right| \nonumber \\
	 & \lesssim e^{-\nu t}\left(  \sqrt{\beta (\beta-1)} \, t\right)^{-\frac{1}{3}}  \sum_{l \neq 0}|l|^{-1-\mu} \Big( \|\widetilde{u}_{0 \mathrm{in}}^{1}\|_{W_{y}^{\frac{3}{2}+\mu, 1} W_{z}^{2+\mu, 1}(\mathbb{R} \times \mathbb{T})} \nonumber \\&\quad+\|\widetilde{u}_{0 \mathrm{in}}^{2}\|_{W_{y}^{\frac{5}{2}+\mu, 1} W_{z}^{2+\mu, 1}(\mathbb{R} \times \mathbb{T})} \Big).
\end{align}

Finally, combining \eqref{3.55} and the incompressible condition $\widetilde{u}_{0}^{3}(t, y, z)=-\partial_{z}^{-1} \partial_{y} \widetilde{u}_{0}^{2}(t, y, z)$ with $l \neq 0$, we obtain immediately
\begin{align}\label{3.57}
	\left|  \widetilde{u}_{0}^{3}(t, y, z)  \right| = & \left| \partial_{z}^{-1} \partial_{y} \widetilde{u}_{0}^{2}(t, y, z) \right| \nonumber \\
	\lesssim & e^{-\nu t}(\sqrt{\beta (\beta-1)} \,t )^{-\frac{1}{3}} \sum_{l \neq 0}|l|^{-2-\mu} \| \left(\widetilde{u}_{0 \mathrm{in}}^{1}, \widetilde{u}_{0 \mathrm{in}}^{2} \right)\|_{W_{y}^{\frac{5}{2}+\mu, 1} W_{z}^{2+\mu, 1}(\mathbb{R} \times \mathbb{T})}.
\end{align}
Thus, \eqref{3.55}--\eqref{3.57} implies the dispersion estimates \eqref{3.6}--\eqref{3.6.2} of $\widetilde{u}_{0}(t, y, z)$.

Here, we want to point out that we do not take full advantage of the dispersive property of $\widetilde{u}_{0}$ in the handle of  the nonlinear stability. This is because  we involve the interaction between non-zero frequency velocity fields $(\neq \cdot \neq)$ in the calculation process. Though we construct good unknowns $Q_{\neq}^2$ and $K_{\neq}^2$   in Fourier space , which enhances the dissipation of the non-zero frequency velocity field $U_{\neq}$. However, under the framework of bootstrap hypothesis,  we can only  get  the stability threshold $\gamma$ is 1. Of course, we will also explore new methods to make the threshold $\gamma$  less than 1, which is left for future research. 

\section{The nonlinear stability }\label{section4}
In the following, we will introduce the ideas of the  proof of nonlinear stability of (\ref{1.3}). To begin with, applying the energy methods in \cite{LO97} and \cite{MB02}, we have the  existence of the local solution, which  is further extended to the global one. 

\begin{lem}[Local existence]\label{lem4.3}
	Under the condition of Theorem \ref{1..1}, let $u_{\mathrm{in}}$ be divergence-free and satisfy \eqref{1..11-2}. Then there exists a $T^{\ast}>0$ independent of $\nu$ and $\beta$ such that there is a unique strong solution $u(t)\in C([0,T^{\ast}];H^{\sigma})$, which satisfies the initial data and is real analytic for $t\in (0,T^{\ast}]$. Moreover, there exists a maximal existence time  $T_{0}$  with  $T^{\ast}<T_{0} \leqslant \infty$  such that the solution  $u(t)$  remains unique and real analytic on  $\left(0, T_{0}\right)$  and, if for some  $\tau \leqslant T_{0}$  we have $ \lim \sup _{t \rightarrow \tau}\|u(t)\|_{H^{\sigma-2}}<\infty $, then  $\tau<T_{0}$.
\end{lem}
%
%

We recall the  new variable \eqref{3.13} which plays a crucial role in the analysis of the linear stability. Here making full use of the  variable \eqref{3.13} and \eqref{2.10}, we obtain that
\begin{align}
&\partial_{t} Q^{2}-\nu \Delta_{L} Q^{2}+i \sqrt{\beta (\beta-1)} \partial_{Z} |\nabla_{L}|^{-1} \check{K}^{2} \nonumber\\
&\quad  =-Q\cdot \nabla_{L}U^{2}- U \cdot\nabla_{L}Q^{2}-2\partial_{i}^{L}U^{j}\partial_{ij}^{L}U^{2}+\partial_{Y}^{L}\left(\partial_{i}^{L}U^{j}\partial_{j}^{L}U^{i} \right), \label{HSX88} \\
&\partial_{t} \check{K}^2-\nu \Delta_{L} \check{K}^2+\partial_{XY}^{L}\Delta_{L}^{-1}\check{K}^{2}-i\sqrt{\beta(\beta-1)}\partial_{Z}|\nabla_{L}|^{-1}Q^{2}  \nonumber\\
&\quad  =i\sqrt{\frac{\beta}{\beta-1}}\partial_{Z}|\nabla_{L}|(U\cdot\nabla_{L}U^{1})-i\sqrt{\frac{\beta}{\beta-1}}\partial_{X}|\nabla_{L}| (U\cdot \nabla_{L}U^{3}). \label{HSX99} 
\end{align}
Thus, we easily deduce that $
\widehat{U^{1}}$ and $\widehat{U^{2}}$ satisfy 
\begin{align}\label{3..4}
\widehat{U^{1}}&=-|k,l|^{-2}k(\eta-kt)\widehat{U^{2}}-i |k,l|^{-2}l \,\widehat{W^{2}}\nonumber\\&=|k,l|^{-2}\frac{k(\eta-kt)}{k^{2}+(\eta-kt)^{2}+l^{2}}\widehat{Q^{2}}-\sqrt{\frac{\beta-1}{\beta}}|k,l|^{-2}\frac{l}{|k,\eta-kt,l|}K^{2}
\end{align}
and 
\begin{align}\label{3.5}
\widehat{U^{3}}&=- |k,l|^{-2} \,l(\eta-kt)\widehat{U^{2}}+i|k,l|^{-2}k \, \widehat{W^{2}}\nonumber\\&=|k,l|^{-2}\frac{l(\eta-kt)}{k^{2}+(\eta-kt)^{2}+l^{2}}\widehat{Q^{2}}+\sqrt{\frac{\beta-1}{\beta}}|k,l|^{-2}\frac{k}{|k,\eta-kt,l|}K^{2}.
\end{align}

With the help of analysis above, we will show the following bootstrap result. In fact, Theorem \ref{1..1} is mainly  proved by using the following bootstrap argument.  For convenient, we set $N \overset{\Delta}{=}\sigma-2>\frac{5}{2}$. 

\section{Bootstrap argument}\label{4.2}
\begin{pro}[Bootstrap step]\label{pro4.1}
	Under the hypothesis of Theorem \ref{1..1}, assume that for some $T>0$, we have the following estimates hold in $[0,T]$.
	
	$(1)$ the bounds on $Q_{\neq}^2$ and  $\check{K}_{\neq}^{2}$:
	\begin{align}
		\left\| m M Q_{\neq}^2 \right\|_{L^{\infty} H^N}+ \nu^{\frac{1}{2}} \left\| m M \nabla_{L} Q_{\neq}^2 \right\|_{L^{2} H^N}+ \left\| \sqrt{-\dot{M}M} m Q_{\neq}^2  \right\|_{L^{2} H^N} & \leqslant 10\varepsilon, \label{4..8}\\
		\left\| m M \check{K}_{\neq}^2 \right\|_{L^{\infty} H^N}+ \nu^{\frac{1}{2}} \left\| m M \nabla_{L} \check{K}_{\neq}^2 \right\|_{L^{2} H^N}+ \left\| \sqrt{-\dot{M}M} m \check{K}_{\neq}^2  \right\|_{L^{2} H^N} & \leqslant 10\varepsilon.     \label{4..5} 
	\end{align}
	
	$(2)$ the bounds on $Q_{0}^{2}$ and $\check{K}_{0}^{2}$:
	\begin{align}
		\left\| Q_0^2 \right\|_{L^{\infty} H^N}+ \nu^{\frac{1}{2}} \left\| \nabla Q_0^2 \right\|_{L^{2} H^N} & \leqslant 10\varepsilon,     \label{4..4}\\
		\left\| \check{K}_0^2 \right\|_{L^{\infty} H^N}+ \nu^{\frac{1}{2}} \left\| \nabla \check{K}_0^2 \right\|_{L^{2} H^N} & \leqslant 10\varepsilon.    \label{4..7}
	\end{align}
	
	$(3)$ the bounds on $U_{0}$:  
	\begin{align}
		\left\| \widetilde{U}_0^1 \right\|_{L^{\infty} H^{N}}+ \nu^{\frac{1}{2}} \left\| \nabla \widetilde{U}_0^1 \right\|_{L^{2} H^{N}} & \leqslant 10\varepsilon,     \label{4..11}\\
		\left\| \widetilde{U}_0^2 \right\|_{L^{\infty} H^{N}}+ \nu^{\frac{1}{2}} \left\| \widetilde{U}_0^2 \right\|_{L^{2} H^{N}}+ \nu^{\frac{1}{2}} \left\|\nabla\widetilde{U}_{0}^{2} \right\|_{L^{2}H^{N}}   & \leqslant 10\varepsilon,     \label{4..12}\\
		\left\| \widetilde{U}_0^3 \right\|_{L^{\infty} H^{N}}+ \nu^{\frac{1}{2}} \left\| \nabla \widetilde{U}_0^3 \right\|_{L^{2} H^{N}} & \leqslant 10\varepsilon,     \label{4..13}\\
		\left\|   \overline{U}_0^1  \right\|_{L^{\infty} H^{N+1}} +\nu^{\frac{1}{2}}\left\|\partial_{y}\overline{U}_{0}^{1} \right\|_{L^{2}H^{N+1}}  & \leqslant 10\varepsilon,     \label{4..9}\\
		\left\|   \overline{U}_0^3  \right\|_{L^{\infty} H^{N+1}} +\nu^{\frac{1}{2}}\left\|\partial_{y}\overline{U}_{0}^{3} \right\|_{L^{2}H^{N+1}}  & \leqslant 10\varepsilon.  \label{4..10}
	\end{align}
	Assume that $  \varepsilon=\left\| u_{\mathrm{in}} \right\|_{H^{N+2}}  \leqslant \delta \nu$, $\nu \in (0, 1)$, the  estimates \eqref{4..8}--\eqref{4..10} hold on $[0,T]$. Then for $\delta$ sufficiently small depending only on $\sigma$,  but not on $T$, these same estimates hold with all the occurrences of 10 on the right-hand side replaced by 5.
\end{pro}
\begin{rem}\label{remark5}
	The proof of Theorem \ref{1..1} follows directly from Proposition \ref{pro4.1}: By standard local well-posedness results which are established in Lemma \ref{lem4.3}, we can assume that there exists $T^{\ast}>0$ such that the assumptions \eqref{4..8}--\eqref{4..10} hold on   $[0,T^{\ast}]$. Proposition \ref{pro4.1} and the continuity of these norms imply that the set of times on \eqref{4..8}--\eqref{4..10} holds is closed, open-empty in $[0,\infty)$. Moreover, combined with the existence of local solution, the global solution is unique and regular on the time $[0,\infty)$.
\end{rem}
\begin{rem}
Note that the nonzero frequency estimates of $Q^{2}$ and $\check{K}^2$ in \eqref{4..8} and \eqref{4..5} emphasize the enhanced dissipation effects of $U_{\neq}^{i}~ (i=1,2,3)$. From the previous zero-frequency estimation of the velocity field in linear stability analysis, we know that rotation cancels the lift-up effect which  appeared in the 3D classical Navier-Stokes equations. Indeed, the zero frequency $u_{0}$ as in \eqref{4..11}--\eqref{4..10} in the bootstrap hypothesis are consistent with the bound \eqref{3.521} we excepted.
\end{rem}
\begin{rem}
We recall that the initial value $u_{\mathrm{in}}$ in Theorem \ref{1..1} satisfies $\left\| u_{\mathrm{in}} \right\|_{H^{\sigma}} \leqslant \delta \nu$,
which implies the stability threshold index $\gamma=1$.
In contrast to the result with $\gamma=\frac{3}{2}$ of Bedrossian, Germain and Masmoudi \cite{MR3612004}, the threshold index in our paper is smaller due to rotation. Unfortunately, we can not take full advantage of  dispersion estimates of zero-frequency to decrease the threshold index. This is due to our concern on the closed estimate \eqref{4..8} in subsection \ref{sec6.1} when we need to deal with the term $(\neq,\neq)$ such as 
\begin{align}
	\mathcal{NLP}Q_{\neq}^{2}(U_{\neq}^{1}, U_{\neq}^{2})&=\int_{0}^{T}\int mM  \left\langle D \right\rangle^{N} Q_{\neq}^{2}mM  \left\langle D \right\rangle^{N}\partial_{Y}^{L}\left(\partial_{Y}^{L}U_{\neq}^{1}\partial_{X}U_{\neq}^{2}\right)_{\neq}dVdt\nonumber\\
	&\leqslant \left\|mM\nabla_{L}Q_{\neq}^{2}\right\|_{L^{2}H^{N}}\left(\left\|mM\nabla_{L}\check{K}_{\neq}^{2}\right\|_{L^{2}H^{N}}+\left\|mM\nabla_{L}Q_{\neq}^{2} \right\|_{L^{2}H^{N}}\right)\nonumber\\&\quad\times
	\left\|mMQ_{\neq}^{2}\right\|_{L^{\infty}H^{N}}\nonumber\\
	&\lesssim \varepsilon\nu^{-\frac{1}{2}} \varepsilon \nu^{-\frac{1}{2}} \varepsilon\nonumber 
	\\
	&\lesssim  
	\varepsilon^{2}(\varepsilon\nu^{-1})\nonumber\\
	&\lesssim
	\varepsilon^{2}.\nonumber
\end{align}
In a certain sense, the threshold index $\gamma=1$ is regarded as   optimal under the bootstrap framework we built.
\end{rem}

\begin{proof}
	For a detailed proof of Proposition \ref{pro4.1}, see the following sections \ref{sec5}--\ref{sec7}.
\end{proof}

With Proposition 5.1 at hand, we have the following proposition. 
\begin{pro}\label{pro4.2}
	Under the bootstrap hypotheses, for $\varepsilon \nu^{-1}$ sufficiently small, the following estimates hold:
	
	$(1)$ the zero frequency velocity $U_0^{i}$ $(i=1, 2, 3)$:
	\begin{align}
		\left\| U_0^1 \right\|_{L^{\infty} H^{N+1}}+ \nu^{\frac{1}{2}} \left\| \nabla U_0^1 \right\|_{L^{2} H^{N+1}} & \lesssim \varepsilon,     \label{4..16}\\
		\left\| U_0^3 \right\|_{L^{\infty} H^{N+1}}+ \nu^{\frac{1}{2}} \left\| \nabla U_0^3 \right\|_{L^{2} H^{N+1}} & \lesssim  \varepsilon,     \label{4..18} \\
		\left\| U_0^2 \right\|_{L^{\infty} H^{N+2}}+ \nu^{\frac{1}{2}} \left\| \nabla U_0^2 \right\|_{L^{2} H^{N+2}} & \lesssim  \varepsilon.    \label{4..17} 
	\end{align}
	
	$(2)$ the non-zero frequency velocity $U_{\neq}^i$ $(i=1, 2, 3)$:
	\begin{align}
		\left\| U_{\neq}^1 \right\|_{L^{\infty} H^N}+ \nu^{\frac{1}{2}} \left\| \nabla_{L} U_{\neq}^1 \right\|_{L^{2} H^N} + \left\| \sqrt{-\dot{M}M}  U_{\neq}^1  \right\|_{L^{2} H^N} & \lesssim \varepsilon,   \label{4..22}\\
		\left\|  U_{\neq}^2 \right\|_{L^{\infty} H^N}+ \nu^{\frac{1}{2}} \left\|  \nabla_{L}  U_{\neq}^2 \right\|_{L^{2} H^N}+ \left\| \sqrt{-\dot{M}M}  U_{\neq}^2  \right\|_{L^{2} H^N} & \lesssim \varepsilon,      \label{4..23}\\
		\left\|  U_{\neq}^3 \right\|_{L^{\infty} H^N}+ \nu^{\frac{1}{2}} \left\| \nabla_{L} U_{\neq}^3 \right\|_{L^{2} H^N}+ \left\| \sqrt{-\dot{M}M}  U_{\neq}^3  \right\|_{L^{2} H^N} & \lesssim \varepsilon.  \label{4..24}
	\end{align}
\end{pro}

\begin{proof}
	For any $1<s \leqslant N$, it holds
	\begin{align}
		\left\|  U_0^2  \right\|_{H^{s+2}} \leqslant \left\|  \Delta U_0^2  \right\|_{H^{s}}+\left\| U_0^2   \right\|_{L^2}=\left\|  Q_0^2  \right\|_{H^{s}}+\left\| U_0^2   \right\|_{L^2}.\nonumber
	\end{align}
	Therefore, by the bootstrap hypotheses \eqref{4..4} and \eqref{4..12}, we infer that \eqref{4..17} holds. Similarly, other estimates on the zero frequency velocity follow from   \eqref{3..4}, \eqref{3.5}  and the bootstrap hypotheses.
	
	From \eqref{3..4} and Plancherel's theorem,   for $|k| \geqslant 1$ we easily have
	\begin{align}
	\left|  \widehat{U_{\neq}^{1}}  \right|&\leqslant \left||k,l|^{-2}\frac{\sqrt{k^{2}+l^{2}}}{\sqrt{k^{2}+(\eta-kt)^{2}+l^{2}}} \frac{\sqrt{k^{2}+(\eta-kt)^{2}+l^{2}}}{\sqrt{k^{2}+l^{2}}}\frac{k(\eta-kt)}{k^{2}+(\eta-kt)^{2}+l^{2}} \widehat{Q^{2}_{\neq}}\right|\nonumber\\&\quad
	+\left|\sqrt{\frac{\beta-1}{\beta}}|k,l|^{-2}\frac{\sqrt{k^{2}+l^{2}}}{\sqrt{k^{2}+(\eta-kt)^{2}+l^{2}}}\frac{\sqrt{k^{2}+(\eta-kt)^{2}+l^{2}}}{\sqrt{k^{2}+l^{2}}}\frac{l}{|k,\eta-kt,l|}K_{\neq}^{2}\right|\nonumber\\& \lesssim \left|  m {K}_{\neq}^{2}  \right| + \left|  m \widehat{Q_{\neq}^{2}}  \right|,\nonumber 
	\end{align} 
which together with \eqref{lem4.1.11}  and \eqref{3.5}  implies
	\begin{align*}
		&\left|  \widehat{U_{\neq}^{2}}  \right| \lesssim \frac{1}{\sqrt{k^2+l^2}} \, p^{-\frac{1}{2}} \left|  m \widehat{Q_{\neq}^{2}}  \right|\lesssim  \left|  m \widehat{Q_{\neq}^{2}}  \right|,\\
		&\left|  \widehat{U_{\neq}^{3}}  \right| \lesssim \left|  m {K}_{\neq}^{2}  \right| + \left|  m \widehat{Q_{\neq}^{2}}  \right|.
	\end{align*}
Hence, \eqref{4..22}--\eqref{4..24} also follow from the bootstrap hypotheses  \eqref{4..8} and \eqref{4..5}. 
\end{proof}

%
%
%

In fact, Theorem \ref{1..1} can be directly derived from Proposition \ref{pro4.2}. Therefore, we only need to prove Proposition \ref{pro4.1} and its proof will be given in the following section.
\section{Energy estimates on $Q_{\neq}^{ 2}, \check{K}_{\neq}^{2}$  and $Q_{0}^{ 2}, \check{K}_{0}^{2}$}  \label{sec5}
In this section, we aim to prove that under the bootstrap hypotheses of Proposition \ref{pro4.1}, the estimates on  $Q_{\neq}^{ 2}, \check{K}_{\neq}^{2}$  and $Q_{0}^{ 2}, \check{K}_{0}^{2}$  hold (i.e., \eqref{4..8}--\eqref{4..7}), with 10 replaced by 5 on the right-hand side.

\subsection{$H^{N}$ estimates on  $Q_{\neq}^2$  and $\check{K}_{\neq}^2$}\label{sec6.1}
First, by \eqref{HSX88} and \eqref{HSX99}, we  give the equations for $Q_{\neq}^2$ and $\check{K}_{\neq}^2$ as follows: 
\begin{align}\label{5.1}
	\begin{cases}
		\partial_{t} Q_{\neq}^{2}-\nu \Delta_{L} Q_{\neq}^{2}+i \sqrt{\beta (\beta-1)} \partial_{Z} |\nabla_{L}|^{-1} \check{K}_{\neq}^{2}\\
		\quad  =-\left(Q\cdot \nabla_{L}U^{2}\right)_{\neq}-\left(U \cdot\nabla_{L}Q^{2}\right)_{\neq}-2\left(\partial_{i}^{L}U^{j}\partial_{ij}^{L}U^{2}\right)_{\neq}+\partial_{Y}^{L}\left(\partial_{i}^{L}U^{j}\partial_{j}^{L}U^{i} \right)_{\neq},  \\
		\partial_{t} \check{K}_{\neq}^2-\nu \Delta_{L} \check{K}_{\neq}^2+\partial_{XY}^{L}\Delta_{L}^{-1}\check{K}_{\neq}^{2}-i\sqrt{\beta(\beta-1)}\partial_{Z}|\nabla_{L}|^{-1}Q^{2}_{\neq} \\
		\quad  =i\sqrt{\frac{\beta}{\beta-1}}\partial_{Z} |\nabla_{L}| \left(U\cdot\nabla_{L}U^{1} \right)_{\neq}-i\sqrt{\frac{\beta}{\beta-1}}\partial_{X}|\nabla_{L}| \left(U\cdot \nabla_{L}U^{3}\right)_{\neq}.
	\end{cases}
\end{align}
An energy estimate yields
\begin{align}\label{5.2}
	&\frac{1}{2}\left(\left\|mMQ_{\neq}^{2}\right\|_{H^{N}}^{2}+ \left\|mM\check{K}_{\neq}^{2}\right\|_{H^{N}}^{2}\right)+\left\| \left( \sqrt{-\dot{M}M}mQ_{\neq}^{2}, \sqrt{-\dot{M}M}m\check{K}_{\neq}^{2} \right)\right\|_{L^{2}H^{N}}^{2}\nonumber\\&\qquad+
	\left\|\left( \sqrt{-\dot{m}m}MQ_{\neq}^{2}, \sqrt{-\dot{m}m}M\check{K}_{\neq}^{2} \right)\right\|_{L^{2}H^{N}}^{2}+\nu \left(\left\|mM\nabla_{L}Q_{\neq}^{2}\right\|_{L^{2}H^{N}}^{2}+\left\|mM\nabla_{L}\check{K}_{\neq}^{2}\right\|_{L^{2}H^{N}}^{2} \right)\nonumber \\&\quad=\frac{1}{2}\left\|\left(mMQ_{\neq}^{2}(0), mM\check{K}_{\neq}^{2}(0) \right)\right\|_{H^{N}}^{2}-\int_{0}^{T}\int mM	 \left\langle D \right\rangle^{N}\check{K}_{\neq}^{2}mM\left\langle D \right\rangle^{N}\partial_{XY}^{L}\Delta_{L}^{-1}\check{K}_{\neq}^{2}dVdt\nonumber\\
	&\qquad+ \int_{0}^{T}\int mM\left\langle D \right\rangle^{N}Q_{\neq}^{2}mM\left\langle D \right\rangle^{N}\left[\partial_{Y}^{L}\left(\partial_{i}^{L}U^{j}\partial_{j}^{L}U^{i} \right)_{\neq}-\Delta_{L}\left(U\cdot \nabla_{L} U^{2}\right)_{\neq} \right]dVdt\nonumber\\
	& \qquad+\int_{0}^{T}\int mM\left\langle D \right\rangle^{N}\check{K}_{\neq}^{2}mM\left\langle D \right\rangle^{N}i\sqrt{\frac{\beta}{\beta-1}}\Big[  \partial_{Z}|\nabla_{L}|\left(U\cdot\nabla_{L}U^{1}\right)_{\neq} \nonumber \\
	&\quad\quad- \partial_{X}|\nabla_{L}|\left(U\cdot \nabla_{L}U^{3}\right)_{\neq}  \Big] dVdt\nonumber\\
	&\quad\triangleq \frac{1}{2}\left\|\left(mMQ_{\neq}^{2}(0), mM\check{K}_{\neq}^{2}(0) \right)\right\|_{H^{N}}^{2}+\mathcal{LS}\check{K}_{\neq}^{2}+\mathcal{NLP}Q_{\neq}^{2}+\mathcal{T}Q_{\neq}^{2}+\mathcal{T}\check{K}_{\neq}^{2}1+\mathcal{T}K_{\neq}^{2}2,
\end{align}
where we have used the following fact
\begin{align*}
	&\int mM \left\langle D \right\rangle^{N} Q_{\neq}^{2}mM\left\langle D \right\rangle^{N}\partial_{t}Q_{\neq}^{2}dV+\int mM\left\langle D \right\rangle^{N}\check{K}_{\neq}^{2}mM\left\langle D \right\rangle^{N}\partial_{t}\check{K}_{\neq}^{2}dV \\
	&\quad=\frac{1}{2}\frac{d}{dt}\left(\left\|mMQ_{\neq}^{2}\right\|_{H^{N}}^{2}+\left\|mM\check{K}_{\neq}^{2}\right\|_{H^{N}}^{2}\right)+\left\|\sqrt{-\dot{m}m}MQ_{\neq}^{2}\right\|_{H^{N}}^2 \\&\qquad+
	\left\|\sqrt{-\dot{m}m}M\check{K}_{\neq}^{2}\right\|_{H^{N}}^{2}+\left\|\sqrt{-\dot{M}M}mQ_{\neq}^{2}\right\|_{H^{N}}^{2}+\left\|\sqrt{-\dot{M}M}m\check{K}_{\neq}^{2}\right\|_{H^{N}}^{2}.
\end{align*}

Now  we are ready to estimate each term on the right-hand side of \eqref{5.2}. For the linear stretching term  $\mathcal{LS}\check{K}_{\neq}^{2}$, by the definition of $m$, H\"older's inequality, it can be bounded as
\begin{align} \label{5.4}
	\mathcal{LS}\check{K}_{\neq}^{2}&=-\int_{0}^{T}\int mM \left\langle D \right\rangle^{N}\check{K}_{\neq}^{2}mM \left\langle D \right\rangle^{N}\partial_{XY}^{L}\Delta_{L}^{-1}\check{K}_{\neq}^{2}dVdt\nonumber\\&=
	\int_{0}^{T}\int mM\widehat{ \left\langle D \right\rangle^{N}}K_{\neq}^{2}mM \widehat{\left\langle D \right\rangle^{N}}\frac{-k(\eta-k t)}{k^{2}+(\eta-kt)^{2}+l^{2}}K_{\neq}^{2}d\xi dt\nonumber\\&=
	\int_{0}^{T}\int mM\widehat{\left\langle D \right\rangle^{N}}K_{\neq}^{2}mM\widehat{\left\langle D \right\rangle^{N}}\frac{-k(\eta-kt)}{k^{2}+(\eta-kt)^{2}+l^{2}}K_{\neq}^{2}\cdot 1_{\left\{t: \, 0<t<\frac{\eta}{k}\right\}}d\xi dt\nonumber\\&\quad+
	\int_{0}^{T}\int mM\widehat{\left\langle D \right\rangle^{N}}K_{\neq}^{2}mM\widehat{\left\langle D \right\rangle^{N}}\frac{-\dot{m}}{m}K_{\neq}^{2}\cdot 1_ { \left\{t: \, \frac{\eta}{k}\leqslant t \leqslant  \frac{\eta}{k}+1000\nu^{-\frac{1}{3}} \right\}}d\xi dt\nonumber\\
	&\quad+\int_{0}^{T}\int mM\widehat{\left\langle D \right\rangle^{N}}K_{\neq}^{2}mM\widehat{\left\langle D \right\rangle^{N}}\frac{-k(\eta-kt)}{k^{2}+(\eta-kt)^{2}+l^{2}}K_{\neq}^{2}\cdot 1_{\left\{t: \, t>\frac{\eta}{k}+1000\nu^{-\frac{1}{3}}\right\}}d\xi dt\nonumber\\&\leqslant
	\left\|\sqrt{-\dot{m}m}M\check{K}_{\neq}^{2}\right\|_{L^{2}H^{N}}^{2}+\frac{\nu}{2}\left\|mM\nabla_{L}\check{K}_{\neq}^{2}\right\|_{L^{2}H^{N}}^{2},
\end{align} 
where in the last inequality  we have used \eqref{4...3}. Indeed,   dissipation overcomes  stretching if $ t>\frac{\eta}{k}+1000  \nu^{-\frac{1}{3}}$ and $\frac{1}{2}$ in the last term as in \eqref{5.4} can be chosen.  

Before estimating $\mathcal{NLP}Q_{\neq}^{2}$, we introduce the following fact, which will be frequently used below. By the definition of $K^{2}$ as in \eqref{3.13}, we obtain that $U^i$ $(i=1, 2, 3)$ satisfies the  following inequality deduced by \eqref{lem4.1.11}, \eqref{3..4} and \eqref{3.5} 
\begin{align}\label{5.15}
	\begin{cases}
		\big|\widehat{U_{\neq}^{1}}\big|\leqslant \sqrt{\dfrac{\beta-1}{\beta}}|k,l|^{-2}\dfrac{|l|}{\sqrt{k^{2}+l^{2}}}\dfrac{\sqrt{k^{2}+l^{2}}}{\sqrt{k^{2}+(\eta-kt)^{2}+l^{2}}}\big|K_{\neq}^{2}\big|\\ \qquad\quad+|k,l|^{-2}\dfrac{|\eta-kt|}{|k,\eta-kt,l|}\dfrac{|k|}{\sqrt{k^{2}+l^{2}}}\dfrac{\sqrt{k^{2}+l^{2}}}{\sqrt{k^{2}+(\eta-kt)^{2}+l^{2}}}\big|\widehat{Q_{\neq}^{2}}\big|\\ 
		\qquad \leqslant C |k,l|^{-2}\left(\big|m M K_{\neq}^{2}\big|+\big|m M\widehat{Q_{\neq}^{2}}\big|\right),\\
		\big|\widehat{U_{\neq}^{3}}\big|\leqslant \sqrt{\dfrac{\beta-1}{\beta}}|k,l|^{-2} \dfrac{|k|}{\sqrt{k^{2}+l^{2}}}\dfrac{\sqrt{k^{2}+l^{2}}}{\sqrt{k^{2}+(\eta-kt)^{2}+l^{2}}}\big|K_{\neq}^{2}\big|  \\ 
		\qquad\quad+|k,l|^{-2}\dfrac{|\eta-kt|}{|k,\eta-kt,l|}\dfrac{|l|}{\sqrt{k^{2}+l^{2}}}\dfrac{\sqrt{k^{2}+l^{2}}}{\sqrt{k^{2}+(\eta-kt)^{2}+l^{2}}}\big|\widehat{Q_{\neq}^{2}}\big|\\
		\qquad\leqslant C |k,l|^{-2} \left(\big|m MK_{\neq}^{2}\big|+\big|m M\widehat{Q_{\neq}^{2}}\big|\right), \\
		\left|  \widehat{U_{\neq}^{2}}  \right| \leqslant C \dfrac{1}{\sqrt{k^2+l^2}} \, \dfrac{1}{\sqrt{k^2+(\eta-kt)^2+l^2}} \left|  m M \widehat{Q_{\neq}^{2}}  \right|. 
	\end{cases}
\end{align}

For the nonlinear pressure term $\mathcal{NLP}Q_{\neq}^{2}$, now we split this into there parts
\begin{align}\label{5.5}
	\mathcal{NLP}Q_{\neq}^{2}&=\int_{0}^{T}\int  mM \left\langle D \right\rangle^{N}Q_{\neq}^{2} mM  \left\langle D \right\rangle^{N}\partial_{Y}^{L}\left(\partial_{i}^{L}U^{j}\partial_{j}^{L}U^{i} \right)_{\neq}dVdt\nonumber\\&=
	\int_{0}^{T}\int mM \left\langle D \right\rangle^{N}
	Q_{\neq}^{2} mM  \left\langle D \right\rangle^{N}\partial_{Y}^{L}\Big[\partial_{i} U_{0}^{j}\partial_{j}^{L}U_{\neq}^{i} \nonumber\\&\quad+\partial_{i}^{L}U_{\neq}^{j}\partial_{j}U_{0}^{i}+\partial_{i}^{L}U_{\neq}^{j}\partial_{j}^{L}U_{\neq}^{i}  \Big]_{\neq}dVdt\nonumber\\
	&\triangleq\mathcal{NLP}Q_{\neq}^{2}(U_{0},U_{\neq})+\mathcal{NLP}Q_{\neq}^{2}(U_{\neq},U_{0})+\mathcal{NLP}Q_{\neq}^{2}(U_{\neq}, U_{\neq}).
\end{align}
We proceed with the term $\mathcal{NLP}Q_{\neq}^{2}(U_{0},U_{\neq})$. It can be divided into 
\begin{align}\label{5.6}
	\mathcal{NLP}Q_{\neq}^{2}(U_{0},U_{\neq})&=\int_{0}^{T} \int mM  \left\langle D \right\rangle^{N} Q_{\neq}^{2} mM  \left\langle D \right\rangle^{N}\partial_{Y}^{L}\left(\partial_{i} U_{0}^{j}\partial_{j}^{L}U_{\neq}^{i} \right)_{\neq} \cdot 1_{i\neq 1}dVdt\nonumber\\
	&=\int_{0}^{T}\int mM  \left\langle D \right\rangle^{N} Q_{\neq}^{2} mM  \left\langle D \right\rangle^{N}\partial_{Y}^{L}\Big[ \partial_{Y} U_{0}^{1}\partial_{X}U_{\neq}^{2}+\partial_{Y} U_{0}^{2}\partial_{Y}^{L}U_{\neq}^{2}\nonumber\\
	&\quad+\partial_{Y} U_{0}^{3}\partial_{Z}U_{\neq}^{2}+\partial_{Z}U_{0}^{1}\partial_{X}U_{\neq}^{3}+\partial_{Z}U_{0}^{2}\partial_{Y}U_{\neq}^{3}+\partial_{Z}U_{0}^{3}\partial_{Z}U_{\neq}^{3}\Big]_{\neq} dVdt\nonumber\\
	&\triangleq 
	\mathcal{NLP}Q_{\neq}^{2}(U_{0}^{1},U_{\neq}^{2})+\mathcal{NLP}Q_{\neq}^{2}(U_{0}^{2},U_{\neq}^{2})+\mathcal{NLP}Q_{\neq}^{2}(U_{0}^{3},U_{\neq}^{2})\nonumber\\
	&\quad+\mathcal{NLP}Q_{\neq}^{2}(U_{0}^{1},U_{\neq}^{3})+\mathcal{NLP}Q_{\neq}^{2}(U_{0}^{2},U_{\neq}^{3})+\mathcal{NLP}Q_{\neq}^{2}(U_{0}^{3},U_{\neq}^{3}).
\end{align} 
On the one hand, by the same argument of \eqref{5.15},  we have
\begin{align}\label{5..102}
	\left|\widehat{\partial_{X} U_{\neq}^{2}}, \widehat{\partial_{Y}^{L} U_{\neq}^{2}}, \widehat{\partial_{Z} U_{\neq}^{2}}\right|\leqslant C\left|mM\widehat{Q_{\neq}^{2}}\right|.
\end{align}  
With \eqref{2.3}, \eqref{4..8}, \eqref{4..16}, the fact \eqref{5..102} and  Corollary \ref{cor4.1} at hand, we conclude that
\begin{align} 
	\mathcal{NLP}Q_{\neq}^{2}(U_{0}^{1},U_{\neq}^{2})&=\int_{0}^{T}\int  mM  \left\langle D \right\rangle^{N} Q_{\neq}^{2} mM  \left\langle D \right\rangle^{N} \partial_{Y}^{L}\left(\partial_{Y} U_{0}^{1}\partial_{X}U_{\neq}^{2} \right)_{\neq}dVdt\nonumber\\
	&\leqslant
	\left\|mM\nabla_{L}Q_{\neq}^{2}\right\|_{L^{2}H^{N}}\left\|U_{0}^{1}\right\|_{L^{\infty}H^{N+1}}\left\|\partial_{X} U_{\neq}^{2}\right\|_{L^{2}H^{N}}\nonumber\\
	&\leqslant
	\left\|mM\nabla_{L}Q_{\neq}^{2}\right\|_{L^{2}H^{N}}\left\|U_{0}^{1}\right\|_{L^{\infty}H^{N+1}}\left\| mM Q_{\neq}^{2}\right\|_{L^{2}H^{N}}\nonumber\\
	&\lesssim
	\varepsilon\nu^{-\frac{1}{2}}\varepsilon\varepsilon\nu^{-\frac{1}{6}}\nonumber\\
	&\lesssim \varepsilon^{2}(\varepsilon\nu^{-\frac{2}{3}})\nonumber\\
	&\lesssim \varepsilon^{2},
\end{align} 
where $\varepsilon\nu^{-\frac{2}{3}}$ sufficiently small. Using  \eqref{2.3},  \eqref{4..8}, \eqref{4..18}, \eqref{4..17}, \eqref{5..102} and Corollary \ref{cor4.1}, the term $\mathcal{NLP}Q_{\neq}^{2}(U_{0}^{2}, U_{\neq}^{2})$  satisfies
\begin{align}
	\mathcal{NLP}Q_{\neq}^{2}(U_{0}^{2},U_{\neq}^{2})&=\int_{0}^{T}\int mM  \left\langle D \right\rangle^{N} Q_{\neq}^{2} mM  \left\langle D \right\rangle^{N} \partial_{Y}\left(\partial_{Y} U_{0}^{2}\partial_{Y}^{L}U_{\neq}^{2} \right)_{\neq}dVdt\nonumber\\
	&\leqslant\left\|mM\nabla_{L}Q_{\neq}^{2}\right\|_{L^{2}H^{N}}\left\|U_{0}^{2}\right\|_{L^{\infty}H^{N+1}}\left\|\partial_{Y}^{L}U_{\neq}^{2}\right\|_{L^{2}H^{N}}\nonumber\\
	&\leqslant\left\|mM\nabla_{L}Q_{\neq}^{2}\right\|_{L^{2}H^{N}}\left\|U_{0}^{2}\right\|_{L^{\infty}H^{N+1}}\left\|mM Q_{\neq}^{2}\right\|_{L^{2}H^{N}}\nonumber\\
	&\lesssim \varepsilon\nu^{-\frac{1}{2}}\varepsilon \varepsilon\nu^{-\frac{1}{6}}\nonumber\\
	&\lesssim \varepsilon^{2}(\varepsilon\nu^{-\frac{2}{3}})\nonumber\\
	&\lesssim \varepsilon^{2}.
\end{align} 
The term $\mathcal{NLP}Q_{\neq}^{2}(U_{0}^{3}, U_{\neq}^{2})$  can be bounded as
\begin{align}
	\mathcal{NLP}Q_{\neq}^{2}(U_{0}^{3}, U_{\neq}^{2})&=\int_{0}^{T}\int  mM  \left\langle D \right\rangle^{N} Q_{\neq}^{2} mM  \left\langle D \right\rangle^{N} \partial_{Y}^{L}\left(\partial_{Y} U_{0}^{3}\partial_{Z}U_{\neq}^{2} \right)_{\neq}dVdt\nonumber\\
	&\leqslant\left\|mM\nabla_{L}Q_{\neq}^{2}\right\|_{L^{2}H^{N}}\left\|U_{0}^{3}\right\|_{L^{\infty}H^{N+1}}\left\|\partial_{Z}U_{\neq}^{2}\right\|_{L^{2}H^{N}}\nonumber\\
	&\leqslant\left\|mM\nabla_{L}Q_{\neq}^{2}\right\|_{L^{2}H^{N}}\left\|U_{0}^{3}\right\|_{L^{\infty}H^{N+1}}\left\|mM Q_{\neq}^{2}\right\|_{L^{2}H^{N}}\nonumber\\
	&\lesssim \varepsilon\nu^{-\frac{1}{2}}\varepsilon \varepsilon\nu^{-\frac{1}{6}}\nonumber\\
	&\lesssim \varepsilon^{2}(\varepsilon\nu^{-\frac{2}{3}})\nonumber\\
	&\lesssim \varepsilon^{2},
\end{align}
where $\varepsilon \nu^{-\frac{2}{3}}$ is chosen small enough. 
On the other hand, in analogy with \eqref{5.15}, we can easily infer that 
\begin{align}\label{5.16}
	\left|\widehat{\partial_{X}U_{\neq}^{1,3}}, \widehat{\partial_{Z}U_{\neq}^{1,3}} \right|\leqslant C\left(\left|mM K_{\neq}^{2}\right|+\left|mM\widehat{Q_{\neq}^{2}}\right|\right).
\end{align}  
Now consider the fourth term appearing in \eqref{5.6}. By using  \eqref{2.3}, \eqref{4..8}, \eqref{4..5}, \eqref{4..16}, \eqref{5.16} and  Corollary \ref{cor4.1}, it yields
\begin{align}
	\mathcal{NLP}Q_{\neq}^{2}(U_{0}^{1}, U_{\neq}^{3})&=\int_{0}^{T}\int mM  \left\langle D \right\rangle^{N} Q_{\neq}^{2}mM  \left\langle D \right\rangle^{N}\partial_{Y}^{L}\left(\partial_{Z}U_{0}^{1}\partial_{X}U_{\neq}^{3}\right)_{\neq}dVdt\nonumber\\
	&\leqslant
	\left\|mM\nabla_{L}Q_{\neq}^{2}\right\|_{L^{2}H^{N}}\left\|U_{0}^{1}\right\|_{L^{\infty}H^{N+1}}\left\|\partial_{X}U_{\neq}^{3}\right\|_{L^{2}H^{N}}\nonumber\\
	&\leqslant
	\left\|mM\nabla_{L}Q_{\neq}^{2}\right\|_{L^{2}H^{N}}\left\|U_{0}^{1}\right\|_{L^{\infty}H^{N+1}} \left(\left\| mM Q_{\neq}^{2}\right\|_{L^{2}H^{N}}+ \left\| m M \check{K}_{\neq}^2 \right\|_{L^{2}H^{N}}\right)\nonumber\\
	&
	\lesssim \varepsilon\nu^{-\frac{1}{2}}\varepsilon \varepsilon\nu^{-\frac{1}{6}}\nonumber\\
	&\lesssim \varepsilon^{2}(\varepsilon\nu^{-\frac{2}{3}})\nonumber\\
	&\lesssim
	\varepsilon^{2}.
\end{align}
For $\mathcal{NLP}Q_{\neq}^{2}(U_{0}^{2}, U_{\neq}^{3})$ and  $\mathcal{NLP}Q_{\neq}^{2}(U_{0}^{3}, U_{\neq}^{3})$,   by utilizing   \eqref{4..8}, \eqref{4..5}, \eqref{4..18}, \eqref{4..17}, \eqref{4..24}, \eqref{5.16} and  Corollary \ref{cor4.1}, we have
\begin{align}
	\mathcal{NLP}Q_{\neq}^{2}(U_{0}^{2}, U_{\neq}^{3})&=
	\int_{0}^{T}\int  mM  \left\langle D \right\rangle^{N} Q_{\neq}^{2} mM  \left\langle D \right\rangle^{N}\partial_{Y}^{L}\left(\partial_{Z}U_{0}^{2}\partial_{Y}^{L}U_{\neq}^{3} \right)_{\neq}dVdt\nonumber\\
	&\leqslant
	\left\| mM \nabla_{L} Q_{\neq}^{2}\right\|_{L^{2}H^{N}}\left\|U_{0}^{2}\right\|_{L^{\infty}H^{N+1}}\left\|\nabla_{L} U_{\neq}^{3}\right\|_{L^{2}H^{N}}\nonumber\\
	&\lesssim \varepsilon\nu^{-\frac{1}{2}}\varepsilon \varepsilon\nu^{-\frac{1}{2}}\nonumber\\
	&\lesssim \varepsilon^{2}(\varepsilon\nu^{-1})\nonumber\\
	&\lesssim \varepsilon^{2}
\end{align}
and 
\begin{align}
	\mathcal{NLP}Q_{\neq}^{2}(U_{0}^{3}, U_{\neq}^{3})&=\int_{0}^{T}\int mM  \left\langle D \right\rangle^{N} Q_{\neq}^{2}mM  \left\langle D \right\rangle^{N} \partial_{Y}^{L}(\partial_{Z}U_{0}^{3}\partial_{Z}U_{\neq}^{3})_{\neq}dVdt\nonumber\\
	&\leqslant 
	\left\|mM\nabla_{L}Q_{\neq}^{2}\right\|_{L^{2}H^{N}}\left\|U_{0}^{3}\right\|_{L^{\infty}H^{N+1}}\left\|\partial_{Z}U_{\neq}^{3}\right\|_{L^{2}H^{N}}\nonumber\\
	&\leqslant 
	\left\|mM\nabla_{L}Q_{\neq}^{2}\right\|_{L^{2}H^{N}}\left\|U_{0}^{3}\right\|_{L^{\infty}H^{N+1}} \left(\left\| mM Q_{\neq}^{2}\right\|_{L^{2}H^{N}}+ \left\| m M \check{K}_{\neq}^2 \right\|_{L^{2}H^{N}}\right)   \nonumber\\
	&\lesssim \varepsilon\nu^{-\frac{1}{2}}\varepsilon \varepsilon\nu^{-\frac{1}{6}} \nonumber\\
	&\lesssim \varepsilon^{2}(\varepsilon\nu^{-\frac{2}{3}})\nonumber\\
	&\lesssim \varepsilon^{2},
\end{align}
where $\varepsilon\nu^{-1}$ sufficiently small. Similarly, we easily deduce that
\begin{align}
	\mathcal{NLP}Q_{\neq}^{2}(U_{\neq},U_{0})\lesssim \varepsilon^{2}(\varepsilon\nu^{-1})\lesssim
	\varepsilon^{2}.
\end{align} 

However, the  term $\mathcal{NLP}Q_{\neq}^{2}(U_{\neq}, U_{\neq})$  is much trickier to deal with. To do this, we further decompose it into some parts as follows:
\begin{align}\label{5.13}
	\mathcal{NLP}Q_{\neq}^{2}(U_{\neq}, U_{\neq})&=\int_{0}^{T}\int
	mM  \left\langle D \right\rangle^{N} Q_{\neq}^{2} mM  \left\langle D \right\rangle^{N} \partial_{Y}^{L}\left(\partial_{i}^{L}U_{\neq}^{j}\partial_{j}^{L}U_{\neq}^{i}\right)_{\neq}dVdt\nonumber\\
	&=\int_{0}^{T}\int mM  \left\langle D \right\rangle^{N}Q_{\neq}^{2}mM  \left\langle D \right\rangle^{N} \partial_{Y}^{L}\left(\partial_{Y}^{L}U_{\neq}^{1}\partial_{X}U_{\neq}^{2} \right)_{\neq}dVdt\nonumber\\&\quad+
	\int_{0}^{T}\int mM  \left\langle D \right\rangle^{N}Q_{\neq}^{2} mM  \left\langle D \right\rangle^{N}\partial_{Y}^{L}\left(\partial_{X}U_{\neq}^{1}\partial_{X}U_{\neq}^{1}+\partial_{Z}U_{\neq}^{1}\partial_{X}U_{\neq}^{3} \right)_{\neq}dVdt\nonumber\\&\quad+
	\int_{0}^{T}\int  mM  \left\langle D \right\rangle^{N}Q_{\neq}^{2}
	mM  \left\langle D \right\rangle^{N}\partial_{Y}^{L}\left(\partial_{Y}^{L}U_{\neq}^{2}\partial_{Y}^{L}U_{\neq}^{2}\right)_{\neq}dVdt\nonumber\\&\quad+\int_{0}^{T}\int  mM  \left\langle D \right\rangle^{N}Q_{\neq}^{2}  mM  \left\langle D \right\rangle^{N} \partial_{Y}^{L}\left(\partial_{X}U_{\neq}^{2}\partial_{Y}^{L}U_{\neq}^{1}+\partial_{Z}U_{\neq}^{2}\partial_{Y}^{L}U_{\neq}^{3} \right)_{\neq}dVdt\nonumber\\&\quad+
	\int_{0}^{T}\int  mM  \left\langle D \right\rangle^{N}Q_{\neq}^{2}  mM  \left\langle D \right\rangle^{N} \partial_{Y}^{L}\left(\partial_{Y}^{L}U_{\neq}^{3}\partial_{Z}U_{\neq}^{2} \right)_{\neq}dVdt\nonumber\\&\quad+
	\int_{0}^{T}\int  mM  \left\langle D \right\rangle^{N}Q_{\neq}^{2}  mM  \left\langle D \right\rangle^{N} \partial_{Y}^{L}\left(\partial_{X}U_{\neq}^{3}\partial_{Z}U_{\neq}^{1}+\partial_{Z}U_{\neq}^{3}\partial_{Z}U_{\neq}^{3}\right)_{\neq}dVdt\nonumber\\&\triangleq
	\mathcal{NLP}Q_{\neq}^{2}(U_{\neq}^{1},U_{\neq}^{2})+ \mathcal{NLP}Q_{\neq}^{2}(U_{\neq}^{1},U_{\neq}^{1,3})+ \mathcal{NLP}Q_{\neq}^{2}(U_{\neq}^{2},U_{\neq}^{2})
	\nonumber\\&\quad+
	\mathcal{NLP}Q_{\neq}^{2}(U_{\neq}^{2},U_{\neq}^{1,3})+
	\mathcal{NLP}Q_{\neq}^{2}(U_{\neq}^{3},U_{\neq}^{2})+ \mathcal{NLP}Q_{\neq}^{2}(U_{\neq}^{3},U_{\neq}^{1,3}).
\end{align}
We are now in a position of estimating each term on the right-hand side of \eqref{5.13}  based on Propositions \ref{pro4.1} and \ref{pro4.2}.  
Indeed, for $\mathcal{NLP}Q_{\neq}^{2}(U_{\neq}^{1},U_{\neq}^{2})$, it allows us to conclude from \eqref{4..8}, \eqref{4..5}, \eqref{5.15} and the fact \eqref{5..102} that
\begin{align}
	\mathcal{NLP}Q_{\neq}^{2}(U_{\neq}^{1}, U_{\neq}^{2})&=\int_{0}^{T}\int mM  \left\langle D \right\rangle^{N} Q_{\neq}^{2}mM  \left\langle D \right\rangle^{N}\partial_{Y}^{L}\left(\partial_{Y}^{L}U_{\neq}^{1}\partial_{X}U_{\neq}^{2}\right)_{\neq}dVdt\nonumber\\
	&\leqslant \left\|mM\nabla_{L}Q_{\neq}^{2}\right\|_{L^{2}H^{N}}\left(\left\|mM\nabla_{L}\check{K}_{\neq}^{2}\right\|_{L^{2}H^{N}}+\left\|mM\nabla_{L}Q_{\neq}^{2} \right\|_{L^{2}H^{N}}\right)
	\left\|mMQ_{\neq}^{2}\right\|_{L^{\infty}H^{N}}\nonumber\\
	&\lesssim \varepsilon\nu^{-\frac{1}{2}} \varepsilon \nu^{-\frac{1}{2}} \varepsilon\nonumber 
	\\
	&\lesssim  
	\varepsilon^{2}(\varepsilon\nu^{-1})\nonumber\\
	&\lesssim
	\varepsilon^{2},
\end{align}
where $\varepsilon\nu^{-1}$ is chosen small enough. 
And we apply  \eqref{4..8}, \eqref{4..5}, Corollary \ref{cor4.1} and \eqref{5.16} to obtain
\begin{align}
	\mathcal{NLP}Q_{\neq}^{2}(U_{\neq}^{1}, U_{\neq}^{1,3})&=\int_{0}^{T}\int mM  \left\langle D \right\rangle^{N} Q_{\neq}^{2} mM  \left\langle D \right\rangle^{N} \partial_{Y}^{L}\left(\partial_{X}U_{\neq}^{1}\partial_{X}U_{\neq}^{1}+\partial_{Z}U_{\neq}^{1}\partial_{X}U_{\neq}^{3} \right)_{\neq}dVdt\nonumber\\
	&=-\int_{0}^{T}\int mM  \left\langle D \right\rangle^{N}\partial_{Y}^{L}Q_{\neq}^{2} mM  \left\langle D \right\rangle^{N}\left(\partial_{X}U_{\neq}^{1}\partial_{X}U_{\neq}^{1}+\partial_{Z}U_{\neq}^{1}\partial_{X}U_{\neq}^{3}\right)_{\neq}dVdt\nonumber\\
	&\leqslant
	\left\|mM\nabla_{L}Q_{\neq}^{2}\right\|_{L^{2}H^{N}}\left(\left\|mM\check{K}_{\neq}^{2}\right\|_{L^{2}H^{N}}+\left\|mMQ_{\neq}^{2}\right\|_{L^{2}H^{N}} \right)\nonumber\\
	&\quad\times \left(\left\|mM\check{K}_{\neq}^{2}\right\|_{L^{\infty}H^{N}}+\left\|mMQ_{\neq}^{2}\right\|_{L^{\infty}H^{N}}\right)\nonumber\\
	&\lesssim \varepsilon\nu^{-\frac{1}{2}} \varepsilon \nu^{-\frac{1}{6}} \varepsilon\nonumber 
	\\
	&\lesssim \varepsilon^{2}(\varepsilon\nu^{-\frac{2}{3}})\nonumber\\
	&\lesssim \varepsilon^{2}.
\end{align}
Similarly,  using \eqref{4..8}, \eqref{4..5}, \eqref{5.15}, \eqref{5..102} and Corollary \ref{cor4.1}, the term $\mathcal{NLP}Q_{\neq}^{2}(U_{\neq}^{2}, U_{\neq}^{2})$ can be bounded as
\begin{align}
	\mathcal{NLP}Q_{\neq}^{2}(U_{\neq}^{2}, U_{\neq}^{2})&=-\int_{0}^{T}\int mM  \left\langle D \right\rangle^{N} \partial_{Y}^{L} Q_{\neq}^{2}mM  \left\langle D \right\rangle^{N} \left(\partial_{Y}^{L}U_{\neq}^{2}\partial_{Y}^{L}U_{\neq}^{2} \right)_{\neq}dVdt\nonumber\\
	&\leqslant
	\left\|mM\nabla_{L}Q_{\neq}^{2}\right\|_{L^{2}H^{N}}\left\|mMQ_{\neq}^{2}\right\|_{L^{2}H^{N}}\left\|mMQ_{\neq}^{2}\right\|_{L^{\infty}H^{N}}\nonumber\\
	&\lesssim  \varepsilon\nu^{-\frac{1}{2}} \varepsilon \nu^{-\frac{1}{6}} \varepsilon\nonumber 
	\\
	&\lesssim  
	\varepsilon^{2}(\varepsilon\nu^{-\frac{2}{3}})\nonumber\\
	&\lesssim \varepsilon^{2},
\end{align}
and  $\mathcal{NLP}Q_{\neq}^{2}(U_{\neq}^{2}, U_{\neq}^{1,3})$ satisfies
\begin{align}
	\mathcal{NLP}Q_{\neq}^{2}(U_{\neq}^{2}, U_{\neq}^{1,3})&=-\int_{0}^{T}\int mM  \left\langle D \right\rangle^{N}\partial_{Y}^{L} Q_{\neq}^{2} mM  \left\langle D \right\rangle^{N}\left(\partial_{X}U_{\neq}^{2}\partial_{Y}^{L}U_{\neq}^{1}+\partial_{Z}U_{\neq}^{2}\partial_{Y}^{L}U_{\neq}^{3} \right)_{\neq}dVdt\nonumber\\
	&\leqslant
	\left\|mM\nabla_{L}Q_{\neq}^{2}\right\|_{L^{2}H^{N}}\left\|mMQ_{\neq}^{2}\right\|_{L^{\infty}H^{N}} \nonumber \\
	& \quad \times\left(\left\|mM\nabla_{L}\check{K}_{\neq}^{2}\right\|_{L^{2}H^{N}}+\left\|mM\nabla_{L}Q_{\neq}^{2}\right\|_{L^{2}H^{N}}\right)\nonumber\\
	&\lesssim
	\varepsilon\nu^{-\frac{1}{2}}\varepsilon \varepsilon \nu^{-\frac{1}{2}} \nonumber 
	\\
	&\lesssim  
	\varepsilon^{2}(\varepsilon\nu^{-1})\nonumber\\
	&\lesssim \varepsilon^{2},
\end{align}
where $\varepsilon\nu^{-1}$ sufficiently small. Finally, we focus on the terms $\mathcal{NLP}Q_{\neq}^{2}(U_{\neq}^{3}, U_{\neq}^{2})$ and $\mathcal{NLP}Q_{\neq}^{2}(U_{\neq}^{3}, U_{\neq}^{1,3})$ in  \eqref{5.13}. Combining \eqref{4..8}, \eqref{4..5}, \eqref{5.15},  \eqref{5..102} and \eqref{5.16} gives
\begin{align}
	\mathcal{NLP}Q_{\neq}^{2}(U_{\neq}^{3}, U_{\neq}^{2})&=-\int_{0}^{T}\int mM  \left\langle D \right\rangle^{N}\partial_{Y}^{L} Q_{\neq}^{2} mM  \left\langle D\right\rangle^{N}\left( \partial_{Y}^{L}U_{\neq}^{3}\partial_{Z}U_{\neq}^{2} \right)_{\neq}dVdt\nonumber\\
	&\leqslant
	\left\|mM\nabla_{L}Q_{\neq}^{2}\right\|_{L^{2}H^{N}} \left(\left\|mM\nabla_{L}\check{K}_{\neq}^{2}\right\|_{L^{2}H^{N}}+\left\|mM\nabla_{L}Q_{\neq}^{2}\right\|_{L^{2}H^{N}}\right)\left\|mM	Q_{\neq}^{2}\right\|_{L^{\infty}H^{N}}\nonumber\\
	&\lesssim
	\varepsilon\nu^{-\frac{1}{2}} \varepsilon \nu^{-\frac{1}{2}} \varepsilon\nonumber \\
	&\lesssim  \varepsilon^{2}(\varepsilon\nu^{-1})\nonumber\\
	&\lesssim \varepsilon^{2}
\end{align}
and 
\begin{align}
	\mathcal{NLP}Q_{\neq}^{2}(U_{\neq}^{3}, U_{\neq}^{1,3})&=-\int_{0}^{T}\int mM  \left\langle D \right\rangle^{N}\partial_{Y}^{L} Q_{\neq}^{2} mM  \left\langle D\right\rangle^{N}\left( \partial_{X}U_{\neq}^{3}\partial_{Z}U_{\neq}^{1}+\partial_{Z}   U_{\neq}^{3}\partial_{Z}U_{\neq}^{3} \right)_{\neq}dVdt\nonumber\\
	&\leqslant
	\left\|mM\nabla_{L}Q_{\neq}^{2}\right\|_{L^{2}H^{N}} \left(\left\|mM\check{K}_{\neq}^{2}\right\|_{L^{2}H^{N}}+\left\|mMQ_{\neq}^{2}\right\|_{L^{2}H^{N}}\right) \nonumber\\
	&\quad \times \left(\left\|mM\check{K}_{\neq}^{2}\right\|_{L^{\infty}H^{N}}+\left\|mM Q_{\neq}^{2}\right\|_{L^{\infty}H^{N}} \right)\nonumber \\
	&\lesssim
	\varepsilon\nu^{-\frac{1}{2}} \varepsilon \nu^{-\frac{1}{6}} \varepsilon\nonumber \\
	&\lesssim  \varepsilon^{2}(\varepsilon\nu^{-\frac{2}{3}})\nonumber\\
	&\lesssim\varepsilon^{2}.
\end{align}

To bound the transport term $\mathcal{T}Q_{\neq}^{2}$, we split  $\mathcal{T}Q_{\neq}^{2}$ into three parts
\begin{align}\label{6.22}
	\mathcal{T}Q_{\neq}^{2}&=-\int_{0}^{T}\int mM  \left\langle D \right\rangle^{N} Q_{\neq}^{2} mM  \left\langle D \right\rangle^{N} \Delta_{L}(U\cdot \nabla_{L}U^{2})_{\neq}dVdt\nonumber\\&=
	-\int_{0}^{T}\int mM  \left\langle D \right\rangle^{N} Q_{\neq}^{2} mM  \left\langle D \right\rangle^{N}\Delta_{L}\left( U_{0}\cdot \nabla_{L}U_{\neq}^{2}+U_{\neq}\cdot\nabla U_{0}^{2}+U_{\neq}\cdot \nabla_{L}U_{\neq}^{2}\right)_{\neq}dVdt\nonumber\\&\triangleq
	\mathcal{T}Q_{\neq}^{2}(U_{0},U_{\neq}^{2})+\mathcal{T}Q_{\neq}^{2}(U_{\neq},U_{0}^{2})+\mathcal{T}Q_{\neq}^{2}(U_{\neq},U_{\neq}^{2}).
\end{align}
Now, we use \eqref{4..8},
\eqref{4..16}--\eqref{4..17}, \eqref{5..102} and Corollary \ref{cor4.1} to estimate the term $\mathcal{T}Q_{\neq}^{2} (U_{0},U_{\neq}^{2})$
\begin{align}
	\mathcal{T}Q_{\neq}^{2}(U_{0}, U_{\neq}^{2})&=\int_{0}^{T}\int mM  \left\langle D \right\rangle^{N} \nabla_{L}Q_{\neq}^{2} mM  \left\langle D \right\rangle^{N} \nabla_{L}\left(U_{0}\cdot \nabla_{L}U_{\neq}^{2} \right)_{\neq}dVdt\nonumber\\
	&\leqslant \left\|mM\nabla_{L}Q_{\neq}^{2}\right\|_{L^{2}H^{N}}\left\|U_{0}\right\|_{L^{\infty}H^{N+1}}\left\|mMQ_{\neq}^{2}\right\|_{L^{2}H^{N}}\nonumber\\
	&\quad+
	\left\|mM\nabla_{L}Q_{\neq}^{2}\right\|_{L^{2}H^{N}}\left\|U_{0}\right\|_{L^{\infty}H^{N}}\left\|mM Q_{\neq}^{2}\right\|_{L^{2}H^{N}}\nonumber 
	\\
	&\lesssim
	\varepsilon\nu^{-\frac{1}{2}} \varepsilon  \varepsilon \nu^{-\frac{1}{6}}\nonumber 
	\\
	&\lesssim  
	\varepsilon^{2}(\varepsilon\nu^{-\frac{2}{3}})\nonumber\\
	&\lesssim\varepsilon^{2}.
\end{align}
Together  \eqref{4..17}--\eqref{4..24} with \eqref{4..8}, one obtains
\begin{align}
	\mathcal{T}Q_{\neq}^{2}(U_{\neq}, U_{0}^{2})&=\int_{0}^{T}\int mM  \left\langle D \right\rangle^{N} \nabla_{L}Q_{\neq}^{2} mM  \left\langle D \right\rangle^{N} \nabla_{L}\left(U_{\neq}\cdot \nabla_{L}U_{0}^{2} \right)_{\neq}dVdt\nonumber\\
	&\leqslant \left\|mM\nabla_{L}Q_{\neq}^{2}\right\|_{L^{2}H^{N}}\left\|U_{\neq}\right\|_{L^{2}H^{N}}\left\|U_{0}^{2}\right\|_{L^{\infty}H^{N+2}}\nonumber\\
	&\quad+
	\left\|mM\nabla_{L}Q_{\neq}^{2}\right\|_{L^{2}H^{N}}\left\|\nabla_{L} U_{\neq}^{2, 3}\right\|_{L^{2}H^{N}}\left\|U_{0}^{2}\right\|_{L^{\infty}H^{N+1}}\nonumber \\
	&\lesssim
	\varepsilon\nu^{-\frac{1}{2}}   \varepsilon \nu^{-\frac{1}{6}} \varepsilon+\varepsilon\nu^{-\frac{1}{2}}   \varepsilon \nu^{-\frac{1}{2}} \varepsilon \nonumber \\
	&\lesssim  
	\varepsilon^{2}(\varepsilon\nu^{-1})\nonumber\\
	&\lesssim\varepsilon^{2},
\end{align}
where $\varepsilon\nu^{-1}$ is chosen small enough. Analogously to $\mathcal{T}Q_{\neq}^{2}(U_{\neq}, U_{0}^{2})$, using the fact $|\widehat{\nabla_{L}}|\leqslant |m\widehat{\Delta_{L}}|$, \eqref{4.11}, \eqref{4..8},  \eqref{4..22}--\eqref{4..24} and Corollary \ref{cor4.1}  yields 
\begin{align}
	\mathcal{T}Q_{\neq}^{2}(U_{\neq}, U_{\neq}^{2})&=\int_{0}^{T}\int mM  \left\langle D \right\rangle^{N} \nabla_{L}Q_{\neq}^{2} mM  \left\langle D \right\rangle^{N} \nabla_{L}\left(U_{\neq}\cdot \nabla_{L}U_{\neq}^{2} \right)_{\neq}dVdt\nonumber\\
	&\leqslant \left\|mM\nabla_{L}Q_{\neq}^{2}\right\|_{L^{2}H^{N}}\left\|\nabla_{L}U_{\neq}\right\|_{L^{2}H^{N}}\left\|mMQ_{\neq}^{2}\right\|_{L^{\infty}H^{N}}\nonumber\\
	&\quad+
	\left\|mM\nabla_{L}Q_{\neq}^{2}\right\|_{L^{2}H^{N}}\left\|U_{\neq}\right\|_{L^{\infty}H^{N}}\left\|mM Q_{\neq}^{2}\right\|_{L^{2}H^{N}}\sup m^{-1}\nonumber 
	\\&\lesssim
	\varepsilon\nu^{-\frac{1}{2}}  \varepsilon \nu^{-\frac{1}{2}}\varepsilon+	\varepsilon\nu^{-\frac{1}{2}}  \varepsilon \varepsilon \nu^{-\frac{1}{6}}\nu^{-\frac{1}{3}}\nonumber 
	\\&\lesssim  
	\varepsilon^{2}(\varepsilon\nu^{-1})\nonumber\\&\lesssim\varepsilon^{2}.
\end{align}

The next step to be analysed is the transport term $\mathcal{T}\check{K}_{\neq}^{2}1$. To do this, we first decompose $\mathcal{T}\check{K}_{\neq}^{2}1$ into
\begin{align}\label{6.26}
	\mathcal{T}\check{K}_{\neq}^{2}1&=\int_{0}^{T}\int mM  \left\langle D \right\rangle^{N} \check{K}_{\neq}^{2} mM  \left\langle D \right\rangle^{N} i\sqrt{\frac{\beta}{\beta-1}}\partial_{Z}|\nabla_{L}|\left(U\cdot\nabla_{L}U^{1}\right)_{\neq}dVdt\nonumber\\&=
	\int_{0}^{T}\int mM  \left\langle D \right\rangle^{N} \check{K}_{\neq}^{2}mM  \left\langle D \right\rangle^{N} i\sqrt{\frac{\beta}{\beta-1}}\partial_{Z}|\nabla_{L}|\Big(U_{0}\cdot\nabla_{L}U_{\neq}^{1}\nonumber\\&\quad+U_{\neq}\cdot\nabla U_{0}^{1}+U_{\neq}\cdot\nabla_{L}U_{\neq}^{1}\Big)_{\neq}dVdt\nonumber\\&\triangleq \mathcal{T}\check{K}_{\neq}^{2}1(U_{0},U_{\neq}^{1})+\mathcal{T}\check{K}_{\neq}^{2}1(U_{\neq},U_{0}^{1})+\mathcal{T}\check{K}_{\neq}^{2}1(U_{\neq},U_{\neq}^{1}).
\end{align}
For $\mathcal{T}\check{K}_{\neq}^{2}1(U_{0},U_{\neq}^{1})$, by utilizing \eqref{4..8}, \eqref{4..5}, \eqref{4..16}--\eqref{4..22}, \eqref{5.15} and Corollary  \ref{cor4.1},  there holds that
\begin{align}
	\mathcal{T}\check{K}_{\neq}^{2}1(U_{0},U_{\neq}^{1})&=\int_{0}^{T}\int  mM  \left\langle D \right\rangle^{N} \check{K}_{\neq}^{2}  mM  \left\langle D \right\rangle^{N} i\sqrt{\frac{\beta}{\beta-1}}\partial_{Z}|\nabla_{L}|(U_{0}\cdot\nabla_{L}U_{\neq}^{1})_{\neq}dVdt\nonumber \\
	&\leqslant
	\left\|mM\nabla_{L}\check{K}_{\neq}^{2}\right\|_{L^{2}H^{N}}\left\|U_{0}\right\|_{L^{\infty}H^{N+1}}\left\|\nabla_{L}U_{\neq}^{1}\right\|_{L^{2}H^{N}}\nonumber\\
	&\quad+\left\|mM\nabla_{L}\check{K}_{\neq}^{2}\right\|_{L^{2}H^{N}} \left\|U_{0}\right\|_{L^{\infty}H^{N}}\left(\left\|mM\nabla_{L}\check{K}_{\neq}^{2}\right\|_{L^{2}H^{N}}+\left\|mM\nabla_{L}Q_{\neq}^{2}\right\|_{L^{2}H^{N}} \right)\nonumber\\
	&\lesssim \varepsilon\nu^{-\frac{1}{2}}  \varepsilon \varepsilon \nu^{-\frac{1}{2}} +  \varepsilon\nu^{-\frac{1}{2}}  \varepsilon \varepsilon \nu^{-\frac{1}{2}}\nonumber\\
	&\lesssim  \varepsilon^{2}(\varepsilon\nu^{-1})\nonumber\\
	&\lesssim\varepsilon^{2},
\end{align}
where $\varepsilon\nu^{-1}$ is chosen small enough. 
Meantime, we use \eqref{4..8}, \eqref{4..5}, \eqref{4..16}, \eqref{4..23}, \eqref{4..24}, \eqref{5..102}, \eqref{5.16} and Corollary  \ref{cor4.1} to obtain
\begin{align}
	\mathcal{T}\check{K}_{\neq}^{2}1(U_{\neq},U_{0}^{1})&=\int_{0}^{T}\int  mM  \left\langle D \right\rangle^{N} \check{K}_{\neq}^{2}  mM  \left\langle D \right\rangle^{N} i\sqrt{\frac{\beta}{\beta-1}}\partial_{Z}|\nabla_{L}|(U_{\neq}\cdot\nabla U_{0}^{1})_{\neq}dVdt\nonumber\\
	&=-\int_{0}^{T}\int  mM  \left\langle D \right\rangle^{N}|\nabla_{L}| \check{K}_{\neq}^{2}  mM  \left\langle D \right\rangle^{N} i \sqrt{\frac{\beta}{\beta-1}}\partial_{Z} (U_{\neq}^{2}\partial_{Y}U_{0}^{1}+U_{\neq}^{3}\partial_{Z}U_{0}^{1})_{\neq}dVdt  \nonumber \\
	&\leqslant
	\left\|mM\nabla_{L}\check{K}_{\neq}^{2}\right\|_{L^{2}H^{N}}\left(\left\|mMQ_{\neq}^{2}\right\|_{L^{2}H^{N}}+\left\|mM\check{K}_{\neq}^{2}\right\|_{L^{2}H^{N}}\right) \left\|U_{0}^{1}\right\|_{L^{\infty}H^{N+1}}\nonumber\\
	&\quad+\left\|mM\nabla_{L}\check{K}_{\neq}^{2}\right\|_{L^{2}H^{N}} \left\|U_{\neq}^{2,3}\right\|_{L^{\infty}H^{N}}\left\|\nabla U_{0}^{1}\right\|_{L^{2}H^{N+1}}\nonumber\\
	&\lesssim \varepsilon\nu^{-\frac{1}{2}}  \varepsilon \nu^{-\frac{1}{6}}  \varepsilon +  \varepsilon\nu^{-\frac{1}{2}}  \varepsilon \varepsilon \nu^{-\frac{1}{2}}\nonumber
	\\&\lesssim  
	\varepsilon^{2}(\varepsilon\nu^{-1})\nonumber\\&
	\lesssim\varepsilon^{2}.
\end{align}
For $\mathcal{T}\check{K}_{\neq}^{2}1(U_{\neq},U_{\neq}^{1})$, combining \eqref{4..8}, \eqref{4..5}, \eqref{5.15}, \eqref{5..102}, \eqref{5.16} with \eqref{4..22}--\eqref{4..24} gives
\begin{align}
	\mathcal{T}\check{K}_{\neq}^{2}1(U_{\neq},U_{\neq}^{1})&=\int_{0}^{T}\int  mM  \left\langle D \right\rangle^{N} \check{K}_{\neq}^{2}  mM  \left\langle D \right\rangle^{N} i\sqrt{\frac{\beta}{\beta-1}}\partial_{Z}|\nabla_{L}|(U_{\neq}\cdot\nabla_{L}U_{\neq}^{1})_{\neq}dVdt\nonumber\\
	&\leqslant\left\|mM\nabla_{L}\check{K}_{\neq}^{2}\right\|_{L^{2}H^{N}}\left(\left\|mMQ_{\neq}^{2}\right\|_{L^{\infty}H^{N}}+\left\|mM\check{K}_{\neq}^{2}\right\|_{L^{\infty}H^{N}}\right) \left\|\nabla_{L} U_{\neq}^{1}\right\|_{L^{2}H^{N}}\nonumber\\
	&\quad+\left\|mM\nabla_{L}\check{K}_{\neq}^{2}\right\|_{L^{2}H^{N}} \left\|U_{\neq}\right\|_{L^{\infty}H^{N}} 
	\left(\left\|mM\nabla_{L}Q_{\neq}^{2}\right\|_{L^{2}H^{N}}+\left\|mM\nabla_{L}\check{K}_{\neq}^{2}\right\|_{L^{2}H^{N}} \right) \nonumber\\
	&\lesssim \varepsilon\nu^{-\frac{1}{2}}  \varepsilon   \varepsilon \nu^{-\frac{1}{2}} +  \varepsilon\nu^{-\frac{1}{2}}  \varepsilon \varepsilon \nu^{-\frac{1}{2}}\nonumber
	\\&\lesssim  
	\varepsilon^{2}(\varepsilon\nu^{-1})\nonumber\\&
	\lesssim\varepsilon^{2}.
\end{align}

Finally, similar to \eqref{6.26}, we divide the transport term $\mathcal{T}\check{K}_{\neq}^{2}2$ into
\begin{align}\label{6.30}
	\mathcal{T}\check{K}_{\neq}^{2}2&=-\int_{0}^{T}\int mM  \left\langle D \right\rangle^{N} \check{K}_{\neq}^{2} mM  \left\langle D \right\rangle^{N} i\sqrt{\frac{\beta}{\beta-1}}\partial_{X}|\nabla_{L}|\left(U\cdot\nabla_{L}U^{3}\right)_{\neq}dVdt\nonumber\\&=-
	\int_{0}^{T}\int mM  \left\langle D \right\rangle^{N} \check{K}_{\neq}^{2}mM  \left\langle D \right\rangle^{N} i\sqrt{\frac{\beta}{\beta-1}}\partial_{X}|\nabla_{L}|(U_{0}\cdot\nabla_{L}U_{\neq}^{3}\nonumber\\&\quad+U_{\neq}\cdot\nabla U_{0}^{3}+U_{\neq}\cdot\nabla_{L}U_{\neq}^{3})_{\neq}dVdt\nonumber\\&\triangleq \mathcal{T}\check{K}_{\neq}^{2}2(U_{0},U_{\neq}^{3})+\mathcal{T}\check{K}_{\neq}^{2}2(U_{\neq},U_{0}^{3})+\mathcal{T}\check{K}_{\neq}^{2}2(U_{\neq},U_{\neq}^{3}).
\end{align}
Combining \eqref{4..8}, \eqref{4..5}, \eqref{5.16}, the fact $\partial_{X} U_0=0$  with \eqref{4..16}--\eqref{4..17}, we bound $\mathcal{T}\check{K}_{\neq}^{2}2(U_{0},U_{\neq}^{3})$ as
\begin{align}
	\mathcal{T}\check{K}_{\neq}^{2}2(U_{0},U_{\neq}^{3})&=\int_{0}^{T}\int  mM  \left\langle D \right\rangle^{N} |\nabla _{L}| \check{K}_{\neq}^{2}  mM  \left\langle D \right\rangle^{N} i\sqrt{\frac{\beta}{\beta-1}}\partial_{X} (U_{0}\cdot\nabla_{L}U_{\neq}^{3})_{\neq}dVdt\nonumber \\
	&\leqslant \left\|mM\nabla_{L}\check{K}_{\neq}^{2}\right\|_{L^{2}H^{N}} \left\|U_{0}\right\|_{L^{\infty}H^{N}}\left(\left\|mM\nabla_{L}\check{K}_{\neq}^{2}\right\|_{L^{2}H^{N}}+\left\|mM\nabla_{L}Q_{\neq}^{2}\right\|_{L^{2}H^{N}} \right)\nonumber\\
	&\lesssim \varepsilon\nu^{-\frac{1}{2}}  \varepsilon \varepsilon \nu^{-\frac{1}{2}} \nonumber\\
	&\lesssim  
	\varepsilon^{2}(\varepsilon\nu^{-1})\nonumber\\&
	\lesssim\varepsilon^{2}.
\end{align}
To deal with $\mathcal{T}\check{K}_{\neq}^{2}2(U_{\neq},U_{0}^{3})$,  using \eqref{4..8}, \eqref{4..5}, \eqref{4..18}, \eqref{4..22}--\eqref{4..24}, Corollary  \ref{cor4.1},  \eqref{5..102}, \eqref{5.16}  and the fact $\partial_{X} U_0^3=0$, we have
\begin{align}
	\mathcal{T}\check{K}_{\neq}^{2}2(U_{\neq},U_{0}^{3})&=\int_{0}^{T}\int  mM  \left\langle D \right\rangle^{N}|\nabla_{L}| \check{K}_{\neq}^{2}  mM  \left\langle D \right\rangle^{N} i\sqrt{\frac{\beta}{\beta-1}}\partial_{X} (U_{\neq}\cdot\nabla U_{0}^{3})_{\neq}dVdt\nonumber\\
	&  \leqslant\left\|mM\nabla_{L}\check{K}_{\neq}^{2}\right\|_{L^{2}H^{N}} \left\|\partial_{X} U_{\neq}^{2, 3} \right\|_{L^{2}H^{N}}\left\|  U_{0}^{3}\right\|_{L^{\infty}H^{N+1}}\nonumber\\
	&  \leqslant\left\|mM\nabla_{L}\check{K}_{\neq}^{2}\right\|_{L^{2}H^{N}}  \left(\left\|mM \check{K}_{\neq}^{2}\right\|_{L^{2}H^{N}}+\left\|mM Q_{\neq}^{2}\right\|_{L^{2}H^{N}} \right) \left\|  U_{0}^{3}\right\|_{L^{\infty}H^{N+1}}\nonumber\\
	&\lesssim \varepsilon\nu^{-\frac{1}{2}}  \varepsilon \nu^{-\frac{1}{6}}  \varepsilon \nonumber
	\\&\lesssim  
	\varepsilon^{2}(\varepsilon\nu^{-\frac{2}{3}})\nonumber\\&
	\lesssim\varepsilon^{2}.
\end{align}
Similar to the estimate of  $\mathcal{T}\check{K}_{\neq}^{2}2(U_{\neq},U_{0}^{3})$, we further bound $\mathcal{T}\check{K}_{\neq}^{2}2(U_{\neq},U_{\neq}^{3})$ as
\begin{align}
	\mathcal{T}\check{K}_{\neq}^{2}2(U_{\neq},U_{\neq}^{3}) &=\int_{0}^{T}\int  mM  \left\langle D \right\rangle^{N}|\nabla_{L}| \check{K}_{\neq}^{2}  mM  \left\langle D \right\rangle^{N} i\sqrt{\frac{\beta}{\beta-1}}\partial_{X} (U_{\neq}\cdot\nabla_{L}U_{\neq}^{3})_{\neq} dVdt  \nonumber \\
	&\leqslant\left\|mM\nabla_{L}\check{K}_{\neq}^{2}\right\|_{L^{2}H^{N}}\left(\left\|\partial_{X}U_{\neq}^{1,3}\right\|_{L^{\infty}H^{N}}+\left\|\partial_{X}U_{\neq}^{2}\right\|_{L^{\infty}H^{N}}\right)\left\|\nabla_{L} U_{\neq}^{3}\right\|_{L^{2}H^{N}} \nonumber\\
	&\quad+ 
	\left\|mM\nabla_{L}\check{K}_{\neq}^{2}\right\|_{L^{2}H^{N}}\left\|U_{\neq}\right\|_{L^{\infty}H^{N}}\left\|\partial_{X}\nabla_{L}U_{\neq}^{3}\right\|_{L^{2}H^{N}} \nonumber\\
	&\leqslant
	\left\|mM\nabla_{L}\check{K}_{\neq}^{2}\right\|_{L^{2}H^{N}}\left( \left\|mM\check{K}_{\neq}^{2}\right\|_{L^{\infty}H^{N}}+\left\|mMQ_{\neq}^{2}\right\|_{L^{\infty}H^{N}} \right)\left\|\nabla_{L} U_{\neq}^{3}\right\|_{L^{2}H^{N}} \nonumber\\
	&\quad
	+\left\|mM\nabla_{L}\check{K}_{\neq}^{2}\right\|_{L^{2}H^{N}}\left\|U_{\neq}\right\|_{L^{\infty}H^{N}}\left(\left\|mM\nabla_{L}\check{K}_{\neq}^{2}\right\|_{L^{2}H^{N}}+\left\|mM\nabla_{L}Q_{\neq}^{2}\right\|_{L^{2}H^{N}}  \right)   \nonumber\\&
	\lesssim\varepsilon\nu^{-\frac{1}{2}}\varepsilon  \varepsilon \nu^{-\frac{1}{2}}    \nonumber
	\\&\lesssim  
	\varepsilon^{2}(\varepsilon\nu^{-1})\nonumber\\&
	\lesssim\varepsilon^{2},
\end{align}
where $\varepsilon\nu^{-1}$ is chosen small enough.  Submitting the estimates  \eqref{5.4}, \eqref{5.5}, \eqref{6.22}, \eqref{6.26}  and \eqref{6.30} into \eqref{5.2}, these two estimates \eqref{4..8} and \eqref{4..5} hold with $10$ replaced by $5$ on the right-hand side.
\subsection{Energy estimates on $Q_{0}^{ 2}, \check{K}_{0}^{2}$}\label{4.3.2}
Recall that $Q_{0}^{ 2}$ and $ \check{K}_{0}^{2}$ satisfy the equations as follows
\begin{align} 
	\begin{cases}
		\partial_{t} Q_{0}^{2}-\nu \Delta Q_{0}^{2}+i\sqrt{\beta(\beta-1)} \partial_{Z}|\nabla|^{-1} \check{K}_{0}^{2} = \partial_{Y}\left(\partial_{i}^{L}U^{j}\partial_{j}^{L}U^{i} \right)_{0}-\Delta(U\cdot\nabla_{L}U^{2})_{0},\nonumber\\
		\partial_{t}\check{K}_{0}^{2}-\nu\Delta\check{K}_{0}^{2}-i\sqrt{\beta(\beta-1)}\partial_{Z}|\nabla|^{-1}Q_{0}^{2}=i\sqrt{\frac{\beta}{\beta-1}}\partial_{Z}|\nabla|(U\cdot \nabla_{L}U^{1})_{0}. 
	\end{cases}
\end{align}
Naturally, the $H^{N}$ energy estimate gives
\begin{align}\label{6.34}
	&\frac{1}{2}\left(\left\|Q_{0}^{2}\right\|_{H^{N}}^{2}+\left\|\check{K}_{0}^{2}\right\|_{H^{N}}^{2}\right)+\nu \left\|\nabla Q_{0}^{2}\right\|_{L^{2}H^{N}}^{2}+
	\nu\left\|\nabla\check{K}_{0}^{2}\right\|_{L^{2}H^{N}}^{2}\nonumber\\&
	=\frac{1}{2}\left(\left\|Q_{0}^{2}(0)\right\|_{H^{N}}^{2}+\left\|\check{K}_{0}^{2}(0)\right\|_{H^{N}}^{2}\right)\nonumber\\&\quad+
	\int_{0}^{T}\int  \left\langle D \right\rangle^{N} Q_{0}^{2} \left\langle D \right\rangle^{N} \left[\partial_{Y}\left(\partial_{i}^{L}U^{j}\partial_{j}^{L}U^{i}\right)_{0}-\Delta\left(U\cdot \nabla_{L}U^{2}\right)_{0}\right]dVdt\nonumber\\&\quad+\int_{0}^{T} \int \left\langle D \right\rangle^{N} \check{K}_{0}^{2} \left\langle D \right\rangle^{N} i\sqrt{\frac{\beta}{\beta-1}}\partial_{Z}|\nabla|(U\cdot \nabla_{L}U^{1})_{0}dVdt\nonumber\\&\triangleq \frac{1}{2}\left(\left\|Q_{0}^{2}(0)\right\|_{H^{N}}^{2}+\left\|\check{K}_{0}^{2}(0)\right\|_{H^{N}}^2\right)+
	\mathcal{NLP}Q_{0}^{2}+\mathcal{T}Q_{0}^{2}+\mathcal{T}\check{K}_{0}^{2}.
\end{align}
We begin with estimating the nonlinear pressure term $\mathcal{NLP}Q_{0}^{2}$ in \eqref{6.34}. We first write it as
\begin{align}\label{6.35}
	\mathcal{NLP}Q_{0}^{2}&=\int_{0}^{T}\int   \left\langle D \right\rangle^{N}  Q_{0}^{2}  \left\langle D \right\rangle^{N}\partial_{Y}
	\left(\partial_{i}^{L}U^{j}\partial_{j}^{L}U^{i} \right)_{0}dVdt\nonumber\\&=
	\int_{0}^{T}\int  \left\langle D \right\rangle^{N} Q_{0}^{2}  \left\langle D \right\rangle^{N} \partial_{Y}\left[ \partial_{i}U_{0}^{j}\partial_{j}U_{0}^{i}+\left(\partial_{i}^{L}U_{\neq}^{j}\partial_{j}^{L}U_{\neq}^{i}\right)_{0}\right]dVdt\nonumber\\&\triangleq \mathcal{NLP}Q_{0}^{2}(0,0)+ \mathcal{NLP}Q_{0}^{2}(\neq,\neq).
\end{align}
On the one hand, by using \eqref{4..4}, \eqref{4..18} and \eqref{4..17}, we have
\begin{align}
	\mathcal{NLP}Q_{0}^{2}(0,0)&=-\int_{0}^{T}\int \left\langle D \right\rangle^{N} \partial_{Y} Q_{0}^{2}\left\langle D \right\rangle^{N} \left(\partial_{i}U_{0}^{j}\partial_{j}U_{0}^{i}\right)\cdot 1_{i, j\neq 1}dVdt\nonumber\\&
	\leqslant \left\|\nabla Q_{0}^{2}\right\|_{L^{2}H^{N}}\left\|U_{0}^{2,3}\right\|_{L^{\infty}H^{N+1}}\left\|\nabla U_{0}^{2,3}\right\|_{L^{2}H^{N}}\nonumber\\&  
	\lesssim\varepsilon\nu^{-\frac{1}{2}}\varepsilon\varepsilon \nu^{-\frac{1}{2}}\nonumber
	\\&\lesssim  
	\varepsilon^{2}(\varepsilon\nu^{-1})\nonumber\\&
	\lesssim\varepsilon^{2},
\end{align}
where $\varepsilon\nu^{-1}$ is sufficiently small.  On the other hand, by \eqref{4.11}, \eqref{4..8}--\eqref{4..4}, \eqref{4..23}, \eqref{4..24}  and Corollary \ref{cor4.1}, it yields 
\begin{align}
	\mathcal{NLP}Q_{0}^{2}(\neq,\neq)&=\int_{0}^{T}\int \left\langle D \right\rangle^{N}  Q_{0}^{2}\left\langle D \right\rangle^{N}\partial_{Y} \left(\partial_{i}^{L}U_{\neq}^{j}\partial_{j}^{L}U_{\neq}^{i}\right)_0 \cdot 1_{i,j\neq 1}dVdt\nonumber\\&
	\leqslant \left\|Q_{0}^{2}\right\|_{L^{\infty}H^{N}}\left\|\nabla_{L}U_{\neq}^{2}\right\|_{L^{2}H^{N}}\left\|mMQ_{\neq}^{2}\right\|_{L^{2}H^{N}}\sup m^{-1}      \nonumber\\&\quad+\left\|\nabla Q_{0}^{2}\right\|_{L^{2}H^{N}}\left(\left\|mM\check{K}_{\neq}^{2}\right\|_{L^{2}H^{N}}+\left\|mMQ_{\neq}^{2}\right\|_{L^{2}H^{N}} \right)\sup m^{-1}\left\|mM	Q_{\neq}^{2}\right\|_{L^{\infty}H^{N}}\nonumber\\&\quad+
	\left\|Q_{0}^{2}\right\|_{L^{\infty}H^{N}}\left(\left\|mMQ_{\neq}^{2}\right\|_{L^{2}H^{N}}+\left\|mM\check{K}_{\neq}^{2}\right\|_{L^{2}H^{N}} \right)\sup m^{-1}\left\|\nabla_{L}U_{\neq}^{3}\right\|_{L^{2}H^{N}}\nonumber
	\\& \lesssim\varepsilon \varepsilon\nu^{-\frac{1}{2}}\varepsilon\nu^{-\frac{1}{6}} \nu^{-\frac{1}{3}}+ \varepsilon\nu^{-\frac{1}{2}}\varepsilon\nu^{-\frac{1}{6}}\nu^{-\frac{1}{3}} \varepsilon+\varepsilon\varepsilon \nu^{-\frac{1}{6}}\nu^{-\frac{1}{3}} \varepsilon\nu^{-\frac{1}{2}} \nonumber
	\\&\lesssim  
	\varepsilon^{2}(\varepsilon\nu^{-1})\nonumber\\&
	\lesssim\varepsilon^{2}.
\end{align}

We now turn our attention to the transport term $\mathcal{T}Q_{0}^{2}$. Notice that
\begin{align}\label{6.38}
	\mathcal{T}Q_{0}^{2}&=-\int_{0}^{T}\int \left\langle D \right\rangle^{N} Q_{0}^{2}\left\langle D \right\rangle^{N}\Delta\left(U\cdot \nabla_{L}U^{2}\right)_{0}dVdt\nonumber\\&=
	-\int_{0}^{T}\int  \left\langle D \right\rangle^{N} Q_{0}^{2} \left\langle D \right\rangle^{N} \Delta\left[\left(U_{0}\cdot\nabla U_{0}^{2}\right)+\left(U_{\neq}\cdot\nabla_{L}U_{\neq}^{2}\right)_{0}\right]dVdt\nonumber\\&\triangleq \mathcal{T}Q_{0}^{2}(0,0)+\mathcal{T}Q_{0}^{2}(\neq,\neq).
\end{align}
It follows from \eqref{4..4}, \eqref{4..18} and \eqref{4..17}   that
\begin{align}
	\mathcal{T}Q_{0}^{2}(0,0)&=\int_{0}^{T}\int \left\langle D \right\rangle^{N} \nabla Q_{0}^{2}\left\langle D \right\rangle^{N}\nabla \left(U_{0}\cdot\nabla U_{0}^{2}\right)dVdt\nonumber\\&\leqslant \left\|\nabla Q_{0}^{2}\right\|_{L^{2}H^{N}}\left\|U_{0}^{2, 3}\right\|_{L^{\infty}H^{N+1}}\left\|\nabla U_{0}^{2}\right\|_{L^{2}H^{N+1}}\nonumber\\&\lesssim 
	\varepsilon\nu^{-\frac{1}{2}}\varepsilon \varepsilon\nu^{-\frac{1}{2}} \nonumber\\&\lesssim\varepsilon^{2}(\varepsilon\nu^{-1})\nonumber\\&\lesssim\varepsilon^{2}.
\end{align}
The estimate for $\mathcal{T}Q_{0}^{2}(\neq,\neq)$ is more refined. Utilizing \eqref{4.11}, \eqref{4..8}--\eqref{4..4}, \eqref{4..22}--\eqref{4..24} and Corollary \ref{cor4.1} yields
\begin{align}
	\mathcal{T}Q_{0}^{2}(\neq,\neq)&=-\int_{0}^{T}\int \left\langle D \right\rangle^{N} Q_{0}^{2}\left\langle D \right\rangle^{N} \Delta( U_{\neq}\cdot\nabla_{L} U_{\neq}^{2})_{0}dVdt\nonumber\\&=-\int_{0}^{T}\int \left\langle D \right\rangle^{N} Q_{0}^{2} \left\langle D \right\rangle^{N} \left( \Delta_{L}U_{\neq}\cdot \nabla_{L}U_{\neq}^{2}+2\nabla_{L}U_{\neq}\cdot Q_{\neq}^{2}+U_{\neq}\cdot \nabla_{L}Q_{\neq}^{2} \right)_{0}dVdt\nonumber\\&\leqslant -\int_{0}^{T}\int \left\langle D \right\rangle^{N} Q_{0}^{2} \left\langle D \right\rangle^{N}\left( \Delta_{L}U_{\neq}\cdot\nabla_{L}U_{\neq}^{2}\right)dVdt\nonumber\\&\quad+\left\|Q_{0}^{2}\right\|_{L^{\infty}H^{N}}\left(\left\|mM\check{K}_{\neq}^{2}\right\|_{L^{2}H^{N}}+\left\|mMQ_{\neq}^{2}\right\|_{L^{2}H^{N}}\right)\sup m^{-2} \left\|mMQ_{\neq}^{2}\right\|_{L^{2}H^{N}} \nonumber\\&\quad  +\left\|Q_{0}^{2}\right\|_{L^{\infty}H^{N}}\left\|U_{\neq}\right\|_{L^{2}H^{N}}\left\|mM\nabla_{L}Q_{\neq}^{2}\right\|_{L^{2}H^{N}}\sup m^{-1}\nonumber\\&  \lesssim\varepsilon^{2}(\varepsilon\nu^{-1})+\varepsilon(\varepsilon\nu^{-\frac{1}{6}}\nu^{-\frac{1}{3}})^{2}+\varepsilon\varepsilon\nu^{-\frac{1}{6}}\varepsilon \nu^{-\frac{1}{2}}\nu^{-\frac{1}{3}}\nonumber\\&\lesssim \varepsilon^{2}(\varepsilon\nu^{-1})\nonumber\\&\lesssim
	\varepsilon^{2},
\end{align}
where  $\varepsilon\nu^{-1}$ is sufficiently small and we have used the following estimate
\begin{align}
	&\int_{0}^{T}\int \left\langle D \right\rangle^{N} Q_{0}^{2} \left\langle D \right\rangle^{N} \left(\Delta_{L}U_{\neq}\cdot\nabla_{L}U_{\neq}^{2}\right)dVdt\nonumber\\&\quad=\int_{0}^{T}\int  \left\langle D \right\rangle^{N} Q_{0}^{2}  \left\langle D \right\rangle^{N} \left[(\partial_{XX}+\partial_{ZZ})U_{\neq}^{i} \partial_{i}^{L} U_{\neq}^{2}\right]dVdt\nonumber\\&\qquad+\int_{0}^{T}\int  \left\langle D \right\rangle^{N} Q_{0}^{2} \left\langle D \right\rangle^{N}\left(\partial_{YY}^{L}U_{\neq}^i \partial_{i}^{L} U_{\neq}^{2}\right)dVdt\nonumber\\&\quad\leqslant
	\left\|Q_{0}^{2}\right\|_{L^{\infty}H^{N}}\left\|m MQ_{\neq}^{2}\right\|_{L^{2}H^{N}}\left\|mMQ_{\neq}^{2}\right\|_{L^{2}H^{N}}\nonumber\\&\qquad+\left\|  Q_{0}^{2}\right\|_{L^{\infty}H^{N}}\left(\left\|mM Q_{\neq}^{2}\right\|_{L^{2}H^{N}}+\left\|mM \check{K}_{\neq}^{2}\right\|_{L^{2}H^{N}}\right)  \left\|m M Q_{\neq}^{2}\right\|_{L^{2}H^{N}}\nonumber\\&\qquad+ \left\|\nabla Q_{0}^{2}\right\|_{L^{2}H^{N}}\left\|mM Q_{\neq}^{2}\right\|_{L^{2}H^{N}}\left\|m M Q_{\neq}^{2}\right\|_{L^{\infty}H^{N}}\nonumber\\&\qquad+\left\|\nabla Q_{0}^{2}\right\|_{L^{2}H^{N}}\left(\left\|mM Q_{\neq}^{2}\right\|_{L^{2}H^{N}}+\left\|mM \check{K}_{\neq}^{2}\right\|_{L^{2}H^{N}}\right)\sup m^{-1}\left\|m M Q_{\neq}^{2}\right\|_{L^{\infty}H^{N}}\nonumber\\&\qquad+\left\|Q_{0}^{2}\right\|_{L^{\infty}H^{N}}\left(\left\|m M Q_{\neq}^{2}\right\|_{L^{2}H^{N}}+\left\|mM \check{K}_{\neq}^{2}\right\|_{L^{2}H^{N}}\right)\sup m^{-1}\left\|m M Q_{\neq}^{2}\right\|_{L^{2}H^{N}}\nonumber\\&\quad\lesssim
	\varepsilon(\varepsilon\nu^{-\frac{1}{6}})^{2}+\varepsilon(\varepsilon\nu^{-\frac{1}{6}})^{2}+\varepsilon\nu^{-\frac{1}{2}}  \varepsilon\nu^{-\frac{1}{6}} \varepsilon +
	\varepsilon\nu^{-\frac{1}{2}}  \varepsilon\nu^{-\frac{1}{6}}\nu^{-\frac{1}{3}} \varepsilon+\varepsilon \varepsilon \nu^{-\frac{1}{6}} \nu^{-\frac{1}{3}} \varepsilon\nu^{-\frac{1}{6}}\nonumber\\&\quad\lesssim
	\varepsilon^{2}(\varepsilon\nu^{-1})\nonumber\\&\quad\lesssim
	\varepsilon^{2}.\nonumber
\end{align}

For the last term $\mathcal{T}\check{K}_{0}^{2}$ in \eqref{6.34}, we decompose it into 
\begin{align}\label{6.41}
	\mathcal{T}\check{K}_{0}^{2}&=\int_{0}^{T}\int \left\langle D \right\rangle^{N} \check{K}_{0}^{2} \left\langle D \right\rangle^{N} i\sqrt{\frac{\beta}{\beta-1}}\partial_{Z}|\nabla|\left(U\cdot\nabla_{L}U^{1}\right)_{0}dVdt\nonumber\\& =
	\int_{0}^{T}\int \left\langle D \right\rangle^{N} \check{K}_{0}^{2}\left\langle D \right\rangle^{N} i\sqrt{\frac{\beta}{\beta-1}}\partial_{Z}|\nabla|\left[ \left(U_{0}\cdot \nabla U_{0}^{1}\right)+\left(U_{\neq}\cdot \nabla_{L}U_{\neq}^{1}\right)_{0}\right]dVdt\nonumber\\&\triangleq
	\mathcal{T}\check{K}_{0}^{2}(0,0)+\mathcal{T}\check{K}_{0}^{2}(\neq,\neq).
\end{align}
Using \eqref{4..7}, \eqref{4..16}--\eqref{4..17} and the fact $\partial_{X} U_0^1=0$ leads to 
\begin{align}
	\mathcal{T}\check{K}_{0}^{2}(0,0)&=\int_{0}^{T}\int 
	\left\langle D \right\rangle^{N} \check{K}_{0}^{2} \left\langle D \right\rangle^{N} i\sqrt{\frac{\beta}{\beta-1}}\partial_{Z}|\nabla|\left(U_{0}\cdot\nabla U_{0}^{1}\right)dVdt\nonumber\\&=\int_{0}^{T}\int \left\langle D \right\rangle^{N} \check{K}_{0}^{2} \left\langle D \right\rangle^{N} i\sqrt{\frac{\beta}{\beta-1}}\partial_{Z}|\nabla_{y,z}|\left(U_{0}^{2}\partial_{Y}U_{0}^{1}+U_{0}^{3}\partial_{Z}U_{0}^{1} \right)dVdt\nonumber\\&\lesssim
	\left\|\nabla\check{K}_{0}^{2}\right\|_{L^{2}H^{N}}\left\|U_{0}^{2,3}\right\|_{L^{\infty}H^{N+1}}\left\|\nabla U_{0}^{1}\right\|_{L^{2}H^{N+1}}\nonumber\\&\lesssim 
	\varepsilon \nu^{-\frac{1}{2}} \varepsilon \varepsilon \nu^{-\frac{1}{2}} \nonumber \\&
	\lesssim  \varepsilon^2 \left( \varepsilon \nu^{-1}\right)\nonumber\\& \lesssim  \varepsilon^2,
\end{align}
where  $\varepsilon\nu^{-1}$ is sufficiently small. Making using of 
\eqref{4..8}, \eqref{4..5}, \eqref{4..7}, \eqref{4..22}--\eqref{4..24}, the relationship  between $U_{\neq}$ and $Q_{\neq}^2$, $\check{K}_{\neq}^2$ (\eqref{5..102}, \eqref{5.16}),  and the definition of $m$, we  conclude that
\begin{align}
	\mathcal{T}\check{K}_{0}^{2}(\neq,\neq)&=\int_{0}^{T}\int 
	\left\langle D \right\rangle^{N} \check{K}_{0}^{2} \left\langle D \right\rangle^{N} i\sqrt{\frac{\beta}{\beta-1}}\partial_{Z}|\nabla|\left(U_{\neq}\cdot\nabla_{L} U_{\neq}^{1}\right)_{0}dVdt\nonumber\\&=
	\int_{0}^{T}\int \left\langle D \right\rangle^{N} \check{K}_{0}^{2} \left\langle D \right\rangle^{N} i\sqrt{\frac{\beta}{\beta-1}}\partial_{Z}|\nabla_{y,z}|\left( U_{\neq}\cdot\nabla_{L}U_{\neq}^{1}\right)_{0}dVdt\nonumber
	\\&=-\int_{0}^{T}\int \left\langle D \right\rangle^{N} |\nabla_{y,z}|\check{K}_{0}^{2} \left\langle D \right\rangle^{N} i\sqrt{\frac{\beta}{\beta-1}}\left(\partial_{Z}U_{\neq}\cdot \nabla_{L}U_{\neq}^{1}+U_{\neq}\cdot \nabla_{L}\partial_{Z}U_{\neq}^{1}\right)_{0}dVdt\nonumber\\& 
	\lesssim \left\|\nabla \check{K}_{0}^{2}\right\|_{L^{2}H^{N}}\left(\left\|mM\check{K}_{\neq}^{2}\right\|_{L^{\infty}H^{N}}+\left\|mMQ_{\neq}^{2}\right\|_{L^{\infty}H^{N}}\right)\left\|\nabla_{L}U_{\neq}^{1}\right\|_{L^{2}H^{N}}\nonumber\\&\quad+\left\|\nabla\check{K}_{0}^{2}\right\|_{L^{2}H^{N}}\left\|U_{\neq}\right\|_{L^{2}H^{N}}\left(\left\|mM\check{K}_{\neq}^{2}\right\|_{L^{\infty}H^{N}}+\left\|mMQ_{\neq}^{2}\right\|_{L^{\infty}H^{N}}\right)\sup m^{-1}\nonumber\\&\lesssim 
	\varepsilon \nu^{-\frac{1}{2}} \varepsilon \varepsilon \nu^{-\frac{1}{2}}+ \varepsilon \nu^{-\frac{1}{2}}  \varepsilon\nu^{-\frac{1}{6}} \varepsilon \nu^{-\frac{1}{3}} \nonumber \\&
	\lesssim  \varepsilon^2 \left( \varepsilon \nu^{-1}\right)\nonumber\\& \lesssim  \varepsilon^2.
\end{align}
Putting the estimates \eqref{6.35}, \eqref{6.38} and \eqref{6.41} into \eqref{6.34}, we finish the proof of \eqref{4..4} and \eqref{4..7}, with $10$ replaced by $5$.

\section{Energy estimates on zero frequency velocity $U$}\label{sec7}
The purpose of this section is to derive low-frequency controls on the velocity. Specifically, we aim to prove that under the bootstrap assumption of Proposition \ref{pro4.1}, the estimate of $U_0^{i}$ $(i=1,2,3)$ is valid (i.e.,  \eqref{4..11}--\eqref{4..10}), with 10 replaced by 5 on the right-hand side. By the definition of new variable as in \eqref{2.5}, we can move from one coordinate system to another.  In particular, we have $\left\| U_0^i \right\|_{H^N}= \left\| u_0^i \right\|_{H^N}$. Therefore, it suffices to prove these estimates on $u_{0}^{i}$, rather than $U_{0}^{i}$. Moreover, in order to estimate the zero frequency of the velocity field, we need to deal with  the simple zero frequency and double zero frequency, respectively. 

Due to incompressible condition, we have $\overline{u}_{0}^{2}\equiv 0$. Thus, we here care about the handling of $\overline{u}_{0}^{1}$ and $\overline{u}_{0}^{3}$, which satisfy the following equations
\begin{align}\label{7.1} 
	\begin{cases}
		\partial_{t}\overline{u}_{0}^{1}-\nu\partial_{yy}\overline{u}_{0}^{1}=-\overline{\left(u\cdot\nabla u^{1}\right)_{0}}=-\partial_{y}\overline{\left(u^{2}u^{1}\right)_{0}}=-\partial_{y}\overline{\left(\widetilde{u}_{0}^{2}\widetilde{u}_{0}^{1}\right)}-\partial_{y}\overline{\left(u_{\neq}^{2}u_{\neq}^{1}\right)_{0}},  \\
		\partial_{t}\overline{u}_{0}^{3}-\nu\partial_{yy}\overline{u}_{0}^{3}=-\partial_{y}\overline{\left(u^{2}u^{3} \right)_{0}}=-\partial_{y}\overline{\left(\widetilde{u}_{0}^{2}\widetilde{u}_{0}^{3}\right)}-\partial_{y}\overline{\left(u_{\neq}^{2}u_{\neq}^{3}\right)_{0}}.
	\end{cases}
\end{align}

First of all, we consider the equation \eqref{7.1}$_{1}$ which gives the $H^{N+1}$ estimate for $\overline{u}_{0}^{1}$ that
\begin{align}\label{7.2}
	\frac{1}{2}\left\|\overline{u}_{0}^{1}\right\|_{H^{N+1}}^{2}+\nu\left\|\partial_{y}\overline{u}_{0}^{1}\right\|_{L^{2}H^{N+1}}^{2}&=\frac{1}{2}\left\|\overline{u}_{0\mathrm{in}}^{1}\right\|_{H^{N+1}}^{2}-\int_{0}^{T}\int  \left\langle D \right\rangle^{N+1} \overline{u}_{0}^{1}  \left\langle D \right\rangle^{N+1}\partial_{y}\overline{\left(\widetilde{u}_{0}^{2}\widetilde{u}_{0}^{1} \right)}dVdt\nonumber\\&\quad-
	\int_{0}^{T}\int  \left\langle D \right\rangle^{N+1}\overline{u}_{0}^{1}  \left\langle D \right\rangle^{N+1} \partial_{y}\overline{\left(u_{\neq}^{2}u_{\neq}^{1} \right)_{0}}dVdt\nonumber\\&\triangleq \frac{1}{2}\left\|\overline{u}_{0\mathrm{in}}^{1}\right\|_{H^{N+1}}^{2}+\mathcal{T}(\widetilde{u}_{0}^{2},\widetilde{u}_{0}^{1}) +\mathcal{T}\left(u_{\neq}^{2},u_{\neq}^{1}\right).
\end{align}
Using \eqref{4..9}, \eqref{4..16}, \eqref{4..17} and the fact $\left\|\widetilde{u}_{0}^{2}\right\|_{L^{2}H^{N+1}} \leqslant \left\|\partial_{z} \widetilde{u}_{0}^{2}\right\|_{L^{2}H^{N+1}}$, we have
\begin{align}\label{7.3}
	\mathcal{T}\left(\widetilde{u}_{0}^{2},\widetilde{u}_{0}^{1}\right)&=\int_{0}^{T}\int \left\langle D \right\rangle^{N+1}\partial_{y}\overline{u}_{0}^{1}\left\langle D \right\rangle^{N+1}\overline{\left(\widetilde{u}_{0}^{2}\widetilde{u}_{0}^{1} \right)}dVdt\nonumber\\&\leqslant 
	\left\|\partial_{y}\overline{u}_{0}^{1}\right\|_{L^{2}H^{N+1}}\left\|\nabla \widetilde{u}_{0}^{2}\right\|_{L^{2}H^{N+1}}\left\|\widetilde{u}_{0}^{1}\right\|_{L^{\infty}H^{N+1}}\nonumber\\&\lesssim
	\varepsilon \nu^{-\frac{1}{2}}\varepsilon \nu^{-\frac{1}{2}} \varepsilon\nonumber\\&  
	\lesssim  \varepsilon^2 \left( \varepsilon \nu^{-1}\right)\nonumber\\& \lesssim  \varepsilon^2,
\end{align}
where $\varepsilon\nu^{-1}$ is sufficiently small. Meanwhile, from \eqref{4..8}, \eqref{4..9}, \eqref{4..22}, \eqref{4..23}, \eqref{5..102} and Corollary \ref{cor4.1}, it implies that
\begin{align}\label{7.4}
	\mathcal{T}\left(u_{\neq}^{2}, u_{\neq}^{1}\right)&=\int_{0}^{T}\int \left\langle D \right\rangle^{N+1}\partial_{y}\overline{u}_{0}^{1}\left\langle D \right\rangle^{N+1}\overline{\left(u_{\neq}^{2} u_{\neq}^{1} \right)_{0}}dVdt\nonumber\\& 
	\leqslant\left\|\partial_{y}\overline{u}_{0}^{1}\right\|_{L^{2}H^{N+1}}\left\|u_{\neq}^{1}u_{\neq}^{2}\right\|_{L^{2}H^{N+1}}\nonumber\\&\leqslant
	\left\|\partial_{y}\overline{u}_{0}^{1}\right\|_{L^{2}H^{N+1}} \left\|u_{\neq}^{2}\right\|_{L^{\infty}H^{N+1}}\left\|u_{\neq}^{1}\right\|_{L^{2}H^{N+1}}\nonumber\\&\leqslant\left\|\partial_{y}\overline{u}_{0}^{1}\right\|_{L^{2}H^{N+1}}\left(\left\|U_{\neq}^{2}\right\|_{L^{\infty}H^{N}}+\left\|\nabla_{L} U_{\neq}^{2}\right\|_{L^{\infty}H^{N}}\right)\nonumber\\&\quad\times\left(\left\|U_{\neq}^{1}\right\|_{L^{2}H^{N}}+\left\|\nabla_{L} U_{\neq}^{1}\right\|_{L^{2}H^{N}}\right)\nonumber
	\\&\lesssim\varepsilon \nu^{-\frac{1}{2}}\varepsilon  (\varepsilon \nu^{-\frac{1}{6}}+\varepsilon \nu^{-\frac{1}{2}})  \nonumber\\&\lesssim  \varepsilon^2 \left( \varepsilon \nu^{-1}\right)\nonumber\\& \lesssim  \varepsilon^2.
\end{align}

Next, by \eqref{7.1}$_{2}$, $H^{N+1}$ estimate for $\overline{u}_{0}^{3}$ gives
\begin{align}\label{7.5}
	&\frac{1}{2}\left\|\overline{u}_{0}^{3}\right\|_{H^{N+1}}^{2}+\nu\left\|\partial_{y}\overline{u}_{0}^{3}\right\|_{L^{2}H^{N+1}}^{2}\nonumber\\&\quad=\frac{1}{2}\left\|\overline{u}_{0\mathrm{in}}^{3}\right\|_{H^{N+1}}^{2}-\int_{0}^{T}\int  \left\langle D \right\rangle^{N+1} \overline{u}_{0}^{3}  \left\langle D \right\rangle^{N+1}\partial_{y}\overline{\left(\widetilde{u}_{0}^{2}\widetilde{u}_{0}^{3} \right)}dVdt\nonumber\\&\qquad-
	\int_{0}^{T}\int  \left\langle D \right\rangle^{N+1}\overline{u}_{0}^{3}  \left\langle D \right\rangle^{N+1} \partial_{y}\overline{\left(u_{\neq}^{2}u_{\neq}^{3} \right)_{0}}dVdt\nonumber\\&\quad\triangleq \frac{1}{2}\left\|\overline{u}_{0\mathrm{in}}^{3}\right\|_{H^{N+1}}^{2}+\mathcal{T}(\widetilde{u}_{0}^{2},\widetilde{u}_{0}^{3}) +\mathcal{T}\left(u_{\neq}^{2},u_{\neq}^{3}\right).
\end{align}
For $\mathcal{T}\left(\widetilde{u}_{0}^{2},\widetilde{u}_{0}^{3}\right)$,  using \eqref{4..10}, \eqref{4..18}, \eqref{4..17} and $\left\|\widetilde{u}_{0}^{2}\right\|_{L^{2}H^{N+1}} \leqslant \left\|\partial_{z} \widetilde{u}_{0}^{2}\right\|_{L^{2}H^{N+1}}$, we obtain
\begin{align}\label{7.6}
	\mathcal{T}\left(\widetilde{u}_{0}^{2},\widetilde{u}_{0}^{3}\right)&=\int_{0}^{T}\int \left\langle D \right\rangle^{N+1}\partial_{y}\overline{u}_{0}^{3}\left\langle D \right\rangle^{N+1}\overline{\left(\widetilde{u}_{0}^{2}\widetilde{u}_{0}^{3} \right)}dVdt\nonumber\\&\leqslant 
	\left\|\partial_{y}\overline{u}_{0}^{3}\right\|_{L^{2}H^{N+1}}\left\|\nabla \widetilde{u}_{0}^{2}\right\|_{L^{2}H^{N+1}}\left\|\widetilde{u}_{0}^{3}\right\|_{L^{\infty}H^{N+1}}\nonumber\\&\lesssim
	\varepsilon \nu^{-\frac{1}{2}}\varepsilon \nu^{-\frac{1}{2}} \varepsilon\nonumber\\&  
	\lesssim  \varepsilon^2 \left( \varepsilon \nu^{-1}\right)\nonumber\\& \lesssim  \varepsilon^2.
\end{align}
Analogously to $\mathcal{T}\left(u_{\neq}^{2}, u_{\neq}^{1}\right)$, it follows from \eqref{4..8}, \eqref{4..10}, \eqref{4..23}, \eqref{4..24} and Corollary \ref{cor4.1} that
\begin{align}\label{7.7}
	\mathcal{T}\left(u_{\neq}^{2}, u_{\neq}^{3}\right)&=\int_{0}^{T}\int \left\langle D \right\rangle^{N+1}\partial_{y}\overline{u}_{0}^{3}\left\langle D \right\rangle^{N+1}\overline{\left(u_{\neq}^{2}u_{\neq}^{3} \right)_{0}}dVdt\nonumber\\& 
	\leqslant\left\|\partial_{y}\overline{u}_{0}^{3}\right\|_{L^{2}H^{N+1}}\left(\left\|U_{\neq}^{2}\right\|_{L^{\infty}H^{N}}+\left\| mM Q_{\neq}^{2}\right\|_{L^{\infty}H^{N}}\right)\nonumber\\&\quad\times\left(\left\|U_{\neq}^{3}\right\|_{L^{2}H^{N}}+\left\|\nabla_{L} U_{\neq}^{3}\right\|_{L^{2}H^{N}}\right)\nonumber
	\\&\lesssim\varepsilon \nu^{-\frac{1}{2}}\varepsilon  (\varepsilon \nu^{-\frac{1}{6}}+\varepsilon \nu^{-\frac{1}{2}})  \nonumber\\&\lesssim  \varepsilon^2 \left( \varepsilon \nu^{-1}\right)\nonumber\\& \lesssim  \varepsilon^2,
\end{align}
where $\varepsilon\nu^{-1}$ is sufficiently small.  Therefore, collecting \eqref{7.2}--\eqref{7.7}, we finish the proof of \eqref{4..9} and \eqref{4..10} in Proposition \ref{pro4.1}, with 10 replaced by 5. 

Now we need to estimate  the simple zero-frequency of the velocity. Unlike double zero frequency velocity, $\widetilde{u}_{0}^{j}$ ($j=1, 2$) satisfy the following equations
\begin{align*}
	\begin{cases}
		\partial_{t}\widetilde{u}_{0}^{1}-\nu\Delta_{y,z}\widetilde{u}_{0}^{1}+(1-\beta)\widetilde{u}_{0}^{2}=-\widetilde{\left(u\cdot\nabla u^{1} \right)_{0}}, \\
		\partial_{t}\widetilde{u}_{0}^{2}-\nu\Delta_{y,z}\widetilde{u}_{0}^{2}+\beta \partial_{zz} \Delta_{y, z}^{-1}  \widetilde{u}_{0}^{1} =-\widetilde{\left(u\cdot\nabla u^{2} \right)_{0}}+\partial_{y} \Delta_{y, z}^{-1}\widetilde{ \left(\partial_{i}u_j \partial_{j}u_i \right)_{0}},
	\end{cases}
\end{align*}
which contains the coupling term $(1-\beta)\widetilde{u}_{0}^{2}$ and $\beta \partial_{zz} \Delta_{y, z}^{-1}  \widetilde{u}_{0}^{1}$. This undermines the validity of attempting to directly close the prior assumption \eqref{4..11} and \eqref{4..12} by using the $\widetilde{u}_{0}^{j}$ equations alone. To overcome it, we first apply Duhamel's principle and \eqref{3.45}$_{1}$ to deduce that
\begin{align}\label{7.9}
	\widehat{\widetilde{u}_{0}^{1}}(t)&= e^{-\nu(\eta^{2}+l^{2})t}\widehat{\widetilde{u}_{0\mathrm{in}}^{1}}+(\beta-1)\int_{0}^{t}e^{-\nu(\eta^{2}+l^{2})(t-\tau)}\widehat{\widetilde{u}_{0}^{2}}(\tau)d\tau\nonumber\\&\quad-\int_{0}^{t}e^{-\nu(\eta^{2}+l^{2})(t-\tau)}\widehat{\widetilde{\left(u\cdot \nabla u^{1} \right)_{0}}}(\tau)d\tau\nonumber\\&
	=e^{-\nu(\eta^{2}+l^{2})t}\cos(ht)\widehat{\widetilde{u}_{0\mathrm{in}}^{1}}+\frac{\beta-1}{h}e^{-\nu(\eta^{2}+l^{2})t}\sin(ht)\widehat{\widetilde{u}_{0\mathrm{in}}^{2}}\nonumber\\&\quad -
	\int_{0}^{t}e^{-\nu(\eta^{2}+l^{2})(t-\tau)}\widehat{\widetilde{\left(u\cdot \nabla u^{1} \right)_{0}}}(\tau)d\tau,
\end{align}
where $h\triangleq \sqrt{\beta(\beta-1)}\frac{|l|}{|\eta,l|}.$
By \eqref{5.15}$_{1}$ and Plancherel's theorem, we have 
\begin{equation*}
	\left\| \nabla_{L} U_{\neq}^1  \right\|_{H^N} \leqslant \left( \left\| mM \check{K}_{\neq}^2  \right\|_{H^N}+ \left\| mM Q_{\neq}^2  \right\|_{H^N}\right)\sup m^{-1}.
\end{equation*}
Further, by definition of $h$, one have 
\begin{equation*}
	\frac{\beta-1}{h}=\sqrt{\frac{\beta-1}{\beta}} \frac{|\eta, l|}{|l|} \lesssim 1+|\eta|.
\end{equation*}
Hence, directly calculate gives
\begin{align}\label{7.10}
	\left\|\widetilde{u}_{0}^{1}\right\|_{H^{N}}^{2}=\left\|\widehat{\widetilde{u}_{0}^{1}}\right\|_{H^{N}}^{2}&\leqslant\left\|e^{-\nu(\eta^{2}+l^{2})t}\cos(ht)\widehat{\widetilde{u}_{0\mathrm{in}}^{1}}\right\|_{H^{N}}^{2}+\left\|\frac{\beta-1}{h}e^{-\nu(\eta^{2}+l^{2})t}\sin(ht)\widehat{\widetilde{u}_{0\mathrm{in}}^{2}}\right\|_{H^{N}}^{2}\nonumber\\&\quad+\left\|\int_{0}^{t}e^{-\nu(\eta^{2}+l^{2})(t-\tau)}\widehat{\widetilde{\left(u\cdot\nabla u^{1} \right)_{0}}}(\tau)d\tau\right\|_{H^{N}}^2\nonumber\\&\leqslant\left\|\widetilde{u}_{0\mathrm{in}}^{1}\right\|_{H^{N}}^{2}+\left\|\widetilde{u}_{0\mathrm{in}}^{2}\right\|_{H^{N+1}}^{2}+\int_{0}^{t}\left\|e^{-\nu(\eta^{2}+l^{2})(t-\tau)}\widehat{\widetilde{\left(u\cdot\nabla u^{1}\right)_{0}}}(\tau) \right\|_{H^{N}}^{2}d\tau\nonumber\\&\leqslant
	\left\|\widetilde{u}_{0\mathrm{in}}^{1}\right\|_{H^{N}}^{2}+\left\|\widetilde{u}_{0\mathrm{in}}^{2}\right\|_{H^{N+1}}^{2}+\int_{0}^{t}e^{-\nu(t-\tau)}\left\|
	\left(u\cdot\nabla u^{1} \right)_{0}\right\|_{H^{N}}^{2}d\tau\nonumber\\&\leqslant \left\|\widetilde{u}_{0\mathrm{in}}^{1}\right\|_{H^{N}}^{2}+\left\|\widetilde{u}_{0\mathrm{in}}^{2}\right\|_{H^{N+1}}^{2} \nonumber \\
	&\quad+\int_{0}^{t}e^{-\nu(t-\tau)}\left(\left\|
	\left(u_{0}\cdot\nabla u_{0}^{1} \right)\right\|_{H^{N}}^{2}+\left\|\left(u_{\neq}\cdot\nabla_{L}u_{\neq}^{1}\right)_{0}\right\|_{H^{N}}^{2}\right)d\tau\nonumber\\&\leqslant \left\|\widetilde{u}_{0\mathrm{in}}^{1}\right\|_{H^{N}}^{2}+\left\|\widetilde{u}_{0\mathrm{in}}^{2}\right\|_{H^{N+1}}^{2}+\int_{0}^{t}e^{-\nu(t-\tau)}\left\|u_{0}^{2,3}\right\|_{H^{N}}^{2}\left\|\nabla u_{0}^{1}\right\|_{H^{N}}^{2}d\tau\nonumber\\&\quad+
	\int_{0}^{t}e^{-\nu(t-\tau)}\left\|U_{\neq}\right\|_{H^{N}}^{2}\left\|\nabla_{L}U_{\neq}^{1}\right\|_{H^{N}}^{2}d\tau\nonumber\\&\leqslant
	\left\|\widetilde{u}_{0\mathrm{in}}^{1}\right\|_{H^{N}}^{2}+ \left\|\widetilde{u}_{0\mathrm{in}}^{ 2}\right\|_{H^{N+1}}^{2}+\left\|u_{0}^{2,3}\right\|_{L^{\infty}H^{N}}^{2}\left\|\nabla u_{0}^{1}\right\|_{L^{\infty}H^{N}}^{2}\int_{0}^{t}e^{-\nu(t-\tau)}d\tau\nonumber\\&\quad+\left\|U_{\neq}\right\|_{L^{\infty}H^{N}}^{2}\left\|\nabla_{L}U_{\neq}^{1}\right\|_{L^{\infty}H^{N}}^{2}\int_{0}^{t}e^{-\nu(t-\tau)}d\tau \nonumber \\&\lesssim  \varepsilon^2+\varepsilon^2 \varepsilon^2\nu^{-1}+\varepsilon^2\varepsilon^2\nu^{-\frac{2}{3}}\nu^{-1}\nonumber\\&  \lesssim  \varepsilon^2+  \varepsilon^2( \varepsilon^2 \nu^{-1} +\varepsilon^2 \nu^{-\frac{5}{3}})\nonumber\\& \lesssim \varepsilon^2,
\end{align}
where we have used \eqref{4.11}, \eqref{4..8}, \eqref{4..5}, \eqref{4..16}--\eqref{4..24}, and the fact $\partial_{x} u_0^1=0$.
 
As the handing of  \eqref{7.10}, we further conclude that
\begin{align}\label{7.11}
	\left\|\nabla \widetilde{u}_{0}^{1}\right\|_{L^{2}H^{N}}^2&=\int_{0}^{t}\left\|\nabla \widetilde{u}_{0}^{1}\right\|_{H^{N}}^{2}d\tau \nonumber\\&\leqslant
	\left\|\nabla \left(e^{-\nu(\eta^{2}+l^{2})t}\cos(ht)\widehat{\widetilde{u}_{0\mathrm{in}}^{1}} \right)\right\|_{L^{2}H^{N}}^{2}+\left\|\nabla\left(\frac{\beta-1}{h}e^{-\nu(\eta^{2}+l^{2})t}\sin(ht)\widehat{\widetilde{u}_{0\mathrm{in}}^{2}} \right)\right\|_{L^{2}H^{N}}^{2}\nonumber\\&\quad+
	\left\|\int_{0}^{t}\nabla \left(e^{-\nu(\eta^{2}+l^{2})(t-\tau)}\widehat{\widetilde{\left(u\cdot \nabla u^{1}\right)_{0}}}(\tau) \right)d\tau\right\|_{L^{2}H^{N}}^{2}\nonumber\\&\leqslant
	\left\|\nabla \widetilde{u}_{0\mathrm{in}}^{1}\right\|_{L^{2}H^{N}}^{2}+\left\|\nabla\widetilde{u}_{0\mathrm{in}}^{2}\right\|_{L^{2}H^{N+1}}^{2}+\left\|\int_{0}^{t}\nabla\left(e^{-\nu(\eta^{2}+l^{2})(t-\tau)}\widehat{\widetilde{\left(u\cdot\nabla u^{1}\right)_{0}}}(\tau)\right)d\tau \right\|_{L^{2}H^{N}}^{2}\nonumber\\&\triangleq
	\left\|\nabla \widetilde{u}_{0\mathrm{in}}^{1}\right\|_{L^{2}H^{N}}^{2}+\left\|\nabla \widetilde{u}_{0\mathrm{in}}^{2}\right\|_{L^{2}H^{N+1}}^{2}+\mathcal{T}(u,\nabla u^{1}).
\end{align}
For the last term $\mathcal{T}(u,\nabla u^{1})$, using \eqref{4..16}--\eqref{4..24}, Young's inequality (for the convolution of two functions) and Lemma \ref{lem2.3} yields
\begin{align}\label{7.12}
	\mathcal{T}(u,\nabla u^{1})&= \left\|\int_{0}^{t}\nabla\left(e^{-\nu(\eta^{2}+l^{2})(t-\tau)}\widehat{\widetilde{\left(u\cdot\nabla u^{1}\right)_{0}}}(\tau)\right)d\tau \right\|_{L^{2}H^{N}}^{2}\nonumber\\&\leqslant
	\left\|\int_{0}^{t}\left\|\nabla\left(e^{-\nu(\eta^{2}+l^{2})(t-\tau)} \widehat{\widetilde{\left(u\cdot\nabla u^{1}\right)_{0}}} (\tau)\right)\right\|_{H^{N}}d\tau \right\|_{L^{2}(0,T)}^{2}\nonumber\\&\lesssim
	\left\|\int_{0}^{t}\left(1+(t-\tau)^{-\frac{1}{2}\times 2}\right)\left\|\left(u\cdot\nabla u^{1}\right)_{0}(\tau)\right\|_{H^{N-1}}d\tau\right\|_{L^{2}(0,T)}^{2}\nonumber\\&\lesssim \left\|\left(u\cdot \nabla u^{1} \right)_0\right\|_{L^{1}H^{N-1}}^{2}\left(\left\|t^{-1}\right\|_{L^{2}}^{2}+1 \right)\nonumber\\&\lesssim \left\|u_{0}^{2}\partial_{y}u_{0}^{1}+u_{0}^{3}\partial_{z}u_{0}^{1}+\left(U_{\neq}\cdot \nabla_{L}U_{\neq}^{1}\right)_{0} \right\|_{L^{1}H^{N-1}}^{2}\nonumber\\&\lesssim
	\left\|\nabla u_{0}^{2,3}\right\|_{L^{2}H^{N-1}}^{2}\left\|\nabla u_{0}^{1}\right\|_{L^{2}H^{N-1}}^{2}+\left\|U_{\neq}\right\|_{L^{2}H^{N-1}}^2\left\|\nabla_{L}U_{\neq}^{1}\right\|_{L^{2}H^{N-1}}^{2}\nonumber\\&\lesssim
	\varepsilon^2\nu^{-1}	\varepsilon^2\nu^{-1} +  \varepsilon^2  \nu^{-\frac{1}{3}}  \varepsilon^2  \nu^{-1}   \nonumber\\& \lesssim \varepsilon^2\nu^{-1}(\varepsilon^2\nu^{-1}+\varepsilon^2\nu^{-\frac{1}{3}})\nonumber\\&\lesssim \varepsilon^2\nu^{-1}.
\end{align}
Choosing $\varepsilon \nu^{-1}$ small enough, together \eqref{7.10}, \eqref{7.11} and \eqref{7.12} gives
\begin{align}\label{7.13}
	\left\|\widetilde{u}_{0}^1\right\|_{L^{\infty}H^{N}}+\nu^{\frac{1}{2}}\left\|\nabla\widetilde{u}_{0}^{1}\right\|_{L^{2}H^{N}}\lesssim \varepsilon.
\end{align}

In addition, by the relationship $\Big|\widehat{\widetilde{u}_{0}^{2}}\Big|=\Big|\widehat{\Delta_{y,z}^{-1} Q_{0}^{2}}\Big|\leqslant \Big|\widehat{Q_{0}^{2}}\Big|,$ we also obtain
\begin{align}\label{7.14}
	\left\|\widetilde{u}_{0}^{2}\right\|_{L^{\infty}H^{N}}+\nu^{\frac{1}{2}}\left\|\nabla\widetilde{u}_{0}^{2}\right\|_{L^{2}H^{N}}\leqslant 
	\left\|Q_{0}^{2}\right\|_{L^{\infty}H^{N}}+\nu^{\frac{1}{2}}\left\|\nabla Q_{0}^{2}\right\|_{L^{2}H^{N}}\lesssim  \varepsilon.
\end{align}
Since $\nabla\cdot u_{0}=0$, we have $\partial_{y}u_{0}^{2}+\partial_{z}u_{0}^{3}=0.$
When $l\neq 0,$ one gets
\begin{align*}
	\widetilde{u}_{0}^{3}=-\partial_{z}^{-1}\partial_{y}\widetilde{u}_{0}^{2}=-\partial_{z}^{-1}\partial_{y}\Delta_{y,z}^{-1}Q_{0}^{2}.
\end{align*}
Hence, we deduce that
\begin{align*}
	\Big|\widehat{\widetilde{u}_{0}^{3}}\Big|\leqslant \frac{|\eta|}{|l|\left(\eta^{2}+l^{2}\right)}  \Big|\widehat{Q_{0}^{2}}\Big|\leqslant \Big|\widehat{Q_{0}^{2}}\Big|, \,  \text{ for } |l| \geqslant 1
\end{align*}
and 
\begin{align}\label{7.15}
	\left\|\widetilde{u}_{0}^{3}\right\|_{L^{\infty}H^{N}}+\nu^{\frac{1}{2}}\left\|\nabla\widetilde{u}_{0}^{3}\right\|_{L^{2}H^{N}}\leqslant 
	\left\|Q_{0}^{2}\right\|_{L^{\infty}H^{N}}+\nu^{\frac{1}{2}}
	\left\|\nabla Q_{0}^{2}\right\|_{L^{2}H^{N}}\lesssim\varepsilon.
\end{align}
Thus, \eqref{7.13}, \eqref{7.14} and \eqref{7.15} imply \eqref{4..11}--\eqref{4..13} hold with 10 replaced by 5.
\\[2mm]

\vspace{2mm} 
\noindent\textbf{Data availability statement.}
\vskip2mm

No new data were created or analysed in this study.

\vspace{4mm} \noindent\textbf{Acknowledgements.}
Xu was partially supported by the National Key R\&D Program of China (grant 2020YFA0712900) and the National Natural Science Foundation of China (grants 12171040, 11771045 and 12071069).

%

\addcontentsline{toc}{section}{References}
\begin{thebibliography}{99}	
    \bibitem{MR1400515}
    A. Babin, A. Mahalov, B. Nicolaenko, Global splitting, integrability and regularity of 3D Euler and Navier-Stokes equations for uniformly rotating fluids. European J. Mech. B Fluids. 15(3): 291--300 (1996).	
    
    \bibitem{MR1480996}
    A. Babin, A. Mahalov, B. Nicolaenko, Regularity and integrability of 3D Euler and Navier-Stokes equations for rotating fluids. Asymptot. Anal. 15(2):  103--150 (1997).
    
    \bibitem{MR1736966}
    A. Babin, A. Mahalov, B. Nicolaenko, Global regularity of 3D rotating Navier-Stokes equations for resonant domains. Indiana Univ. Math. J. 48(3): 1133--1176 (1999).
    
    \bibitem{MR1855663}
    A. Babin, A. Mahalov, B. Nicolaenko, 3D Navier-Stokes and Euler equations with initial data characterized by uniformly large vorticity. Indiana Univ. Math. J. 50: 1-35 (2001).

%
    \bibitem{MR3612004}
    J. Bedrossian, P. Germain, N. Masmoudi, On the stability threshold for the 3D Couette flow in
    Sobolev regularity. Ann. of Math.  185(2): 541--608 (2017).
    
    \bibitem{MR3974608}
    J. Bedrossian, P. Germain, N. Masmoudi, Stability of the Couette flow at high Reynolds numbers in two dimensions and three dimensions. Bull. Amer. Math. Soc.  56(3):  373--414 (2019).
    
    \bibitem{MR4126259}
    J. Bedrossian, P. Germain, N. Masmoudi, Dynamics near the subcritical transition of the 3D Couette flow I: Below threshold case. Mem. Amer. Math. Soc. 266(1294), v+158pp (2020).
    
    \bibitem{BGM2022}
    J. Bedrossian, P. Germain, N. Masmoudi, Dynamics near the subcritical transition of the 3D Couette flow II: Above threshold case. Mem. Amer. Math. Soc. 279(1377) v+135pp (2022).
    
    \bibitem{BHIW1-2024}
    J. Bedrossian, S. He, S. Iyer, F. Wang, Uniform Inviscid Damping and Inviscid Limit of the 2D Navier-Stokes equation with Navier Boundary Conditions. arXiv:2405.19249  (2024).
    
    \bibitem{BHIW2-2024}
    J. Bedrossian, S. He, S. Iyer, F. Wang, Stability threshold of nearly-Couette shear flows with Navier boundary conditions in 2D. arXiv:2311.00141  (2024).
    
    \bibitem{MR3448924}
    J. Bedrossian, N. Masmoudi, V. Vicol, Enhanced dissipation and inviscid damping in the inviscid limit of the Navier-Stokes equations near the two dimensional Couette flow. Arch. Ration. Mech. Anal. 219(3): 1087--1159  (2016).
    
   
    \bibitem{MR3867637}
    J. Bedrossian, V. Vicol, F. Wang, The Sobolev stability threshold for 2D shear flows near Couette. J. Nonlinear Sci. 28(6): 2051--2075 (2018).
    
    \bibitem{BL1976}
    J. Bergh and J. L$\ddot{o}$fstr$\ddot{o}$m, Interpolation space, an introduction. Springer, Berlin, 1976.
    
    \bibitem{MR1886008}
    S. J. Chapman, Subcritical transition in channel flows. J. Fluid Mech. 451: 35--97 (2002).
    
    \bibitem{MR2228849}
    J. Y. Chemin, B. Desjardins, I. Gallagher, E. Grenier, Mathematical geophysics. The Clarendon Press, Oxford University Press, Oxford (2006).
    
    \bibitem{CDLZ2024}
    Q. Chen, S. Ding, Z. Lin, Z. Zhang, Nonlinear stability for 3-D plane Poiseuille flow in a finite channel. arXiv:2310.11694 (2024).
    
    \bibitem{MR4121130}
    Q. Chen, T. Li, D. Wei, Z. Zhang, Transition threshold for the 2-D Couette flow in a finite channel. Arch. Ration. Mech. Anal. 238(1): 125--183 (2020).
    
    \bibitem{CWZ2023}
    Q. Chen, D. Wei, Z. Zhang, Linear inviscid damping and enhanced dissipation for monotone shear flows. Comm. Math. Phys. 400(1): 215?276 (2023).
    
    \bibitem{CWZ2024}
    Q. Chen, D. Wei, Z. Zhang, Transition threshold for the 3D Couette flow in a finite channel. Mem. Amer. Math. Soc. 296(1478), v+178pp (2024).
    
    \bibitem{CWW2024}
    S. Cui, L. Wang, W. Wang, Suppression of blow-up in Patlak-Keller-Segel system coupled with linearized Navier-Stokes equations via the 3D Couette flow. arXiv:2405.10337 (2024).

    \bibitem{DWXZ2021}
    B. Dong, J. Wu, X. Xu, N. Zhu, Stability and exponential decay for the 2D anisotropic Boussinesq equations with horizontal dissipation. Calc. Var. Partial Differential Equations. 60(3): 1--21 (2021).
    
    \bibitem{MR2736445}
    Y. Duguet, L. Brandt, B. R. J. Larsson, Towards minimal perturbations in transitional plane Couette flow. Phys. Rev. E.  82(3): 1-13 (2010).
    
    \bibitem{Ellingsen1975}
    T. Ellingsen, E. Palm, Stability of linear flow. Phys. Fluids. 18: 487--488 (1975).
    
    \bibitem{GSR2007}
    I. Gallagher, L. Saint-Raymond, On the influence of the Earth's rotation on geophysical flows. Handbook of mathematical fluid dynamics. Elsevier/North-Holland, Amsterdam. 2007.
 

   \bibitem{GHPW2023}
    Y. Guo, C. Huang, B. Pausader, K. Widmayer, On the stabilizing effect of rotation in the 3d Euler equations. Comm. Pure Appl. Math. 76(12): 3553--3641 (2023).
    
    \bibitem{GPW2023}
    Y. Guo, B. Pausader, K. Widmayer, Global axisymmetric Euler flows with rotation. Invent. Math. 231(1): 169--262 (2023).
    
    
    \bibitem{MR3096523}
    T. Iwabuchi, R. Takada, Global solutions for the Navier-Stokes equations in the rotational framework. Math. Ann. 357(2):  727--741 (2013).
    
    \bibitem{MR3229792}
    T. Iwabuchi, R. Takada, Global well-posedness and ill-posedness for the Navier-Stokes equations with the Coriolis force in function spaces of Besov type. J. Funct. Anal. 267(5):   1321--1337 (2014)
    \bibitem{MR3468733}
    
    T. Iwabuchi, R. Takada, Dispersive effect of the Coriolis force and the local well-posedness for the Navier-Stokes equations in the rotational framework. Funkcial. Ekvac. 58(3):  365--385 (2015).
    
    \bibitem{JW2020}
    H. Jia, R. Wan, Long time existence of classical solutions for the rotating Euler equations and related models in the optimal Sobolev
    space. Nonlinearity. 33(8): 3763--3780 (2020).
    
    \bibitem{Kelvin1887}
    L. Kelvin, Stability of fluid motion: rectilinear motion of viscous fluid between two parallel plates. Phil. Mag. 24:  188--196 (1887).
    
    \bibitem{MR3229600}
    Y. Koh, S. Lee, R. Takada, Dispersive estimates for the Navier-Stokes equations in the rotational framework. Adv. Differential Equations. 19(9-10):  857--878  (2014).
    
    \bibitem{KLT2014}
    Y. Koh, S. Lee, R. Takada, Strichartz estimates for the Euler equations in the rotational framework. J. Differential Equations. 256(2) 707--744 (2014).
    

    \bibitem{KOT2003}
    H. Kozono, T. Ogawa, Y. Taniuchi, Navier-Stokes equations in the Besov space near $L^\infty$ and BMO. Kyushu J. Math. 57(2): 303--324 (2003).
 
    \bibitem{LO97}
    C. D. Levermore, M. Oliver, Analyticity of solutions for a generalized Euler equation. J. Differential Equations. 133(2): 321--339 (1997).
    
    \bibitem{MR4115008}
    K. Liss, On the Sobolev stability threshold of 3D Couette flow in a uniform magnetic field. Comm. Math. Phys. 377(2):  859--908 (2020).
    
    \bibitem{Lundbladh1994}
    A. Lundbladh, D. S. Henningson, S. C. Reddy, Threshold amplitudes for transition in channel flows. Transition, Turbulence and Combustion: Volume I Transition.  309--318  (1994).
    
    \bibitem{MB02}
    A. J. Majda, A. L. Bertozzi, Vorticity and incompressible flow. Cambridge Texts in Applied Mathematics. Cambridge University Press, Cambridge, (2002).
    
     \bibitem{MR4451473}
    N. Masmoudi, S. H. Belkacem, W. Zhao, Stability of the Couette flow for a 2D Boussinesq system without thermal diffusivity. Arch. Ration. Mech. Anal. 245(2):  645--752 (2022).
    
    \bibitem{MR4176913}
    N. Masmoudi, W. Zhao, Enhanced dissipation for the 2D Couette flow in critical space. Comm. Partial Differential Equations. 45(12): 1682--1701(2020).
    
    \bibitem{MR1631950}
    S. C. Reddy, P. J. Schmid, J. S. Baggett, D. S. Henningson, On stability of streamwise streaks and transition thresholds
    in plane channel flows. J. Fluid Mech. 365: 269--303  (1998).

    \bibitem{MR1801992}
    P. J. Schmid, D. S. Henningson, Stability and transition in shear flows.
    Applied Mathematical Sciences. Springer-Verlag, New York, 2001.   

	\bibitem{SM1993}
	E. M. Stein, T. S. Murphy, Harmonic analysis: real-variable methods, orthogonality, and oscillatory integrals. Princeton University Press, Princeton, NJ, (1993).
	
	\bibitem{T2016}
	R. Takada, Long time existence of classical solutions for the 3D
	incompressible rotating Euler equations. J. Math. Soc. Japan. 68(2): 579--608 (2016).
	
	\bibitem{TTRD1993}
	L. Trefethen, A. Trefethen, S. Reddy, T. Driscoll, Hydrodynamic stability without eigenvalues. Science. 261 (5121): 578-584 (1993).
	
	\bibitem{MR4373161}
	D. Wei, Z. Zhang, Transition threshold for the 3D Couette flow in Sobolev space. Comm. Pure Appl. Math. 74(11): 2398--2479  (2021).
	
	\bibitem{MR3185102}
	A. M. Yaglom, Hydrodynamic instability and transition to turbulence. Fluid Mechanics and its Applications. Springer, Dordrecht, (2012).
	
	\bibitem{Zelati2023}
	M. C. Zelati, A. D. Zotto, Suppression of lift-up effect in the 3D Boussinesq equations around a stably stratified Couette flow. Quarterly of Applied Mathematics, (2024).
	
	\bibitem{ZZW2024}
	M. C. Zelati, A. D. Zotto, K. Widmayer, Stability of viscous three-dimensional stratified Couette flow via dispersion and mixing. arXiv:2402.15312 (2024).
	

	
\end{thebibliography}

\end{document}